\documentclass[reqno,11pt]{amsart}
\usepackage{a4wide}
\usepackage{mathbbol}
\usepackage{mathrsfs}
\usepackage{amssymb,mathrsfs,graphicx,extpfeil}
\usepackage{epsfig}
\usepackage{indentfirst, latexsym, amssymb, enumerate,amsmath,graphicx}
\usepackage{float}
\usepackage{colortbl}
\usepackage{epsfig,subfigure}
\pdfoutput=1%For ARXIV
\allowdisplaybreaks[4]
\usepackage{amsmath}
\title[Hilbert Expansion for Kinetic Equations]{Hilbert expansion for kinetic equations with non-relativistic Coulomb collision}

\author[Y. J. Lei]{Yuanjie Lei}
\address[Y. J. Lei]{{\newline School of Mathematics and Statistics, Huazhong University of Science and Technology, Wuhan 430074, China}}
\email{leiyuanjie@hust.edu.cn}

\author[S. Q. Liu]{Shuangqian Liu}
\address[S. Q. Liu]{{\newline School of Mathematics and Statistics and and Hubei Key Laboratory of Mathematical Sciences, Central China Normal University, Wuhan 430079, China}}
\email{sqliu@ccnu.edu.cn}

\author[Q. H. Xiao]{Qinghua Xiao}
\address[Q. H. Xiao]{\newline Innovation Academy for Precision Measurement Science and Technology, Chinese Academy of Sciences, Wuhan 430071, China}
\email{xiaoqh@apm.ac.cn}

\author[H. J. Zhao]{Huijiang Zhao}
\address[H. J. Zhao]{{\newline School of Mathematics and Statistics, Wuhan University, Wuhan 430072, China; Computational
Science Hubei Key Laboratory, Wuhan University, Wuhan 430072, China}}
\email{hhjjzhao@whu.edu.cn}

\newtheorem{theorem}{Theorem}[section]
\newtheorem{lemma}{Lemma}[section]
\newtheorem{corollary}{Corollary}[section]
\newtheorem{proposition}{Proposition}[section]

\newtheorem{remark}{Remark}[section]

\def\charf {\mbox{{\text 1}\kern-.30em {\text l}}}
\def\nablax{\nabla_x}

\newcommand{\R}{\mathbb{R}}

\newcommand{\Z}{\mathbb{Z}}

\newcommand{\FE}{\mathbf{E}}

\newcommand{\FM}{\mathbf{M}}

\newcommand{\FI}{\mathbf{I}}

\newcommand{\CA}{\mathcal{A}}

\newcommand{\CC}{\mathcal{C}}

\newcommand{\CK}{\mathcal{K}}
\newcommand{\CL}{\mathcal{L}}

\newcommand{\CP}{\mathcal{P}}
\newcommand{\CQ}{\mathcal{Q}}

\newcommand{\CU}{\mathcal{U}}

\newcommand{\na}{\nabla}

\newcommand{\al}{\alpha}
\newcommand{\bet}{\beta}
\newcommand{\ga}{\gamma}

\newcommand{\de}{\delta}
\newcommand{\si}{\sigma}
\newcommand{\pa}{\partial}

\newcommand{\eps}{\epsilon}

\newcommand{\ta}{\theta}

\newcommand{\vps}{\varepsilon}

\newcommand{\Ga}{\Gamma}

\newcommand{\lag}{\langle}
\newcommand{\rag}{\rangle}

\numberwithin{equation}{section}

\begin{document}
%%%%%%%%%%%%%%%%

\date{\today}


\thanks{\textbf{Acknowledgment.} The research of Yuanjie Lei was supported by the National Natural Science Foundation of China grants 11871335 and 11971187, the research of Shuangqian Liu was supported by the National Natural Science Foundation of China grant 11971201, the research of Qinghua Xiao was supported by the National Natural Science Foundation of China grants 11871469 and 12271506, the research of Huijiang Zhao was supported by the National Natural Science Foundation of China grants 11731008 and 12221001. This work was also partially supported by the Fundamental Research Funds for the Central Universities and by the Science and Technology Department of Hubei Province grant 2020DFH002.}

% \thispagestyle{style}

%--------------------------------------------------------------------------------------
%abstract------------------------------------------------------------------------------
\begin{abstract} In this paper, we study the hydrodynamic limits of both the Landau equation and the Vlasov-Maxwell-Landau system in the whole space. Our main purpose is two-fold: the first one is to give a rigorous derivation of the compressible Euler equations from the Landau equation via the Hilbert expansion; while the second one is to prove, still in the setting of Hilbert expansion, that the unique classical solution of the Vlasov-Maxwell-Landau system converges, which is shown to be globally in time, to the resulting global smooth solution of the Euler-Maxwell system, as the Knudsen number goes to zero. The main ingredient of our analysis is to derive some novel interplay energy estimates on the solutions of the Landau equation and the Vlasov-Maxwell-Landau system which are small perturbations of both a local Maxwellian and a global Maxwellian, respectively. Our result solves an open problem in the hydrodynamic limit for the Landau-type equations with Coulomb potential and the approach developed in this paper can seamlessly be used to deal with the problem on the validity of the Hilbert expansion for other types of kinetic equations.

\end{abstract}

\maketitle
\thispagestyle{empty}

\tableofcontents

%%%%%%%%%%%%%%%%%%%%%%%%%%%%%%%%%%%%%%%%%%%%%
\section{Introduction}
%%%%%%%%%%%%%%%%%%%%%%%%%%%%%%%%%%%%%%%%%%%%%

\setcounter{equation}{0}
The continuous transitions from the kinetic equations for perfect gases or particles to hydrodynamics is related to the sixth problem of Hilbert. As the first example of the program proposed by himself, Hilbert introduced the famous Hilbert expansion in the Knudsen number $\vps$ cf. \cite{Hilbert, Saint-Raymond-2009}. The rigorous justification of the validity of such an expansion for kinetic equations is a challenging problem since then.

In this article, at the kinetic level, we consider the following rescaled Landau equation
\begin{align} \label{LE}
\begin{aligned}
& \partial_t F^{\varepsilon} + v \cdot \nabla_x F^{\varepsilon} = \frac{1}{\varepsilon}\mathcal{C}(F^{\varepsilon},F^{\varepsilon})
 \end{aligned}
\end{align}
and the rescaled Vlasov-Maxwell-Landau (VML for short in the sequel) system
\begin{align}\label{main1}
\begin{aligned}
& \partial_t F^{\varepsilon} + v\cdot \nabla_x F^{\varepsilon}- e_-\Big(E^{\epsilon}+v\times B^{\varepsilon} \Big)\cdot\nabla_v F^{\varepsilon} = \frac{1}{\varepsilon}\mathcal{C}(F^{\varepsilon},F^{\varepsilon}),\\
 & \partial_tE^{\varepsilon}-  c\nabla_x \times B^{\varepsilon} =4\pi \int_{\mathbb R^3} v F^{\varepsilon} dv, \\
 &\partial_tB^{\varepsilon}+ c\nabla_x \times E^{\varepsilon}=0,\\
& \nabla_x\cdot E^{\varepsilon}=4\pi e_-\Big(1 -\int_{\mathbb R^3}  F^{\varepsilon} dv\Big), \qquad \nabla_x\cdot B^{\varepsilon}=0.
\end{aligned}
\end{align}
Here, $\varepsilon>0$ is the Knudsen number, $\nabla_x=(\partial_{x_1}, \partial_{x_2}, \partial_{x_3})$, $\nabla_v=(\partial_{v_1}, \partial_{v_2}, \partial_{v_3})$. The unknown function  $F^{\varepsilon}= F^{\varepsilon}(t, x, v)\geq0$ is the number density function for electrons at time $t\geq0$, position $x=(x_1, x_2, x_3) \in \mathbb R^3$ and velocity $v=(v_1, v_2, v_3)\in \mathbb R^3$. $-e_-$ is the charge of the electrons, $c$ is the speed of light, and $[E^\varepsilon, B^\varepsilon]=[E^\varepsilon(t,x), B^\varepsilon(t,x)]$ stands for the electromagnetic field.  Without loss of generality and for notational simplicity, we normalize the constants $e_-$ and $c$ to be one in the rest of this paper. Moreover, the Landau collision operator $\CC(\cdot,\cdot)$ is given by
\begin{eqnarray*}%\label{L-op}
\CC(G,H)&=&\nabla_v\cdot\left\{\int_{{\R}^{3}}\phi(v-v')\left[G(v')\nabla_vH(v)-H(v)\nabla_{v'}G(v')\right]dv'\right\}\nonumber\\
&=&\sum\limits_{i,j=1}^{3}\partial_{v_i}\int_{{\R}^{3}}\phi^{ij}(v-v')\left[G(v')\partial_{v_j}H(v)-H(v)\partial_{v'_j}G(v')\right]dv'.
\end{eqnarray*}
$\phi(v):=\left(\phi^{ij}(v)\right)_{3\times 3}$ is a non-negative matrix with $\phi^{ij}(v)$ being given by
\begin{equation}\label{cker}
\phi^{ij}(v)=\left\{\delta_{ij}-\frac{v_iv_j}{|v|^2}\right\}|v|^{\ga+2}.
\end{equation}
Here, $\de_{ij}$ is the Kronecker delta and $\gamma$ is a parameter leading to the standard classification of the hard potential $(\gamma>0)$, Maxwellian molecule $(\gamma=0)$ or soft potential $(\gamma<0)$. In this paper, we focus on the case of $\gamma=-3$, which correspondes to the Coulomb interaction in plasma physics.

The purpose of this paper is to rigorously justify that the Hilbert expansion is valid for both the Landau equation \eqref{LE} and the VML system \eqref{main1} in the whole space with the initial datum
\begin{align}\label{L-id}
F^\vps(0,x,v)=F_0^\vps(x,v)
\end{align}
for the Landau equation \eqref{LE} and
\begin{align}\label{VL-id}
F^\vps(0,x,v)=F_0^\vps(x,v),\  E^\vps(0,x)=E^\vps_0(x),\ B^\vps(0,x)=B^\vps_0(x)
\end{align}
for the VML system \eqref{main1}, respectively.

To make the presentation clear, we divide the rest of this section into several subsections and the first one is concerned with the formal Hilbert expansions for the Landau equation \eqref{LE} and the VML system \eqref{main1}.
%%%%%%%%%%%%%%%%%%%%%%%%%%%%%%%%%%%%%%%%%%%%%
\subsection{Hilbert expansions}
%%%%%%%%%%%%%%%%%%%%%%%%%%%%%%%%%%%%%%%%%%%%%

This subsection focuses on the formal Hilbert expansion for both the Landau equation \eqref{LE} and the VML system \eqref{main1}.

%%%%%%%%%%%%%%%%%%%%%%%%%%%%%%%%%%%%%%%%%%%%%
\subsubsection{Hilbert expansion of the Landau equation}

Let $k\geq 2$ be some integer, we consider the following Hilbert expansion for the Landau equation \eqref{LE}
\begin{align}
 F^{\varepsilon}(t,x,v)=\sum_{n=0}^{2k-1}\varepsilon^nF_n+\varepsilon^kF^{\varepsilon}_R.\label{exp-le}
\end{align}
To determine the coefficients $F_0(t, x, v), \ldots, F_{2k-1}(t, x, v)$, we plug the above expansion into the rescaled equations \eqref{LE} to obtain
\begin{eqnarray}\label{expan0}
&&\partial_t\left(\sum_{n=0}^{2k-1}\varepsilon^nF_n+\varepsilon^kF^{\varepsilon}_R\right) + v \cdot \nabla_x \left(\sum_{n=0}^{2k-1}\varepsilon^nF_n+\varepsilon^kF^{\varepsilon}_R\right)\nonumber\\
&=& \frac{1}{\varepsilon}\mathcal{C}\left(\sum_{n=0}^{2k-1}\varepsilon^nF_n+\varepsilon^kF^{\varepsilon}_R,
 \sum_{n=0}^{2k-1}\varepsilon^nF_n+\varepsilon^kF^{\varepsilon}_R\right),
\end{eqnarray}
and from which one can deduce by comparing the coefficients of $\varepsilon^k$ for $k=-1,0,\cdots, 2k-1$ on both side of the above equation that the coefficients $F_j(t, x, v) (j=0,1,\cdots, 2k-1)$ satisfy
\begin{align}
\frac{1}{\varepsilon}:&\quad ~\mathcal{C}(F_0,F_0)=0,\nonumber\\
\varepsilon^0:&\quad ~\partial_tF_0+v\cdot\nabla_x F_0=\mathcal{C}(F_1,F_0)+\mathcal{C}(F_0,F_1),\nonumber\\
\cdots&\cdots\cdots\cdots\cdots\label{expan1}\\
\varepsilon^n:&\quad ~\partial_tF_n+v\cdot \nabla_xF_n=\sum_{\substack{i+j=n+1\\i,j\geq0}}\mathcal{C}(F_i,F_j), \nonumber\\
\cdots&\cdots\cdots\cdots\cdots\nonumber\\
\varepsilon^{2k-1}:&\quad~\partial_tF_{2k-1}+v\cdot\nabla_x F_{2k-1}=\sum_{\substack{i+j=2k\\i,j\geq1}}\mathcal{C}(F_i,F_j).\nonumber
\end{align}
Then we see from \eqref{expan0} and \eqref{expan1} that the remainder term $F_R^{\varepsilon}(t,x,v)$  solves
\begin{align}%\label{remain}
\partial_tF_R^{\varepsilon}&+v\cdot\nabla_x F_R^{\varepsilon} -\frac{1}{\varepsilon}\left[\mathcal{C}\left(F_R^{\varepsilon},F_0\right)+\mathcal{C}\left(F_0,F_R^{\varepsilon}\right)\right]\notag\\
 &=\varepsilon^{k-1}\mathcal{C}\left(F_R^{\varepsilon},F_R^{\varepsilon}\right) +\sum_{i=1}^{2k-1}\varepsilon^{i-1}\left[\mathcal{C}\left(F_i F_R^{\varepsilon}\right)+\mathcal{C}\left(F_R^{\varepsilon}, F_i\right)\right]+\CP,\nonumber
\end{align}
with $\CP=\sum\limits_{i+j\geq 2k+1\atop{2\leq i,j\leq2k-1}}\varepsilon^{i+j-k}\mathcal{C}(F_i,F_j)$.

\eqref{expan1}$_1$ tells us that $F_0(t,x,v)$ is nothing but a local Maxwellian $\mathbf{M}$, i.e.
\begin{equation}\label{maxwell}
F_0(t,x,v)=\FM=\FM_{[\rho,u,T]}= \frac{\rho(t,x) }{(2\pi RT(t,x))^{3/2}
} \exp\left(-\frac{|v-u(t,x)|^2}{2RT(t,x)}\right).
\end{equation}
Here $\rho(t,x), u(t,x)$ and $T(t,x)$ are the fluid density, bulk velocity and temperature, respectively, $R$ is the gas constant which will be taken to be $\frac{2}{3}$ in this paper for convenience.

With this local Maxwellian $\mathbf{M}$, we write $F^{\varepsilon}(t,x,v)$ as
 \begin{equation}\label{decom0}
 F^{\varepsilon}=\mathbf{M}^{\frac{1}{2}}f^{\varepsilon},
 \end{equation}
then it is easy to see that $f^{\varepsilon}(t,x,v)$ satisfies
\begin{align}\label{m0F1}
    \big(\partial_t+v\cdot\nabla_x\big)f^{\varepsilon}
        +\frac{1}{\varepsilon} \mathcal{L}_{\mathbf{M}}[f^{\varepsilon}]
        =&-f^{\varepsilon}\mathbf{M}^{-\frac{1}{2}}\big(\partial_t+v\cdot\nabla_x\big)\mathbf{M}^{\frac{1}{2}}
    +\varepsilon^{k-1}\Gamma_{\mathbf{M}} ( f^{\varepsilon},
    f^{\varepsilon} )\\
    &+\sum_{i=1}^{2k-1}\varepsilon^{i-1}\Big[\Gamma_{\mathbf{M}}(\mathbf{M}^{-\frac{1}{2}}F_i,f^{\varepsilon})+\Gamma_{\mathbf{M}}(
 f^{\varepsilon}, \mathbf{M}^{-\frac{1}{2}} F_i)\Big]+\CP_0,\nonumber
\end{align}
where $\CP_0=\mathbf{M}^{-\frac{1}{2}}\CP$, $\mathcal{L}_{\mathbf{M}}[f]$ is the linearized Landau collision operator defined by
\begin{align}\label{lm-def}
-\mathcal{L}_{\mathbf{M}}[f]=\mathcal{A}_{\mathbf{M}}[f]+\mathcal{K}_{\mathbf{M}}[f]=&\mathbf{M}^{-\frac{1}{2}}
\big[\mathcal{C}( \mathbf{M},\mathbf{M}^{\frac{1}{2}}f  )+\mathcal{C}( \mathbf{M}^{\frac{1}{2}} f, \mathbf{M} )\big],
\end{align}
and
$\Gamma_{\mathbf{M}}(f,g)$ is the nonlinear Landau collision operator given by
\begin{align}\label{gm-def}
\Gamma_{\mathbf{M}}(f,g)=&\mathbf{M}^{-\frac{1}{2}} \mathcal{C}( \mathbf{M}^{\frac{1}{2}} f, \mathbf{M}^{\frac{1}{2}} g).
\end{align}
Moreover, $\mathcal{A}_{\mathbf{M}}[f]$ and $\mathcal{K}_{\mathbf{M}}[f]$ are defined as
\begin{equation*}
\mathcal{A}_{\mathbf{M}}[f]=\mathbf{M}^{-\frac{1}{2}}
\mathcal{C}( \mathbf{M},\mathbf{M}^{\frac{1}{2}}f  ),\qquad \mathcal{K}_{\mathbf{M}}[f]=\mathbf{M}^{-\frac{1}{2}}
\mathcal{C}( \mathbf{M}^{\frac{1}{2}} f, \mathbf{M} ).
\end{equation*}
On the other hand, we also introduce a global Maxwellian
\begin{equation}\label{Global-Maxwellian}
\mu=  \frac{1 }{(2\pi RT_c)^{3/2}
} \exp\left(-\frac{|v|^2}{ 2RT_c}\right),
\end{equation}
where $T_c>0$ is a positive constant whose precise range will be specified in Corollary \ref{corollary-1}, and set
\begin{equation}\label{decom}
F_R^{\varepsilon}=\mu^{\frac{1}{2}} h^{\varepsilon}.
\end{equation}
Then it is straightforward to see that $h^{\varepsilon}(t,x,v)$ satisfies the following equation
\begin{align}\label{mh}
\big(\partial_t +v\cdot\nabla_x\big) h^{\varepsilon}+\frac{1}{\varepsilon}\mathcal{L}[ h^{\varepsilon}]=& -\frac{1}{\varepsilon}\mathcal{L}_d[ h^{\varepsilon}]+\varepsilon^{k-1}\Gamma( h^{\varepsilon}, h^{\varepsilon})\\
&+\sum_{i=1}^{2k-1}\varepsilon^{i-1}\Gamma(\mu^{-\frac{1}{2}}F_i, h^{\varepsilon})
+\mathcal{C}( h^{\varepsilon}, \mu^{-\frac{1}{2}} F_i)]+\CP_1, \nonumber
\end{align}
where $\CP_1=\mu^{-\frac{1}{2}}\CP$, $\mathcal{L}[f]$ and $\Gamma(f,g)$ are given by
\begin{align}\label{L-def}
-\mathcal{L}[f]=\mathcal{A}[f]+\mathcal{K}[f]=&\mu^{-\frac{1}{2}}
\left[\mathcal{C}( \mu,\mu^{\frac{1}{2}}f  )+\mathcal{C}( \mu^{\frac{1}{2}} f, \mu )\right]
\end{align}
and
\begin{align*}
\Gamma(f,g)=\mu^{-\frac{1}{2}} \mathcal{C}( \mu^{\frac{1}{2}} f, \mu^{\frac{1}{2}} g),
\end{align*}
respectively.

In addition, $\mathcal{L}_d [f]$ is the difference of the linear collision operators $\mathcal{L}_{\FM}[f]$ and $\mathcal{L}[f]$:
\begin{equation}\label{dL}
-\mathcal{L}_d [f]=\mathcal{A}_{d}[f]+\mathcal{K}_{d} [f]=\mu^{-\frac{1}{2}}\left[\mathcal{C}( \mathbf{M}- \mu,\mu^{\frac{1}{2}}f  )+\mathcal{C}( \mu^{\frac{1}{2}} f, \mathbf{M}- \mu )\right].
\end{equation}
Note that the explicit expressions of $\mathcal{A}_{\mathbf{M}} $, $\mathcal{K}_{\mathbf{M}} $, $\mathcal{A}_d $ and $\mathcal{K}_d $ will be given in Section \ref{pr}.
%%%%%%%%%%%%%%%%%%%%%%%%%%%%%%%%%%%%%%%%%%%%%
\subsubsection{Hilbert expansion of the VML system}%\label{VML}
%%%%%%%%%%%%%%%%%%%%%%%%%%%%%%%%%%%%%%%%%%%%%
For the VML system \eqref{main1} and some integer $k\geq 2$, we consider the following Hilbert expansion
\begin{align}\label{expan}
\begin{aligned}
&F^{\varepsilon}=\sum_{n=0}^{2k-1}\varepsilon^nF_n+\varepsilon^kF^{\varepsilon}_R,\quad
E^{\varepsilon}=\sum_{n=0}^{2k-1}\varepsilon^n E_n+\varepsilon^kE^{\varepsilon}_R,\quad
B^{\varepsilon}=\sum_{n=0}^{2k-1}\varepsilon^n B_n+\varepsilon^kB^{\varepsilon}_R.
\end{aligned}
\end{align}

Similar to that of the Landau equation \eqref{LE}, we can deduce that the coefficients $[F_n(t, x, v),$ $E_n(t, x), B_n(t, x)]$ with $0\leq n\leq 2k-1$ and  the remainder $[F_R^{\varepsilon}(t,x,v), E_R^{\varepsilon}(t,x),$ $B_R^{\varepsilon}(t,x)]$
solve
\begin{align}
\frac{1}{\varepsilon}:&\quad \mathcal{C}(F_0,F_0)=0,\nonumber\\
\varepsilon^0:&\quad  \partial_tF_0+v\cdot\nabla_x F_0-\Big(E+v \times B \Big)\cdot\nabla_vF_0=\mathcal{C}(F_1,F_0)+\mathcal{C}(F_0,F_1),\nonumber\\
 &\quad \partial_tE-  \nabla_x \times B =4\pi \int_{\mathbb R^3} vF_0 dv, \nonumber\\
 &\quad \partial_tB+ \nabla_x \times E=0,\nonumber\\
&\quad \nabla_x\cdot E=4\pi \Big(1 -\int_{\mathbb R^3}  F_0 dv\Big), \qquad \nabla_x\cdot B=0,\nonumber\\
\cdots&\cdots\cdots\cdots\cdots\label{expan2}\\
\varepsilon^n:&\quad \partial_tF_n+v\cdot \nabla_xF_n-\Big(E_n+v \times B_n \Big)\cdot\nabla_pF_0-\Big(E+v \times B \Big)\cdot\nabla_vF_n\nonumber\\
&\qquad=\sum_{\substack{i+j=n+1\\i,j\geq0}}\mathcal{C}(F_i,F_j)+\sum_{\substack{i+j=n\\i,j\geq1}}\Big(E_i+v \times B_i \Big)\cdot\nabla_vF_j, \nonumber\\
&\quad\partial_tE_n-\nabla_x \times B_n=4\pi \int_{\mathbb R^3} v F_n dv, \nonumber\\
 &\quad \partial_t B_n+ \nabla_x \times E_n=0,\nonumber\\
&\quad \nabla_x\cdot E_n=-4\pi \int_{\mathbb R^3}  F_n dv, \qquad \nabla_x\cdot  B_n=0,\nonumber\\
\cdots&\cdots\cdots\cdots\cdots\nonumber\\
\varepsilon^{2k-1}:& \quad\partial_tF_{2k-1}+v\cdot\nabla_x F_{2k-1}-\Big(E_{2k-1}+v \times B_{2k-1} \Big)\cdot\nabla_vF_0-\Big(E+v \times B \Big)\cdot\nabla_vF_{2k-1}\nonumber\\
& \quad=\sum_{\substack{i+j=2k\\i,j\geq1}}\mathcal{C}(F_i,F_j)+\sum_{\substack{i+j=2k-1\\i,j\geq1}}\Big(E_i+v \times B_i \Big)\cdot\nabla_vF_j, \nonumber\\
&\quad\partial_tE_{2k-1}-\nabla_x \times B_{2k-1}=4\pi \int_{\mathbb R^3} v F_{2k-1} dv, \nonumber\\
 &\quad\partial_t B_{2k-1}+ \nabla_x \times E_{2k-1}=0,\nonumber\\
&\quad \nabla_x\cdot E_{2k-1}=-4\pi \int_{\mathbb R^3}  F_{2k-1} dv, \qquad \nabla_x\cdot  B_{2k-1}=0\nonumber
\end{align}
and
\begin{align}
&\partial_tF_R^{\varepsilon}+v\cdot\nabla_x F_R^{\varepsilon}-\Big(E_R^{\varepsilon}+v \times B_R^{\varepsilon} \Big)\cdot\nabla_vF_0\nonumber\\
 &\qquad-\Big(E+v \times B \Big)\cdot\nabla_vF_R^{\varepsilon}-\frac{1}{\varepsilon}[\mathcal{C}(F_R^{\varepsilon},F_0)+\mathcal{C}(F_0,F_R^{\varepsilon})]\nonumber\\
&\quad=\varepsilon^{k-1}\mathcal{C}(F_R^{\varepsilon},F_R^{\varepsilon})+\sum_{i=1}^{2k-1}\varepsilon^{i-1}[\mathcal{C}(F_i, F_R^{\varepsilon})+\mathcal{C}(F_R^{\varepsilon}, F_i)]+\varepsilon^k\Big(E_R^{\varepsilon}+v \times B_R^{\varepsilon}\Big)\cdot\nabla_vF_R^{\varepsilon}\nonumber\\
 &\qquad+\sum_{i=1}^{2k-1}\varepsilon^i\Big[\Big(E_i+v \times B_i \Big)\cdot\nabla_vF_R^{\varepsilon}+\Big(E_R^{\varepsilon}+v \times B_R^{\varepsilon} \Big)\cdot\nabla_vF_i\Big]+\varepsilon^{k}\CQ,
\nonumber\\
&\partial_tE_R^{\varepsilon}-\nabla_x \times B_R^{\varepsilon}=4\pi \int_{\mathbb R^3} v F_R^{\varepsilon} dv, \label{remain VML}\\
 &\partial_t B_R^{\varepsilon}+ \nabla_x \times E_R^{\varepsilon}=0,\nonumber\\
&\nabla_x\cdot E_R^{\varepsilon}=-4\pi \int_{\mathbb R^3}  F_R^{\varepsilon} dv, \qquad \nabla_x\cdot  B_R^{\varepsilon}=0,\nonumber
\end{align}
respectively. Here we have denoted $[E_0,B_0]$ by $[E,B]$ and
\begin{align*}
\begin{aligned}
\CQ=\sum_{\substack{i+j\geq 2k+1\\2\leq i,j\leq2k-1}}\varepsilon^{i+j-2k-1}Q(F_i,F_j)
+\sum_{\substack{i+j\geq 2k\\1\leq i,j\leq2k-1}}\varepsilon^{i+j-2k}\Big(E_i+v \times B_i \Big)\cdot\nabla_vF_j.
\end{aligned}
\end{align*}

\eqref{expan2}$_1$ implies that $F_0(t,x,v)=\FM(t,x,v)$ with $\FM(t,x,v)$ being given by \eqref{maxwell}.
If we define $f^{\varepsilon}$ as in \eqref{decom0}, then we can rewrite \eqref{remain VML} as
\begin{align}
\partial_tf^{\varepsilon}&+v\cdot\nabla_xf^{\varepsilon}+\frac{\big(E_R^{\varepsilon}+v \times B_R^{\varepsilon} \big) }{ RT}\cdot \big(v-u\big)\mathbf{M}^{\frac{1}{2}}\nonumber\\
&+\Big(E+v \times B \Big)\cdot\frac{v-u }{ 2RT}f^{\varepsilon}-\Big(E+v \times B \Big)\cdot\nabla_vf^{\varepsilon}+\frac{\mathcal{L}_{\mathbf{M}}[f^{\varepsilon}]}{\varepsilon}\nonumber\\
=&-\mathbf{M}^{-\frac{1}{2}}f^{\varepsilon}\Big[\partial_t+v\cdot\nabla_x-\Big(E+v \times B \Big)\cdot\nabla_v\Big]\mathbf{M}^{\frac{1}{2}}+\varepsilon^{k-1}\Gamma_{\mathbf{M}}(f^{\varepsilon},f^{\varepsilon})\nonumber\\
 &+\sum_{i=1}^{2k-1}
 \varepsilon^{i-1}\Big[\Gamma_{\mathbf{M}}(\mathbf{M}^{-\frac{1}{2}}F_i, f^{\varepsilon})+\Gamma_{\mathbf{M}}(f^{\varepsilon}, \mathbf{M}^{-\frac{1}{2}} F_i)\Big]+\varepsilon^k \Big(E_R^{\varepsilon}+v \times B_R^{\varepsilon}\Big)\cdot\nabla_vf^{\varepsilon}\nonumber\\
 &-\varepsilon^k \Big(E_R^{\varepsilon}+v \times B_R^{\varepsilon}\Big) \cdot\frac{v-u }{ 2RT}f^{\varepsilon}\label{VMLf}\\
 &+\sum_{i=1}^{2k-1}\varepsilon^i\Big[\Big(E_i+v \times B_i \Big)\cdot\nabla_vf^{\varepsilon}+\Big(E_R^{\varepsilon}+v \times B_R^{\varepsilon} \Big)\cdot\nabla_v F_i\Big]\nonumber\\
 &-\sum_{i=1}^{2k-1}\varepsilon^i\Big[\Big(E_i+v \times B_i \Big)\cdot\frac{\big(v-u\big)}{ 2RT}f^{\varepsilon}\Big]+\varepsilon^{k}\CQ_0
\nonumber
\end{align}
and
\begin{align*}
&\partial_tE_R^{\varepsilon}-\nabla_x \times B_R^{\varepsilon}=\int_{\mathbb R^3} v\mathbf{M}^{\frac{1}{2}}f^{\varepsilon} dv, \nonumber\\
 &\partial_t B_R^{\varepsilon}+ \nabla_x \times E_R^{\varepsilon}=0,\\%\label{ME}\\
& \nabla_x\cdot E_R^{\varepsilon}=-\int_{\mathbb R^3}  \mathbf{M}^{\frac{1}{2}} f^{\varepsilon} dv, \qquad \nabla_x\cdot  B_R^{\varepsilon}=0,\nonumber
\end{align*}
where $\CQ_0=\mathbf{M}^{-\frac{1}{2}}\CQ$.

Moreover, if we define $h^{\varepsilon}$ as in \eqref{decom}, we can then deduce that $h^\vps$ satisfies
\begin{align}
\partial_th^{\varepsilon}&+v\cdot\nabla_xh^{\varepsilon}+\frac{\big(E_R^{\varepsilon}+v \times B_R^{\varepsilon} \big) }{ RT}\cdot \big(v-u\big)\mu^{-\frac{1}{2}}\mathbf{M}\nonumber\\
&+\frac{E\cdot v }{ 2RT_c}h^{\varepsilon}-\Big(E+v \times B \Big)\cdot\nabla_vh^{\varepsilon}+\frac{\mathcal{L}[h^{\varepsilon}]}{\varepsilon}\nonumber\\
 =&-\frac{\mathcal{L}_d[h^{\varepsilon}]}{\varepsilon}+\varepsilon^{k-1}\Gamma(h^{\varepsilon},h^{\varepsilon})+\sum_{i=1}^{2k-1}\varepsilon^{i-1}[\Gamma(\mu^{-\frac{1}{2}}F_i, h^{\varepsilon})+\Gamma(h^{\varepsilon}, \mu^{-\frac{1}{2}}F_i)]\nonumber\\
 &+\varepsilon^k \Big(E_R^{\varepsilon}+v \times B_R^{\varepsilon}\Big)\cdot\nabla_vh^{\varepsilon}
 -\varepsilon^k \frac{E_R^{\varepsilon}\cdot v}{ 2RT_c}h^{\varepsilon}\notag\\
 &+\sum_{i=1}^{2k-1}\varepsilon^i\Big[\Big(E_i+v \times B_i \Big)\cdot\nabla_vh^{\varepsilon}+\Big(E_R^{\varepsilon}+v \times B_R^{\varepsilon} \Big)\cdot\nabla_v F_i\Big]-\sum_{i=1}^{2k-1}\varepsilon^i\frac{E_i \cdot v}{ 2RT_c}h^{\varepsilon}+\varepsilon^{k}\CQ_1,
\nonumber
\end{align}
where $\CQ_1=\mu^{-\frac{1}{2}}\CQ$.
\begin{remark}
To simplify notation, without causing confusion, for both the Landau equation \eqref{LE} and the VML system \eqref{main1},
we will use $F^\vps(t,x,v)$ to denote their solutions, $f_j(t,x,v)$  to stand for the coefficients of the Hilbert expansion with $0\leq j\leq 2k-1$, $f^\vps(t,x,v)$ and $h^\vps(t,x,v)$  the corresponding remainders. It is easy to see that $\FM_{[\rho,u,T]}:=\FM^{\frac 12}_{[\rho,u,T]}f_0(t,x,v)$.
\end{remark}

\subsection{Notations}
%%%%%%%%%%%%%%%%%%%%%%%%%%%%%%%%%%%%%%%%%%%%%%%%%%%%%%%%%%%%%%%%%%%%%%%%%%%%%%%%%%

Throughout this paper, $C$ denotes a generic positive constant which is independent of the Knudsen number $\varepsilon$ but may change line by line. The notation $a \lesssim b$ implies that there exists a generic positive constant $C$ independent of the Knudsen number $\varepsilon$ such that $a \leq Cb$ and $a \approx b$ means that both $a \lesssim b$ and $b \lesssim a$ hold.

The multi-indices $ \alpha= [\alpha_1,\alpha_2, \alpha_3]$ and $\beta = [\beta_1, \beta_2, \beta_3]$ will be used to record spatial and velocity derivatives, respectively. And $\partial^{\alpha}_{\beta}=\partial^{\alpha_1}_{x_1}\partial^{\alpha_2}_{x_2}\partial^{\alpha_3}_{x_3} \partial^{\beta_1}_{ v_1}\partial^{\beta_2}_{ v_2}\partial^{\beta_3}_{ v_3}$. Similarly, the notation $\partial^{\alpha}$ will be used when $\beta=0$ and likewise for $\partial_{\beta}$. The length of $\alpha$ is denoted by $|\alpha|=\alpha_1 +\alpha_2 +\alpha_3$. $\alpha'\leq  \alpha$ means that no component of $\alpha'$ is greater than the corresponding component of $\alpha$, and $\alpha'<\alpha$ means that $\alpha'\leq  \alpha$ and $|\alpha'|<|\alpha|$. And it is convenient to write $\nabla_x^k=\partial^{\alpha}$ with $|\alpha|=k$.

$(\cdot, \cdot)$  is used to denote the $L^2\times L^2$ inner product in ${\mathbb{ R}}^3_{ v}$, with the ${L^2}$ norm $|\cdot|_{L^2}$. For notational simplicity, $\langle\cdot,\cdot\rangle$  denotes the ${L^2}\times L^2$ inner product either in ${\mathbb{ R}}^3_{x}\times{\mathbb{ R}}^3_{ v }$ or in ${\mathbb{ R}}^3_{x}$ with the ${L^2}\times L^2$ norm $\|\cdot\|$. For each non-negative integer $k$ and $1\leq p\leq +\infty$, we also use $W^{k,p}$ to denote the standard Sobolev spaces for $(x, v)\in {\mathbb R}^3 \times {\mathbb R}^3$ or $x \in {\mathbb R}^3$, and denote $H^s=W^{s,2}$ with
$\|f\|^2_{H^s}:=\sum_{|\alpha|=0}^s\|\partial^{\alpha}f\|^2.$ To simplify the presentation, for each $i\in\mathbb{N}$, we use  $\left(\nabla_x^if, \nabla^i_xg\right)$ and $\left\langle\nabla_x^if, \nabla^i_xg\right\rangle$ to denote $\sum_{|\alpha|=i}\left(\partial^\alpha f, \partial^\alpha g\right)$ and $\sum_{|\alpha|=i}\left\langle\partial^\alpha f, \partial^\alpha g\right\rangle$, respectively. Recall that $\nabla_x=(\partial_{x_1}, \partial_{x_2}, \partial_{x_3})$, $\nabla_v=(\partial_{v_1}, \partial_{v_2}, \partial_{v_3})$.

Let
$$
\sigma^{ij}(v)=\int_{{\R}^{3}}\phi^{ij}(v-v')\mu(v')dv'
$$
with $\phi^{ij}(v)$ being defined by \eqref{cker}, we now introduce the dissipation norm in $v\in {\mathbb R}^3$ by
\begin{eqnarray}\label{D-norm}
    |f(t,x)|_D^2&:=&\int_{{\mathbb R}^3}\sigma^{ij}(v)\partial_{v_i} f(t,x,v)\partial_{v_j}f(t,x,v)\, dv\nonumber\\
    &&+\frac{1}{4R^2T^2_c}\int_{{\mathbb R}^3} \sigma^{ij}(v)v_iv_j|f(t,x,v)|^2\, dv,
\end{eqnarray}
the corresponding dissipation norm in $(x, v)\in {\mathbb R}^3\times {\mathbb R}^3$ by
\begin{eqnarray*}
    \|f(t)\|_D^2&:=&\iint_{{\mathbb R}^3\times{\mathbb R}^3}\sigma^{ij}(v)\partial_{v_i} f(t,x,v)\partial_{v_j}f(t,x,v)\, dv dx\\
    &&+\frac{1}{4R^2T^2_c}\iint_{{\mathbb R}^3\times{\mathbb R}^3}\sigma^{ij}(v)v_iv_j|f(t,x,v)|^2\, dv dx,
\end{eqnarray*}
and the Sobolev dissipation norm
\begin{align*}
    \|f(t)\|^2_{H^s_{D}}:=\sum_{|\alpha|=0}^s\left\|\partial^{\alpha}f(t)\right\|^2_D,\qquad s\geq2.
\end{align*}
Here and in the sequel, Einstein's summation convention is used.
%where $\partial^{\alpha}=\partial^{\alpha_1}_{x_1}\partial^{\alpha_2}_{x_2}\partial^{\alpha_3}_{x_3}$ with $\alpha=(\alpha_1,\alpha_2,\alpha_3)$ and $|\alpha|=\alpha_1+\alpha_2+\alpha_3$.

For constant $\ell\leq s$, we set $\lag v\rag=\sqrt{1+|v|^2}$ and introduce the following velocity weight functions
\begin{align}\label{tt 01}
    w_{i}(t,v):=\lag v\rag^{\ell-i}\exp\left(\frac{(1+|v|^2)}{8RT_c\ln(\mathrm{e}+t)}\right),\quad  0\leq i\leq s
\end{align}
and then define the weighted norms
\begin{align*}%\label{wM}
    \|wh\|^2_{H^s}:=\sum_{i=0}^s\left\|w_{i}\nabla_x^i h\right\|^2 ,\qquad
    \|wh\|^2_{H^s_D}:=\sum_{i=0}^s\left\|w_{i}\nabla_x^i h\right\|^2_D.
\end{align*}
For simplicity, we also denote
\begin{align*}
    X(t):=\exp\left(\frac{1}{8RT_c\ln(\mathrm{e}+t)}\right)
\end{align*}
and it is eay to see that
\begin{align*}
    Y(t):=-\frac{X'(t)}{X(t)}=\frac{1}{8RT_c(\mathrm{e}+t)\big(\ln(\mathrm{e}+t)\big)^2}>0.
\end{align*}

Finally, it is easy to see that the null space of the linearized Landau operator $\mathcal{L}_{\mathbf{M}}$ is given by
\[\mathcal {N}_\FM=\mbox{span}\left\{\mathbf{M}^{\frac{1}{2}}, (v_i-u_i)\mathbf{M}^{\frac{1}{2}}(1\leq i\leq3),\left(\frac{|v-u|^2}{RT}-3\right)\mathbf{M}^{\frac{1}{2}}\right\},
\]
while the null space of the linearized operator $\mathcal{L}$ is given by
\[\mathcal {N}=\mbox{span}\left\{\mu^{\frac{1}{2}}, v_i\mu^{\frac{1}{2}}(1\leq i\leq3),(|v|^2-3)\mu^{\frac{1}{2}}\right\}.
\]
We now define ${\bf P}_{\mathbf{M}}$ as the orthogonal projection from $L^2_v$ onto $\mathcal {N}_{\bf M}$,
and ${\bf P}$ as the orthogonal projection from $L^2_v$ onto $\mathcal {N}$, respectively. Furthermore, let us use $\chi_j$$(0\leq j\leq4)$ to denote the orthonormal basis of the null space of $\CL_\FM$ and be given by
\begin{align*}
\chi_0= \frac{1}{\sqrt{\rho}}\sqrt{\FM},\
\chi_{i}=\frac{v_i-u_i}{{\sqrt{R\rho T}}}
\sqrt{\FM},\ \ i=1,2,3,\
\chi_{4}=\frac{1}{\sqrt{6\rho}}\left(\frac{|v-u|^2}
{2RT}-3\right)\sqrt{\FM},
\end{align*}
then we write
\begin{align}\label{fn-mac-def}
{\bf P}_{\mathbf{M}}[f_n]=\frac{\rho_n}{\sqrt{\rho}} \chi_0+\sum\limits_{i=1}^3\frac{1}{\sqrt{R\rho T}}u_n^i\chi_i+\frac{T_n}{\sqrt{6\rho}}\chi_4.
\end{align}
Here the hydrodynamic field of $f_n$ is defined to be
$[\rho_n,u_n,T_n]$,
which represents the density, velocity, and temperature fluctuations physically.

%Let $\varphi(v)$ be an exponentially decay smooth function of $v$, define
% \begin{equation*}
%\begin{aligned}
%\sigma^{ij}_{\varphi}(v) &=\sigma^{ij}\ast \varphi=\int_{\mathbb R^3}\sigma^{ij}(v-v')\varphi(v')dv'.
%\end{aligned}
%\end{equation*}

%Correspondingly, for the non-cutoff Boltzmann equation, we define the following dissipation norms
%\begin{align*}
%    |f|_D^2:=&\int_{{\mathbb R}^6}\int_{{\mathbb S}^2}\sigma^{ij}B(v-v_*,\sigma)\mu(v_*)\big[f(v_*)-f(v)\big]\overline{\big[f(v_*)-f(v)\big]}\, d\sigma dv_*dv\\
%    &+\int_{{\mathbb R}^6}\int_{{\mathbb S}^2}\sigma^{ij}B(v-v_*,\sigma)\mu(v_*)f(v_*)\overline{f(v_*)}
%    \big[\mu^{\frac{1}{2}}(v_*)-\mu^{\frac{1}{2}}(v)\big]^2\, d\sigma dv_*dv,
%\end{align*}
%
%\begin{align*}
%    \|f\|_D^2:=&\int_{{\mathbb R}^9}\int_{{\mathbb S}^2}\sigma^{ij}B(v-v_*,\sigma)\mu(v_*)\big[f(v_*)-f(v)\big]\overline{\big[f(v_*)-f(v)\big]}\, d\sigma dv_*dvdx\\
%    &+\int_{{\mathbb R}^9}\int_{{\mathbb S}^2}\sigma^{ij}B(v-v_*,\sigma)\mu(v_*)f(v_*)\overline{f(v_*)}
%    \big[\mu^{\frac{1}{2}}(v_*)-\mu^{\frac{1}{2}}(v)\big]^2\, d\sigma dv_*dvdx,
%\end{align*}
%and
%\begin{align*}
%    \|f\|^2_{H^s_{D}}:=\sum_{|\alpha|=0}^s\|\partial^{\alpha}f\|^2_D.
%\end{align*}
%
%\begin{remark} Note that in the non-cutoff Boltzmann equation case, the norm $N^{s,\gamma}$ was used in \cite{Gressman-Strain-2011} instead of the norm $D$. As illustrated in \cite{Gressman-Strain-2011}, these two norms are equivalent. In the rest of this article, we use the norm $D$ uniformly for convenience.
%\end{remark}
\subsection{Solutions of fluid equations}
Formally, by the first two equations of \eqref{expan1}, we see that the leading order $\FM$ of the Hilbert expansion of the Landau equation \eqref{LE} gives the compressible Euler equations:
\begin{eqnarray}\label{euler}
\partial_t\rho + \nabla_x\cdot(\rho u)&=&0,\nonumber\\
\partial_t[\rho u] + \nabla_x\cdot[\rho u\otimes u]+\nabla_xp&=&0,\\
\partial_t\left[\rho\left(\FE+\frac12|u|^2\right)\right] + \nabla_x\cdot\left[\rho u\left(\FE+\frac12|u|^2\right)\right]+\nabla_x(pu)&=&0,\nonumber
\end{eqnarray}
with the equation of state
\begin{equation}\notag
p=\rho RT=\frac23\rho \FE,
\end{equation}
where $\FE$ is the internal energy. Recalling that we take $R=\frac 23$ in this paper, we can thus conclude that $\FE=T$.

Similarly, by setting $T=\FE=\rho^{2/3}$, the leading order of the Hilbert expansion \eqref{expan} for the VML system \eqref{main1} yields the following compressible Euler-Maxwell system
\begin{eqnarray}\label{EM}
\partial_t\rho + \nabla_x\cdot(\rho u) &=&0,\nonumber\\
\partial_t[\rho u] + \nabla_x\cdot\left(\rho u\otimes u\right)+\frac{2}{3}\nablax \rho^{\frac{5}{3}}+\rho \big(E+u\times B\big)&=&0,\nonumber\\
 \partial_tE-  \nabla_x \times B &=&4\pi \rho u, \\
 \partial_tB+ \nabla_x \times E&=&0,\nonumber\\
\nabla_x\cdot E&=&4\pi \big(1 -\rho\big), \nonumber\\
\nabla_x\cdot B&=&0.\nonumber
\end{eqnarray}

Now we state some well-established well-posedness results on the Cauchy problems of the Euler equations \eqref{euler} and the Euer-Maxwell system \eqref{EM}.

In fact, for the local smooth solvability of the Euler equations \eqref{euler} with prescribed initial datum
\begin{align}\label{e-id}
[\rho, u,T](0,x)=[\rho_0,u_0,T_0](x),
\end{align}
we have the following result, cf.\cite[Theorem 2.1, pp.30]{majda-book}.
\begin{proposition}[Local solvability of the Euler system]\label{e-loc} Assume $[\rho_0-1,u_0,T_0-1]\in H^s$ with $s>\frac{5}{2}$ and
$$
\inf\limits_{x\in\R^3}\rho_0(x)>0,\ \ \inf\limits_{x\in\R^3}T_0(x)>0.
$$
Moreover, let
$$
\eps_0=\|[\rho_0-1,u_0,T_0-1]\|_{H^s}.
$$
Then there exists a time interval $[0,t_e]$ with $t_e>0$ such that the Cauchy problem \eqref{euler} and \eqref{e-id} admits a unique smooth solution $[\rho,u,\ta](t,x)\in C^1([0,t_e]\times\R^3)$ with
$$
\inf\limits_{(t,x)\in [0,t_e]\times\R^3}\rho(t,x)>0,\ \ \inf\limits_{(t,x)\in [0,t_e]\times\R^3}T(t,x)>0
$$
and $t_e$ depending on $\|[\rho_0-1,u_0,T_0-1]\|_{H^s}$, $\inf\limits_{x\in\R^3}\rho_0(x)$ and $ \inf\limits_{x\in\R^3}T_0(x)$.  Furthermore, there exists a constant $C(t_e)>0$ independent of $\epsilon_0$ such that
\begin{align}
\sup\limits_{t\in[0,t_e]}\|[\rho(t,x)-1,u(t,x),T(t,x)-1]\|_{H^s}\leq C(t_e)\eps_0.\label{e-sol-es}
\end{align}
\end{proposition}

For the Cauchy problem of the Euler-Maxwell system \eqref{euler} with initial datum
\begin{align}\label{em-id}
[\rho,u,T,E,B](0,x)=[\rho_{0},u_{0},T_{0},E_{0},B_{0}](x), \ T_0=\rho_0^{2/3},
\end{align}
it is shown in \cite[Theorem 1.3, pp. 2359]{Ionescu-Pausader-JEMS-2014} that such a Cauchy problem does exists a global smooth solution $[\rho(t,x), u(t,x), E(t,x), B(t,x)]$.
\begin{proposition}[Global solvability of the Euler-Maxwell system]\label{em-ex-lem}
Let $N_0>10^4$ be a sufficiently large integer and denote
$$
\left\|\left[\rho_0-1, u_0, T_0-1, E_0, B_0\right]\right\|_{H^{N_0+1}}+\left\|(\FI-\Delta)^{1/2}u_0\right\|_Z+\left\|(\FI-\Delta)^{1/2}E_0\right\|_Z=\epsilon_1,
$$
where the norm $\|\cdot\|_Z$ is defined by \cite[(2.18), pp. 2371]{Ionescu-Pausader-JEMS-2014}.

Assume that
$$
\rho_0=-\na_x\cdot E_0,\ B_0=\na_x\times u_0,
$$
and $\epsilon_1>0$ is chosen suitably small, there exists a unique global smooth solution $[\rho(t,x), u(t,x), $ $T(t,x), E(t,x), B(t,x)]$ with $T(t,x)=\rho(t,x)^{2/3}$ to the Cauchy problem \eqref{EM} and \eqref{em-id} which satisfies
\begin{align}\label{em-decay}
&\sup_{t\in[0,\infty]}\left\|\left[\rho(t,x)-1, T(t,x)-1, u(t,x), E(t,x), B(t,x)\right]\right\|_{H^{N_0+1}}\\
&+\sup_{t\in[0,\infty]}\left[(1+t)^{p_0}\sup_{i\leq N_1}\left(\left\|\nabla_x^i[\rho(t,x)-1]\right\|_{\infty}
+\left\|\nabla_x^i[T(t,x)-1]\right\|_{\infty}\right)\right]\nonumber\\
&+\sup_{t\in[0,\infty]}\Big[(1+t)^{p_0}\sup_{i\leq N_1+1}
\left\|\left[\nabla_x^i u(t,x), \nabla_x^i E(t,x), \nabla_x^i B(t,x)\right]\right\|_{\infty}\Big]\leq C \epsilon_1,\nonumber
\end{align}
where $3\leq N_1\ll N_0$, $p_0>1$, $C$ is some positive constant independent of $\epsilon_1$.
\end{proposition}

From the estimates \eqref{e-sol-es} and \eqref{em-decay}, it is seen that for any given generic positive constant $C_0\geq 1$, if both $\eps_0$ and $\eps_1$ are chosen suitably small, we can deduce that the solutions obtained in Propositions  \ref{e-loc} and \ref{em-ex-lem} satisfy the following property
\begin{corollary}\label{corollary-1}
For any given generic positive constant $C_0\geq 1$, if we assume further that $\eps_0$ and $\eps_1$ are sufficiently small, then one can always find a positive constant $T_c$ such that either
\begin{align}
 %&\sup_{t\in [0, t_1]}\big[\|\rho(t)-1\|_{H^{N_0}}+\|u(t)\|_{H^{N_0}}+\|T(t)-\bar{T}\|_{H^{N_0}}\big]\leq \epsilon_0, \label{tt}\\
\frac{ \epsilon_0}{C_0}\leq T(t,x)-T_c\leq C_0 \epsilon_0 \label{tt0}
 \end{align}
 holds for all $(t,x)\in[0,t_e]\times\R^3$
 or
 \begin{align}
 %&\sup_{t\in [0, t_1]}\big[\|\rho(t)-1\|_{H^{N_0}}+\|u(t)\|_{H^{N_0}}+\|T(t)-\bar{T}\|_{H^{N_0}}\big]\leq \epsilon_0, \label{tt}\\
\frac{ \epsilon_1}{C_0}\leq  T(t,x)-T_c\leq C_0\epsilon_1 \label{tt1}
 \end{align}
holds for all $(t,x)\in[0,\infty)\times\R^3$.
\end{corollary}

\subsection{Main results}
In this subsection, we shall present our main results. Our first main result is concerned with the rigorously justification of the local validity of the Hilbert expansion of the Landau equation \eqref{LE}. Before stating our result, we first define the following energy functional
\begin{align}\label{eg-le}
\mathcal{E}(t)=\sum_{i=0}^2 \varepsilon^i\left(\left\|\nabla_x^if^{\varepsilon}(t)\right\|^2+\varepsilon\left\|\nabla_x^ih^{\varepsilon}(t)\right\|^2\right)
\end{align}
and the following energy dissipation rate functional
\begin{align}\label{dn-le}
\mathcal{D}(t)\backsimeq\sum_{i=0}^2 \varepsilon^{i-1}\left(\left\|\nabla_x^i({\bf I}-{\bf P}_{\mathbf{M}})[f^{\varepsilon}](t)\right\|^2_D+\varepsilon\left\|\nabla_x^i({\bf I}-{\bf P})[h^{\varepsilon}](t)\right\|^2_D\right),
\end{align}
then our first result can be stated as follows.
\begin{theorem}[Local Hilbert expansion of the Landau equation]\label{resultLB}
Assume that
\begin{itemize}
\item The initial daum $F^\varepsilon_0(x,v)$ satisfies
$$
F^\vps_0(x,v)=\sum_{n=0}^{2k-1}\varepsilon^nF_{n,0}(x,v)+\varepsilon^kF^{\varepsilon}_{R,0}(x,v)\geq0
$$
and  $[\rho(t,x), u(t,x), T(t,x)]$ is a smooth solution of the Euler equation \eqref{euler}  constructed in Proposition \ref{e-loc};
\item $T_c>0$ is suitably chosen such that
\eqref{tt0}
holds for all $(t,x)\in[0,t_e]\times\R^3$;
\item There exists constant $C>0$ such that for $N\geq2$
\begin{align}
\sum\limits_{\al_0+|\al|\leq N+4k-2n+2}&\left\|\pa_t^{\al_0}\pa^\al\left[\rho_{n,0},u_{n,0},T_{n,0}\right](x)\right\|\leq C,\ \ 1\leq n\leq2k-1;\label{ld-Fn-id}
\end{align}
\item  $k\geq3$ and
\begin{align*}
    \mathcal{E}(0)\lesssim 1,
\end{align*}
\end{itemize}
then there exists a constant $\varepsilon_0 > 0$ such that for $0 <\varepsilon\leq\varepsilon_0$, the Cauchy problem of the Landau equation \eqref{LE} and \eqref{L-id} admits a unique smooth solution $F^{\varepsilon}(t,x,v)$ which satisfies
 $$
 F^{\varepsilon}(t,x,v)=\sum_{n=0}^{2k-1}\varepsilon^nF_n+\varepsilon^kF^{\varepsilon}_R\geq 0
 $$
in the time interval $[0,t_e]$. Moreover, for any $t\in [0, t_e]$, it holds that
\begin{align*}%\label{TT0}
    \lim_{\varepsilon\rightarrow 0^+}\sup_{0\leq t\leq t_e}\left\|{\mathbf{M}}^{-\frac{1}{2}}(t)\left(F^{\varepsilon}(t)-\mathbf{M}(t)\right)\right\|_{H^2}=0,
\end{align*}
and
\begin{align}\label{thm1}
&\mathcal{E}(t)+\int_0^t\mathcal{D}(s)\, d s\lesssim \mathcal{E}(0)+\varepsilon^{2k+3}.
\end{align}
\end{theorem}

Our second result in the article is concerned with the Hilbert expansion of the VML system \eqref{main1}.
As \eqref{eg-le} and \eqref{dn-le}, we first define
the energy functional
\begin{align}\label{eg-vml}
&\mathcal{E}(t)=\sum_{i=0}^2 \varepsilon^i\left[\left(\left\|\sqrt{4\pi RT}\nabla_x^if^{\varepsilon}(t)\right\|^2+\left\|\nabla_x^iE_R^{\varepsilon}(t)\right\|^2+\left\|\nabla_x^i B_R^{\varepsilon}(t)\right\|^2\right)
+\varepsilon^{\frac{4}{3}}\left\|w_i\nabla_x^ih^{\varepsilon}(t)\right\|^2\right],
\end{align}
and the dissipation rate functional
\begin{align}\label{dn-vml}
\mathcal{D}(t)\backsimeq&\sum_{i=0}^2 \varepsilon^i\left[\frac{1}{\varepsilon}\left\|\nabla_x^i({\bf I}-{\bf P}_{\mathbf{M}})\left[f^{\varepsilon}\right](t)\right\|^2_D+\varepsilon^{\frac{1}{3}}\left\|w_i\nabla_x^i h^{\varepsilon}(t)\right\|^2_D\right.\\
&\left.+\varepsilon^{\frac{4}{3}}Y(t)\left\|(1+|v|)w_i\nabla_x^i h^{\varepsilon}(t)\right\|^2\right],\nonumber
\end{align}
then our main result can be stated as follows

\begin{theorem}[Global Hilbert expansion of the VML system]\label{resultVML} Assume that
\begin{itemize}
\item The initial daum $[F^\varepsilon_0(x,v), E^{\varepsilon}_0(x), B^{\varepsilon}_0(x)]$ satisfies
\begin{eqnarray*}
F^\vps_0&=&\sum_{n=0}^{2k-1}\varepsilon^nF_{n,0}+\varepsilon^kF^{\varepsilon}_{R,0}\geq0,
\end{eqnarray*}
and
\begin{align*}
E^{\varepsilon}_0=\sum_{n=0}^{2k-1}\varepsilon^n E_{n,0}+\varepsilon^kE^{\varepsilon}_{R,0},\
B^{\varepsilon}_0=\sum_{n=0}^{2k-1}\varepsilon^n B_{n,0}+\varepsilon^kB^{\varepsilon}_{R,0},
\end{align*}
and $[\rho(t,x), u(t,x)$, $T(t,x), E(t,x),  B(t,x)]$
is a smooth solution constructed in Proposition \ref{em-ex-lem} for the Euler-Maxwell system \eqref{EM};
\item $T_c$  is suitably chosen such that there exists a positive constant $C_0>1$
\eqref{tt1}
is valid
for all $(t,x)\in[0,\infty)\times\R^3$;
\item There exists constant $C>0$ such that for $N\geq2$
\begin{align}
\sum\limits_{\al_0+|\al|\leq N+4k-2n+2}&\left\|\pa_t^{\al_0}\pa^\al[\rho_{n,0},u_{n,0},T_{n,0},E_{n,0},B_{n,0}](x)\right\|\leq C,\ \ 1\leq n\leq2k-1;\label{Fn-id}
\end{align}
\item  $k\geq3$ and
\begin{align*}
    \mathcal{E}(0)\lesssim 1,
\end{align*}
\end{itemize}
then there exists a constant $\varepsilon_0 > 0$ such
that for $0 < \varepsilon\leq\varepsilon_0$,  the Cauchy problem of the VML system \eqref{main1}  and \eqref{VL-id} admits a unique smooth solution
 $[F^{\varepsilon}(t,x,v), E^{\varepsilon}(t,x), B^{\varepsilon}(t,x)]$ satisfying
 \begin{align}%\label{expan}
\begin{aligned}
&F^{\varepsilon}=\sum_{n=0}^{2k-1}\varepsilon^nF_n+\varepsilon^kF^{\varepsilon}_R\geq0,\quad
E^{\varepsilon}=\sum_{n=0}^{2k-1}\varepsilon^n E_n+\varepsilon^kE^{\varepsilon}_R,\quad
B^{\varepsilon}=\sum_{n=0}^{2k-1}\varepsilon^n B_n+\varepsilon^kB^{\varepsilon}_R.
\end{aligned}\notag
\end{align}
Furthermore, it holds that
\begin{align*}%\label{TT0}
   \lim_{\varepsilon\rightarrow 0^+} \sup_{0\leq t\leq \varepsilon^{-1/3}}\left\|\FM^{-\frac{1}{2}}(t)\left(F^{\varepsilon}(t)-\mathbf{M}(t)\right)\right\|_{H^2}=0
\end{align*}
and
\begin{align}\label{TVML1}
&\mathcal{E}(t)+\int_0^{t}\mathcal{D}(s)\, d s\lesssim \mathcal{E}(0)+1,\ 0\leq t\leq\varepsilon^{-1/3}.
\end{align}

\end{theorem}
The following remark is concerned with the initial condition of the coefficients $F_n$$(1\leq n\leq2k-1)$ in the Hilbert expansion.
\begin{remark}
The
$t-$derivatives in \eqref{ld-Fn-id} and  \eqref{Fn-id} are understood through the moment equations of evolution equations \eqref{expan1} and \eqref{expan2}, for instance, $\pa_t[\rho_{n,0},u_{n,0},T_{n,0},E_{n,0},B_{n,0}](x)$ is given by \eqref{em-abc-eq}. It should also be pointed out that only the initial data of the macroscopic part of the coefficients are imposed because the microscopic component is determined by iteration relation in \eqref{ld-Fn-id} and  \eqref{Fn-id}.

\end{remark}

\subsection{Literatures}

There have been many results on the link between the kinetic equations and the systems of fluid dynamics (i.e., Euler or Navier-Stokes equations). It turns out that the relations between the fluid-dynamic picture and the kinetic description can be explored from several mathematical points of view. One of those is to study  the equivalence of kinetic equations and fluid systems in the time-asymptotic sense, see \cite{Kawashima-Matsumura-Nishida-1979}, particularly we refer to \cite{DL-2015, DYY-Landau-cw, Huang-Yang, Li-Wang-Wang-2022, Liu-Yang-Yu-Zhao-2006} for the study of this equivalence from the perspective of dissipative waves.

Another available strategy is to study the small Knudsen number limits of the kinetic equations. Generally speaking, there are also some different kinds of program to investigate such kinds of approximation. One is based on the Chapman-Enskog procedure \cite{Chapman-1990}, the first order approximation of which gives the compressible Euler equations and the second order approximation is the compressible Navier-Stokes system. Along this direction, Nishida \cite{Nishida-1978}, Ukai-Asano \cite{Ukai-Asano-1983} proved this compressible Euler approximation by semigroup method in the framework of analytical function space. Recently, such a validity has been established near a shock wave \cite{Yu-2005}, rarefaction wave \cite{Xin-Zeng-2010}, contact wave \cite{Huang-Wang-Yang-2010-1} and the composition of the basic wave patterns \cite{Huang-Wang-Yang-2010-2, Huang-Wang-Wang-Yang}. We also mention that the validity of the Chapman-Enskog expansion of the Bolzmann equation to the solution of compressible Navier-Stokes system was studied in \cite{Lachowicz-1992} and was recently justified by Liu-Yang-Zhao \cite{LYZ-2014} for the Cauchy problem and by Duan-Liu \cite{DL-2021} for initial-boundary value problem with physical boundary conditions. The second program mainly developed by Bardos-Golse-Levermore is concerned with the incompressible fluid limits via the renormalized solutions, entropy and weak convergence method. We refer to Bardos-Golse-Levermore \cite{Bardos-Golse-Levermore-1991, Bardos-Golse-Levermore-1993}, Golse-Saint-Raymond \cite{Golse-S-04, Golse-S-09}, Jiang-Masmoudi\cite{Jiang-Masmoudi-CPAM-2017}, Lions-Masmoudi \cite{Lions-Masmoudi-2001},  Masmoudi-Saint-Raymond \cite{Masmoudi-Saint-Raymond-2003}, Saint-Raymond \cite{Saint-Raymond-2009}. Moreover, Guo \cite{Guo-CPAM-2006} discussed the incompressible Navier-Stokes-Fourier limits of both the Boltzmann equation and Landau equation via diffusive expansion in the framework of classical solutions.

To justify the small Knudsen number limits of the kinetic equations rigorously, the third approach is based on the Hilbert expansion. For results in this direction, the validity of the Hilbert expansion of the Bolzmann equation in one dimension to the classical solutions of the compressible Euler equations in a finite time, was justified by Caflisch \cite{Caflisch-CPAM-1980} via a robust decomposition of the remainder equation. Later on, Lachowicz studied the multidimensional case including the initial layer in Sobolev space \cite{Lachowicz-1987} and also in $C^0$ space \cite{Lachowicz-1991}. Around 2010, the compressible Euler limits and acoustic limits as well as the nonnegativity of solution of the Boltzmann equation were justified by Guo-Jang-Jiang \cite{Guo-JJ-KRM-2009, Guo-JJ-CPAM-2010} with the aid of a new $\vps$-dependent energy functional and an $L^2-L^\infty$ approach developed in \cite{Guo-ARMA-2010}. When the physical boundary is also taken into consideration, Esposito-Lebowitz-Marra \cite{ELM,ELM2} investigated the Hilbert expansion for the stationary Boltzmann equation in a slab with diffusive reflection boundary condition. Very recently, Guo-Huang-Wang \cite{GHW-ARAM-2021}, Jang-kim \cite{Jang-Kim-2021} and Jiang-Luo-Tang \cite{JLT-21-1, JLT-21-2} studied the half space problem with physical boundary conditions. It should be pointed out that the Hilbert expansion requires the existence of smooth solution of the compressible Euler type equations, thus the solution of the Boltzmann equation is obtained before the shock formations in the compressible Euler flow. For the more complicate model Vlasov-Poisson-Boltzmann (VPB for short in the sequel) system, based on \cite{Guo-CMP-98}, Guo-Jang \cite{Guo-Jang-CMP-2010} proved that the Hilbert expansion of VPB is valid for all time in the $L^2-L^\infty$ setting, which implies the global in time convergence from the VPB system to the Euler-Poisson system. Recently, Ars\'{e}nio-Saint-Raymond \cite{AS-2019} discussed the magnetohydrodynamic limits of the Vlasov-Maxwell-Boltzmann system by mainly using some generalized relative entropy method.

Although the well-posedness of the Landau equation has already been established cf. \cite{Guo-CMP-2002, DLSS-2021}, much less is known about the validity of the Hilbert expansion of the Landau equation, let alone the VML system, which is due to the severe singularity of the large velocities and differential dissipation structure of the Landau collision operator.  We further mention that the stability of contact wave of Landau equation was proved by Duan-Yang-Yu \cite{DYY-Landau-RW} and the global classical solution of the VML system was constructed in \cite{Duan-2014-vml, Y2}. Recently, Rachid \cite{Rachid-2021} investigated the incompressible fluid limits of the Landau equation in a weak convergence regime, Duan-Yang-Yu validate the Chapman-Enskog expansion of the Landau equation \cite{Duan-Yang-Yu-arXiv-Landau} and the Vlasov-Maxwell-Boltzmann system \cite{Duan-Yang-Yu-arXiv-VMB} in a perturbative framework via an energy method and a careful analysis of the Burnett function. In the case of relativistic models, Speck-Strain \cite{Speck-Strain-CMP-2011} studied the local-in-time hydrodynamic limit of the relativistic Boltzmann equation using a Hilbert expansion, and in a series of papers, Guo-Xiao showed the validity of the Hilbert expansion for the relativistic Vlasov-Maxwell-Boltzmann system \cite{Guo-Xiao-CMP-2021} by the generalized $L^2-L^\infty$ approach and Ouyang-Wu-Xiao  discussed the the Hilbert expansion for both the relativistic Landau equation \cite{Ouyang-Wu-Xiao-arxiv-2022-rL} and the relativistic VML system \cite{Ouyang-Wu-Xiao-arxiv-2022-rLM} by an energy method with exponential momentum weights. In this paper, we aim to justify the compressible fluid approximation of both the Landau equation and the VML system in the setting of Hilbert expansion. To the best of our knowledge, our paper appears to be the first result to justify the validity of the Hilbert expansion for such Landau type equations.

\subsection{Strategies and ideas}
One typical difficulty in the rigorous justification of the Hilbert expansion for non-relativistic kinetic equations in perturbative regime is the large velocity growth caused by the leading order $\FM_{[\rho,u,T]}$, which leads to $|v|^3$ growth after acted by transport operator, namely
 \begin{align}\label{vg}
\partial^\alpha f^{\varepsilon}\mathbf{M}^{-\frac{1}{2}}\big(\partial_t+v\cdot\nabla_x\big)\mathbf{M}^{\frac{1}{2}}\lesssim \left\|\pa_{t,x}[\rho,u,T]\right\|_{L^\infty}(1+|v|)^3\left|\partial^\alpha f^\vps\right|.
\end{align}
For the Boltzmann type equations with cut-off potentials, to overcome such a difficulty,  Caflisch \cite{Caflisch-CPAM-1980} and Guo-Jang-Jiang \cite{Guo-JJ-KRM-2009, Guo-Jang-CMP-2010} introduced the global Maxwellian $\mu$ defined by \eqref{Global-Maxwellian} with $ T_c< T(t,x)$, and by making full use of the nice structures and properties of the Boltzmann collision operator with cut-off potentials, they succeeded in justifying rigorously the validity of the Hilbert expansion of these equations. As far as we know, it is still an interesting and challenging problem to see whether the approaches developed in \cite{Caflisch-CPAM-1980, Guo-JJ-KRM-2009, Guo-Jang-CMP-2010} can be adapted to deal with the same problem for the Landau type equations and the Boltzmann type equations for long-range interactions or not.

While for the relativistic Landau type equations, the resulting difficult term corresponding to \eqref{vg} leads to linear growth in momentum only. Based on such an observation, Ouyang-Wu-Xiao \cite{Ouyang-Wu-Xiao-arxiv-2022-rL, Ouyang-Wu-Xiao-arxiv-2022-rLM} proved the validity of the Hilbert expansion for the relativistic Landau equation and the relativistic VML system by  introducing an exponentially weighted energy method around the local Maxwellians
 $\bf M$.

Motivated by these works, the main idea in this paper is to introduce an energy method around both the local Maxwellian $\FM$ and global Maxwellian $\mu$ by setting
\begin{align}\label{relation-local-globsl}
F^\vps_R=\sqrt{\FM}f^\vps=\sqrt{\mu}h^\vps\ \textrm{with}\ T_c<T,
\end{align}
and the key point in our analysis is to make use of an interplay of the energy estimates for $f^{\varepsilon}$ and $h^{\varepsilon}$ such that we can close the energy estimates, from which the wellposedness of the remainder $F_R^\vps$ can be established in Sobolev space.

Now we first outline the whole picture of strategies employed in our proof for the local validity of the Hilbert expansion of the Landau equation \eqref{LE}.

For the energy estimates around the local Maxwellian $\FM$, that is the energy type estimates of $f^{\varepsilon}$, to deal with the problematic term corresponding to \eqref{vg}, we divide ${\mathbb R}^3_v$ into the low velocity part and the high velocity part. For the low velocity part, we can bound its microscopic part and  macroscopic part by $\frac{1}{\varepsilon}\|\partial^{\alpha}({\bf I}-{\bf P}_{\mathbf{M}})[f^{\varepsilon}]\|_D^2$ and $\|f^{\varepsilon}\|^2_{H^{|\alpha|}}$, respectively; while for the high velocity part, noticing that
\begin{align}\label{refh}
\left|\partial^{\alpha}f^{\varepsilon}\right|
\lesssim\sum_{j=0}^i{\lag v\rag}^{2(|\alpha|-j)}\exp\left(-\frac{(T-T_c)|v|^2}{8RTT_c}\right) \left|\partial_x^jh^{\varepsilon}\right|,\quad 0\leq |\alpha|\leq 2,
\end{align}
which is a direct consequence of the assumption that $\mathbf{M}$ and $\mu$ are sufficiently close, we can control it via  $\|h^{\varepsilon}\|^2_{H^{|\alpha|}}$.

To derive  the desired estimates of $h^{\varepsilon}$, we need to estimate the term $\frac{1}{\vps}\big|\big\langle\partial^{\alpha}\mathcal{L}_d[h^{\varepsilon}], \partial^{\alpha}h^{\varepsilon}\big\rangle\big|$, which can be estimated as follows
\begin{align}\label{Ld ex1}
\frac{1}{\vps}\big|\big\langle\partial^{\alpha}\mathcal{L}_d[h^{\varepsilon}], \partial^{\alpha}h^{\varepsilon}\big\rangle\big|\lesssim \frac{\eps_0}{\vps}\big\|h^{\varepsilon}\big\|_{H^{|\alpha|}_D}^2\lesssim \frac{\eps_0}{\vps}\Big(\big\|({\bf I}-{\bf P})[h^{\varepsilon}]\|_{H^{|\alpha|}_D}^2+\big\|{\bf P}[h^{\varepsilon}]\|_{H^{|\alpha|}_D}^2\Big),
\end{align}
but the macroscopic component $\frac{\eps_0}{\vps}\big\|{\bf P}[h^{\varepsilon}]\|_{H^{|\alpha|}_D}^2$ is out of control due to the $\vps$-singularity. To overcome such a difficulty, our another contribution of this paper is to use $f^\vps$ to dominate this trouble term because one can check that
\begin{align}\label{Ld ex2}
\left\|{\bf P}\left[h^{\varepsilon}\right]\right\|^2_{H^{|\alpha|}_D}\lesssim& \left\|{\bf P}\left[f^{\varepsilon}\right]\right\|^2_{H^{|\alpha|}_D} \lesssim \left\|f^{\varepsilon}\right\|^2_{H^{|\alpha|}}
\end{align}
by using \eqref{relation-local-globsl} again.

Now we explain our strategy to make use of the interplay of the estimates $\left\|f^{\varepsilon}(t)\right\|_{H^2}$ and $\left\|h^{\varepsilon}(t)\right\|_{H^2}$ to yield the desired estimates. In fact, when performing the energy estimates on $f^\varepsilon$, as observed in \cite{Ouyang-Wu-Xiao-arxiv-2022-rL, Ouyang-Wu-Xiao-arxiv-2022-rLM}, since the linear collision operator term $\frac{\mathcal{L}_{\mathbf{M}} [f^{\varepsilon}]}{\varepsilon}$ and ${\bf P}_{\mathbf{M}}[f^{\varepsilon}]$ do not commute with the spatial derivative operator $\partial^{\alpha}$, singularities $\varepsilon^{-|\alpha|}$ appear in the estimate of $\|\partial^{\alpha}f^{\varepsilon} \|^2$. Similarly, as a consequence of \eqref{Ld ex1} and \eqref{Ld ex2}, when we deduce an estimate on $\|h^{\varepsilon}\|^2_{H^{|\alpha|}}$, the term $\frac 1\varepsilon\|f^{\varepsilon}\|^2_{H^{|\alpha|}}$ with
the singular factor $\varepsilon^{-1}$ must appear in its upper bound estimates. To remove these singularities, we equip norms $\|\partial^{\alpha}f^{\varepsilon} \|^2$ with an coefficient $\varepsilon^{|\alpha|}$ for $0\leq |\alpha|\leq 2$ and multiply the norms  $\|\partial^{{\alpha}}h^{\varepsilon}\|^2$ by a weight $\varepsilon^{|\alpha|+1}$ to yield the desired energy type estimates. This is the very reason for our introduction of the energy functional $\mathcal{E}(t)$ defined in \eqref{eg-le}. By such an interplay of the energy estimates $\|f^{\varepsilon}\|^2_{H^2} $ and $\|h^{\varepsilon}\|^2_{H^2} $ around both the local Maxwellian $\FM$ and the global Maxwellian $\mu$, we finally arrive at a closed {\it a priori} estimates.

Compared with the local validity of the Hilbert expansion for the Landau equation \eqref{LE}, the problem on the global Hilbert expansion of the VML system \eqref{main1} is more complicated and the main difficulties are two-folds: the first one is due to the presence of the Lorentz force terms, while the another is that we need to show that the Hilbert expansion is valid globally. Our main ideas to solve the first problem can be summarized as in the following:
\begin{itemize}
\item  For the energy estimates of $f^{\varepsilon}$, the biggest difficulty comes from terms like $-(E+v \times B)\cdot\nabla_vf^{\varepsilon}$ which include velocity growth and velocity derivatives. These terms can be estimated in an analogous way as the problematic term \eqref{vg}. In fact, we can also obtain upper bound estimates for $\partial^{\alpha}\nabla_vf^{\varepsilon}$ similar to \eqref{refh}. And for the high velocity part of the corresponding velocity weighted norm estimates of $\partial^{\alpha}\nabla_vf^{\varepsilon}$, it can be controlled via  $\|h^{\varepsilon}\|^2_{H^{|\alpha|}}$ and $\|h^{\varepsilon}\|^2_{H^{|\alpha|}_D}$ for $0\leq |\alpha|\leq 2$.

\item For estimates of the remainder $h^{\varepsilon}$, except for the problems mentioned in the case of the Landau equation \eqref{LE}, there are some additional difficulties. The first one stems from the interaction between particles and electromagnetic field, such as $(E+v \times B)\cdot\nabla_vh^{\varepsilon}$, $\frac{E\cdot v }{ 2RT_c}h^{\varepsilon}$, etc., and these terms include velocity growth and derivatives and thus cannot be controlled directly in the energy estimates.
     %Fortunately, the strategy used to deal with \eqref{vg} is still available to treat the terms involving $(E+v \times B)\cdot\nabla_vf^{\varepsilon}$, but not available for either $(E+v \times B)\cdot\nabla_vh^{\varepsilon}$ or $\frac{E\cdot v }{ 2RT_c}h^{\varepsilon}$.
     To control the term $(E+v \times B)\cdot\nabla_vh^{\varepsilon}$, as in \cite{Guo-JAMS-2011}, a polynomial weight depending on the order of the $x-$derivatives is designed to alleviate the velocity growth, while for the more delicate term $\frac{E\cdot v }{ 2RT_c}h^{\varepsilon}$, it leads to temporal decay but velocity growth in the form of
\begin{align}
\epsilon_1(1+t)^{-p_0}\|\sqrt{{\lag v\rag}}w_0h^{\varepsilon}\|^2\label{e-gth}
\end{align}
for some $p_0>1$, which is also beyond control. To overcome such a difficulty, following the technique developed in \cite{DYZ-2012, DYZ-2013}, we introduce the time-dependent exponential weight function $\exp\left(\frac{(1+|v|^2)}{8RT_c\ln(\mathrm{e}+t)}\right)$ while performing the energy estimates, because this weight function gives the time degenerate damping
$$
\frac{1}{8RT_c(\mathrm{e}+t)\big(\ln(\mathrm{e}+t)\big)^2}\|{\lag v\rag}w_0h^{\varepsilon}\|^2.
$$
From which and thanks to the good temporal-decay of solutions to the Euler-Maxwell system \eqref{EM}, we can then control \eqref{e-gth} suitably.
\end{itemize}

As to the second problem, although generally speaking, we can only deduce that the coefficients $[F_n, E_n, B_n]$ enjoy the polynomial temporal growth for $1\leq n\leq 2k-1$, by employing the nice temporal-decay properties of solutions to the Euler-Maxwell system \eqref{EM}, we can kill the above mentioned time growth via corresponding power of $\varepsilon$ and prove that the Hilbert expansion of the VML system \eqref{main1} is valid for $t\leq\vps^{-1/3}$. It is worth to pointing out that, to guarantee that $\|f^{\varepsilon}\|^2_{H^{|\alpha|}}$ is integrable over $t\leq\vps^{-1/3}$, compared with the design of the energy functional $\mathcal{E}(t)$ to the Landau equation \eqref{LE} defined by \eqref{eg-le}, an additional weight $\vps^{1/3}$ should be added to the norms $\|h^{\varepsilon}\|^2_{H^{|\alpha|}}$ to kill the time growth. With the well-designed energy functional $\mathcal{E}(t)$ given by \eqref{eg-vml}, we arrive at a closed {\it a priori} estimates by an interplay of the energy estimates $\|f^{\varepsilon}\|^2_{H^2} $ and $\|wh^{\varepsilon}\|^2_{H^2} $ around both the local Maxwellian $\FM$ and the global Maxwellian $\mu$, respectively.

Our paper is organized as follows. In Section \ref{pr}, we collect some estimates concerning the collision operators $\mathcal{L}_{\mathbf{M}} $, $\mathcal{L}_d $ and $\Gamma_{\mathbf{M}}$ defined in \eqref{lm-def}, \eqref{dL} and \eqref{gm-def}, respectively. Section \ref{H-LD} is devoted to justifying the local validity of the Hilbert expansion \eqref{exp-le} for the Cauchy problem \eqref{LE}, \eqref{L-id} of the Landau equation \eqref{LE}, while the rigorous justification of the global Hilbert expansion \eqref{expan} of the Cauchy problem \eqref{main1}, \eqref{VL-id} of the VML system \eqref{main1} will be given in Section \ref{h-VML}.

\section{Preliminaries }\label{pr}

\setcounter{equation}{0}

In this section, we collect some estimates concerning the collision operators $\mathcal{L}_{\mathbf{M}} $, $\mathcal{L}_d $ and $\Gamma_{\mathbf{M}}$ defined in \eqref{lm-def}, \eqref{dL} and \eqref{gm-def}, respectively.
To do this, we first give the following explicit expressions of $\mathcal{A}_{\mathbf{M}} $, $\mathcal{K}_{\mathbf{M}} $, $\mathcal{A}_d $, $\mathcal{K}_d $ and $\Gamma_{\mathbf{M}}$.

\begin{lemma}\label{AK0d} It holds that
\begin{align}\label{AK0}
\left\{\begin{array}{rll}
\mathcal{A}_{\mathbf{M}}[f]
=&\partial_i\Big[\sigma^{ij}_{\mathbf{M}}\partial_jf
\Big]
-\sigma^{ij}_{\mathbf{M}}\frac{(v_i-u_i)(v_j-u_j)}{4R^2T^2}f+\partial_i\Big[\sigma^{ij}_{\mathbf{M}}\frac{v_j-u_j}{2RT}\Big]f,\\[2mm]
\mathcal{K}_{\mathbf{M}}[f]=&-\mathbf{M}^{-\frac{1}{2}}\partial_i\Big[\mathbf{M} \Big(\phi^{ij}\ast \mathbf{M}^{\frac{1}{2}}\Big(\partial_jf+\frac{v_j-u_j}{2RT}\Big)\Big)\Big],
\end{array}\right.
\end{align}
\begin{align}\label{AKd}
\left\{\begin{array}{rll}
\mathcal{A}_d[h]
=&\partial_i\Big[\sigma^{ij}_{\mathbf{M}}\Big(\frac{-u_j}{ RT}+\frac{T_c-T}{RT_cT}v_j\Big)h
\Big]
-\sigma^{ij}_{\mathbf{M}}\frac{v_i}{ 2RT_c}\Big(\frac{-u_j}{ RT}+\frac{T_c-T}{RT_cT}v_j\Big)h\\[2mm]
&+\partial_i\Big[\sigma^{ij}_{(\mathbf{M}-\mu)}\partial_jh
\Big]
-\sigma^{ij}_{(\mathbf{M}-\mu)}\frac{v_iv_j}{4R^2T_cT}f+\partial_i\Big[\sigma^{ij}_{(\mathbf{M}-\mu)}\frac{v_j}{ 2RT_c}\Big]h,\\[2mm]
\mathcal{K}_d[h]=&\mu^{-\frac{1}{2}}\partial_i\Big[\Big(\frac{u_j}{ RT}\mathbf{M}+\frac{(T-T_c)\mu-T_c(\mathbf{M}-\mu)}{RT_cT}v_j\Big)\Big(\phi^{ij}\ast \mu^{\frac{1}{2}}h\Big)\Big]\\[2mm]
&- \mu^{-\frac{1}{2}}\partial_i\Big[\big(\mathbf{M}-\mu\big) \Big(\phi^{ij}\ast \mathbf{M}^{\frac{1}{2}}\Big(\partial_jh-\frac{v_j}{ 2RT_c}h\Big)\Big)\Big],
\end{array}\right.
\end{align}
and
\begin{align*}%\label{ga0L}
\Gamma_{\mathbf{M}}(f_1,f_2)
=&\partial_i\Big[\Big(\phi^{ij}\ast\mathbf{M}^{\frac{1}{2}}f_1\Big)
\partial_jf_2-\Big(\phi^{ij}\ast\mathbf{M}^{\frac{1}{2}}\partial_jf_1\Big)
f_2
\Big]
\\
&-\Big(\phi^{ij}\ast \frac{v_j-u_j}{2RT}\mathbf{M}^{\frac{1}{2}}f_1\Big)
\partial_jf_2+\Big(\phi^{ij}\ast \frac{v_j-u_j}{2RT}\mathbf{M}^{\frac{1}{2}}\partial_jf_1\Big)
f_2,\nonumber
\end{align*}
where $\partial_i=\partial_{v_i}$ $(i=1,2,3)$ and
$$
\sigma^{ij}_{g}=\int_{{\R}^{3}}\phi^{ij}(v-v')g(v')dv'.
$$
\end{lemma}
\begin{proof}
The proof will be done in a similar way as in \cite[Lemma 1, pp. 395]{Guo-CMP-2002}. For brevity, we only show \eqref{AKd}.
For this purpose, direct calculation gives
%\begin{align*}
%\mathcal{A}_{\mathbf{M}}[f]=&\mathbf{M}^{-\frac{1}{2}}
%\mathcal{C}( \mathbf{M},\mathbf{M}^{\frac{1}{2}}f  )\\
%=&\mathbf{M}^{-\frac{1}{2}}\partial_i\Big[\mathbf{M}^{\frac{1}{2}}\sigma^{ij}_{\mathbf{M}}\Big(\partial_jf
%+\frac{v_j-u_j}{2RT}f\Big)\Big]\\
%=&\partial_i\Big[\sigma^{ij}_{\mathbf{M}}\partial_jf
%\Big]
%-\sigma^{ij}_{\mathbf{M}}\frac{(v_i-u_i)(v_j-u_j)}{4R^2T^2}f+\partial_i\Big[\sigma^{ij}_{\mathbf{M}}\frac{v_j-u_j}{2RT}\Big]f,
%\end{align*}
%
%and
%\begin{align*}
%\mathcal{K}_{\mathbf{M}}[f]=&\mathbf{M}^{-\frac{1}{2}}
%\mathcal{C}( \mathbf{M}^{\frac{1}{2}}f,\mathbf{M} )\\
%=&\mathbf{M}^{-\frac{1}{2}}\partial_i\Big[\Big(\phi^{ij}\ast \mathbf{M}^{\frac{1}{2}}f\Big)\frac{-v_j+u_j}{ RT}\mathbf{M}-\mathbf{M} \Big(\phi^{ij}\ast \mathbf{M}^{\frac{1}{2}}\Big(\partial_jf-\frac{v_j-u_j}{2RT}f\Big)\Big)\Big]\\
%=&-\mathbf{M}^{-\frac{1}{2}}\partial_i\Big[\mathbf{M} \Big(\phi^{ij}\ast \mathbf{M}^{\frac{1}{2}}\Big(\partial_jf+\frac{v_j-u_j}{2RT}\Big)\Big)\Big].
%\end{align*}
%Similarly, it follows
\begin{align*}
\mathcal{A}_d[h]=&\mu^{-\frac{1}{2}}
\mathcal{C}( \mathbf{M}-\mu,\mu^{\frac{1}{2}}h  )\\
=&\mu^{-\frac{1}{2}}\partial_i\Big[\sigma^{ij}_{(\mathbf{M}-\mu)}\mu^{\frac{1}{2}}\Big(\partial_jh-\frac{v_j}{ 2RT_c}h\Big)
-\mu^{\frac{1}{2}}h \Big(\phi^{ij}\ast \Big(\frac{-v_j+u_j}{ RT}\mathbf{M}+\frac{v_j}{ RT_c}\mu\Big)\Big)\Big]\\
=&\mu^{-\frac{1}{2}}\partial_i\Big[\mu^{\frac{1}{2}}h\Big(\phi^{ij}\ast \Big(\frac{-u_j}{ RT}+\frac{T_c-T}{RT_cT}v_j\Big)\mathbf{M}\Big)\Big]+\mu^{-\frac{1}{2}}\partial_i\Big[\mu^{\frac{1}{2}}\sigma^{ij}_{(\mathbf{M}-\mu)}
\Big(\partial_jh
+\frac{v_j}{ 2RT_c}h\Big)\Big]\\
=&\partial_i\Big[\sigma^{ij}_{\mathbf{M}}\Big(\frac{-u_j}{ RT}+\frac{T_c-T}{RT_cT}v_j\Big)h
\Big]
-\sigma^{ij}_{\mathbf{M}}\frac{v_i}{ 2RT_c}\Big(\frac{-u_j}{ RT}+\frac{T_c-T}{RT_cT}v_j\Big)h\\
&+\partial_i\Big[\sigma^{ij}_{(\mathbf{M}-\mu)}\partial_jh
\Big]
-\sigma^{ij}_{(\mathbf{M}-\mu)}\frac{v_iv_j}{4R^2T_cT}h+\partial_i\Big[\sigma^{ij}_{(\mathbf{M}-\mu)}\frac{v_j}{ 2RT_c}\Big]h
\end{align*}
and
\begin{align*}
\mathcal{K}_d[h]=&\mu^{-\frac{1}{2}}
\mathcal{C}( \mu^{\frac{1}{2}}h, \mathbf{M}-\mu )\\
=&\mu^{-\frac{1}{2}}\partial_i\Big[\Big(\phi^{ij}\ast \mu^{\frac{1}{2}}h\Big)\Big(\frac{-v_j+u_j}{ RT}\mathbf{M}+\frac{v_j}{ RT_c}\mu\Big)\\
&-\big(\mathbf{M}-\mu\big) \Big(\phi^{ij}\ast \mathbf{M}^{\frac{1}{2}}\Big(\partial_jh-\frac{v_j}{ 2RT_c}h\Big)\Big)\Big]\\
=&\mu^{-\frac{1}{2}}\partial_i\Big[\Big(\frac{u_j}{ RT}\mathbf{M}+\frac{(T-T_c)\mu-T_c(\mathbf{M}-\mu)}{RT_cT}v_j\Big)\Big(\phi^{ij}\ast \mu^{\frac{1}{2}}h\Big)\Big]\nonumber\\
&- \mu^{-\frac{1}{2}}\partial_i\Big[\big(\mathbf{M}-\mu\big) \Big(\phi^{ij}\ast \mathbf{M}^{\frac{1}{2}}\Big(\partial_jh-\frac{v_j}{ 2RT_c}h\Big)\Big)\Big].
\end{align*}
%As for the nonlinear operator $\Gamma_{\mathbf{M}}(f_1,f_2)$, one has
%\begin{align*}
%\Gamma_{\mathbf{M}}(f_1,f_2)
%=&\mathbf{M}^{-\frac{1}{2}}
%\mathcal{C}( \mathbf{M}^{\frac{1}{2}}f_1,\mathbf{M}^{\frac{1}{2}}f_2  )\\
%=&\partial_i\Big[\Big(\phi^{ij}\ast\mathbf{M}^{\frac{1}{2}}f_1\Big)
%\Big(\partial_jf_2-\frac{v_j-u_j}{2RT}f_2\Big)\\&-f_2 \Big(\phi^{ij}\ast\mathbf{M}^{\frac{1}{2}} \Big(\partial_jf_1-\frac{v_j-u_j}{2RT}f_1\Big)\Big)\Big]\\
%=&\partial_i\Big[\Big(\phi^{ij}\ast\mathbf{M}^{\frac{1}{2}}f_1\Big)
%\Big(\partial_jf_2-\frac{v_j-u_j}{2RT}f_2\Big)
%-f_2 \Big(\phi^{ij}\ast\mathbf{M}^{\frac{1}{2}} \Big(\partial_jf_1-\frac{v_j-u_j}{2RT}f_1\Big)\Big)\Big]\\
%&-\frac{v_i-u_i}{2RT}\Big[\Big(\phi^{ij}\ast\mathbf{M}^{\frac{1}{2}}f_1\Big)
%\Big(\partial_jf_2-\frac{v_j-u_j}{2RT}f_2\Big)
%\\&\qquad\qquad-f_2 \Big(\phi^{ij}\ast\mathbf{M}^{\frac{1}{2}} \Big(\partial_jf_1-\frac{v_j-u_j}{2RT}f_1\Big)\Big)\Big]\\
%=&\partial_i\Big[\Big(\phi^{ij}\ast\mathbf{M}^{\frac{1}{2}}f_1\Big)
%\partial_jf_2-\Big(\phi^{ij}\ast\mathbf{M}^{\frac{1}{2}}\partial_jf_1\Big)
%f_2
%\Big]
%-\Big(\phi^{ij}\ast \mathbf{M}^{\frac{1}{2}}\frac{v_i-u_i}{2RT}f_1\Big)
%\partial_jf_2\\
%&+\Big(\phi^{ij}\ast \frac{v_i-u_i}{2RT}\mathbf{M}^{\frac{1}{2}}\partial_jf_1\Big)
%f_2.
%\end{align*}
This is exactly \eqref{AKd} and the proof of Lemma \ref{AK0d} is complete.
\end{proof}

The following lemma is quoted from \cite[Corollary 1, pp. 399]{Guo-CMP-2002}, which states the lower bound of $D-$norm defined by \eqref{D-norm}.
\begin{lemma}\label{lower norm}
Let $D-$norm be defined as \eqref{D-norm}, then
there exists $C>0$ such that
\begin{equation}
|g|_{D}^2\geq C\left\{\left|(1+|v|)^{-\frac{3}{2}}\{{\bf P}_v\partial_jg\}\right|_{L^2}^2
+\left|(1+|v|)^{-\frac{1}{2}}\{({\bf I}-{\bf P}_v)\partial_jg\}\right|_{L^2}^2
+\left|(1+|v|)^{-\frac{1}{2}}g\right|_{L^2}^2\right\},\notag
\end{equation}
where ${\bf P}_v$ is the projection defined by
$$
{\bf P}_vh_j=\sum{h_kv_k}\frac{v_j}{|v|^2}, \ \ 1\leq j\leq3,
$$
for any vector-valued function $h(v)=[h_1(v),h_2(v),h_3(v)]$.
\end{lemma}
We now give the coercivity estimates of the linearized operator $\CL$ and $\CL_\FM$ and the upper bound of the difference operator $\CL_d.$
\begin{lemma}\label{L-co-lem} The linearized operator $\CL$ and $\CL_\FM$ satisfy the following coercivity estimates:
\begin{itemize}
\item [(i)] Let $\CL$ be given by \eqref{L-def}, it holds that
\begin{equation}
\big\lag \mathcal{L}h, h\big\rag\gtrsim |({\bf I}-{\bf P})h|_{D}^2.\label{coL}
\end{equation}
\item [(ii)] Let $\FM=\FM_{[\rho,u,T]},$ where $[\rho,u,T]$ is determined by either Proposition \ref{e-loc} or Proposition \ref{em-ex-lem}, then there exists a $\de>0$ such that
\begin{align}
\big(\mathcal{L}_{\mathbf{M}}f, f\big)\geq\de \big| ({\bf I}-{\bf P}_{\mathbf{M}})f\big|_{D}^2. \label{coL0}
\end{align}
Moreover, for $|\beta|>0$, $\ell\geq 0$ and $q\in [0,1)$, it holds that
\begin{align}\label{wLLM}
&\big(\partial^{\alpha}_{\beta}\mathcal{L}_{\mathbf{M}}f, \lag v\rag^{2(\ell-|\beta|)}\mathbf{M}^{-q} \partial_\beta^{\alpha}f\big)\nonumber\\
\geq& \de\big|\lag v\rag^{(\ell-|\beta|)}\mathbf{M}^{-q/2}\partial^{\alpha}_{\beta}f\big|_{D}^2-\big(\eta+C\bar{\epsilon}\big)\sum_{\al'\leq \alpha}\sum_{|\beta'|=|\beta|}\big|\lag v\rag^{(\ell-|\beta'|)}\mathbf{M}^{-q/2}\partial^{\al'}_{\beta'}f\big|_{D}^2\\
&-C(\eta)\sum_{\al'\leq \alpha}\sum_{|\beta'|<|\beta|}\big|\lag v\rag^{(\ell-|\beta'|)}\mathbf{M}^{-q/2}\partial^{\al'}_{\beta'}f\big|^2_{D},\nonumber
\end{align}
where $\bar{\epsilon}=\max\{\epsilon_0,\epsilon_1\}$,
and for $|\beta|=0$, $\ell\geq 0$ and $q\in [0,1)$, it holds that
\begin{align}\label{wLLM0}
&\big(\partial^{\alpha}\mathcal{L}_{\mathbf{M}}f, \lag v\rag^{2\ell}\mathbf{M}^{-q} \partial^{\alpha}f\big)\nonumber\\
\geq& \de\big|\lag v\rag^{\ell}\mathbf{M}^{-q/2}\partial^{\alpha}f\big|_{D}^2
-C\bar{\epsilon}\sum_{\al'< \alpha}{\bf 1}_{|\alpha|\geq1}\big|\lag v\rag^{\ell}\mathbf{M}^{-q/2}\partial^{\al'}h\big|_{D}^2\\
&-C(\eta)\sum_{\al'\leq \alpha}\big|\lag v\rag^{\ell}\mathbf{M}^{-q/2}\partial^{\al'}f\big|^2_{L^2(B_C(\eta))}.\nonumber
\end{align}
Here $\eta>0$ is some small constant,  $B_C(\eta)$ is the ball in ${\mathbb R}^3$ with radius $C(\eta)$.\\
Furthermore, for $|\alpha|>0$, there exists a polynomial $\CU(\cdot)$ with $\CU(0)=0$ such that
\begin{align}\label{coLh}
\big\langle\partial^{\alpha}\mathcal{L}_{\mathbf{M}}f, \partial^{\alpha}f\big\rangle\geq&\frac{3\de}{4}\|\partial^{\alpha}({\bf I}-{\bf P}_{\mathbf{M}})f\|_D^2-C\CU(\|\nabla_x[\rho,u,T]\|_{W^{|\alpha|-1,\infty}})\\
&\times\Big( \frac{1}{\varepsilon}\|({\bf I}-{\bf P}_{\mathbf{M}})f\|_{H^{|\alpha|-1}_D}^2+\varepsilon\|f\|^2_{H^{|\alpha|}}\Big).\nonumber
\end{align}
\end{itemize}
\end{lemma}
\begin{proof} \eqref{coL} is directly from \cite[Lemma 5, pp. 400]{Guo-CMP-2002}. Since $[\rho,u,T]$ satisfies either \eqref{e-sol-es} or \eqref{em-decay}, \eqref{coL0} follows again from a similar argument as in the proof of \cite[Lemma 5, pp. 400]{Guo-CMP-2002}.

We now turn to prove \eqref{wLLM}. Recalling $-\CL_\FM=\mathcal{A}_{\mathbf{M}}+\mathcal{K}_{\mathbf{M}}$ with $\mathcal{A}_{\mathbf{M}}$ and $\mathcal{K}_{\mathbf{M}}$ being given by \eqref{AK0}, direct calculation gives
\begin{align}
&\big(\partial^{\alpha}_{\beta}\mathcal{A}_{\mathbf{M}}f,\lag v\rag^{2(\ell-|\beta|)}\mathbf{M}^{-q} \partial^{\alpha}_\beta f\big)\notag\\
=&-\big|\lag v\rag^{(\ell-|\beta|)}\mathbf{M}^{-q/2}\partial^{\alpha}_{\beta}f\big|_{\tilde{D}}^2\notag\\
&-{\bf 1}_{\al+\beta>0}\sum\limits_{0<\al'+\beta'\leq\al+\beta} C_\al^{\al'}C_\beta^{\beta'}
\big(\pa^{\al'}_{\beta'}\si^{ij}_\FM\pa^{\al-\al'}_{\beta-\beta'}\pa_jf,\lag v\rag^{2(\ell-|\beta|)}\mathbf{M}^{-q} \partial^{\alpha}_\beta\pa_if\big)\notag\\
&-\sum\limits_{\al'\leq\al,\beta'\leq\beta}C_\al^{\al'}C_\beta^{\beta'}
\big(\pa_i\left(\lag v\rag^{(\ell-|\beta|)}\mathbf{M}^{-q/2}\right)\pa^{\al'}_{\beta'}\si^{ij}_\FM\pa^{\al-\al'}_{\beta-\beta'}\pa_jf, \partial^{\alpha}_\beta\pa_if\big)\notag\\
&-{\bf 1}_{\al+\beta>0}\sum\limits_{0<\al'+\beta'\leq\al+\beta}C_\al^{\al'}C_\beta^{\beta'}
\big(\pa^{\al'}_{\beta'}\left(\sigma^{ij}_{\mathbf{M}}\frac{(v_i-u_i)(v_j-u_j)}{4R^2T^2}\right)\pa^{\al-\al'}_{\beta-\beta'}\pa_jf,\lag v\rag^{2(\ell-|\beta|)}\mathbf{M}^{-q} \partial^{\alpha}_\beta\pa_if\big)\notag\\
&+\sum\limits_{\al'+\beta'\leq\al+\beta}C_\al^{\al'}C_\beta^{\beta'}
\big(\pa^{\al'}_{\beta'}\left(\sigma^{ij}_{\mathbf{M}}\frac{v_j-u_j}{2RT}\right)\pa^{\al-\al'}_{\beta-\beta'}\pa_jf,\lag v\rag^{2(\ell-|\beta|)}\mathbf{M}^{-q} \partial^{\alpha}_\beta\pa_if\big),\label{FA-ep}
\end{align}
where
\begin{align}
\big|\lag v\rag^{(\ell-|\beta|)}\mathbf{M}^{-q/2}\partial^{\alpha}_{\beta}f\big|_{\tilde{D}}^2=&
\int_{{\mathbb R}^3}\lag v\rag^{(\ell-|\beta|)}\mathbf{M}^{-q/2}\sigma_\FM^{ij}\partial_{v_i} f\partial_{v_j}f\, dv
\notag\\&+\frac{1}{4R^2T^2}\int_{{\mathbb R}^3}\lag v\rag^{(\ell-|\beta|)}\mathbf{M}^{-q/2}\sigma_\FM^{ij}(v_i-u_i)(v_j-u_j)|f|^2\, dv.\notag
\end{align}

Noting that $[\rho,u,T]$ satisfies either \eqref{tt0} or \eqref{tt1}, one gets by denoting $\bar{\epsilon}=\max\{\epsilon_0,\epsilon_1\}$ that
\begin{align}
\big|\lag v\rag^{(\ell-|\beta|)}\mathbf{M}^{-q/2}\partial^{\alpha}_{\beta}f\big|_{\tilde{D}}^2\geq&
\big|\lag v\rag^{(\ell-|\beta|)}\mathbf{M}^{-q/2}\partial^{\alpha}_{\beta}f\big|_{D}^2-C\bar{\eps}\big|\lag v\rag^{(\ell-|\beta|)}\mathbf{M}^{-q/2}\partial^{\alpha}_{\beta}f\big|_{D}^2.\notag
\end{align}

Next, by using \eqref{tt0} and \eqref{tt1} and by performing the similar calculation as in the proof of \cite[Lemma 8, pp. 319]{Strain-Guo-ARMA-2008}, we see that the rest terms in the right hand side of \eqref{FA-ep}
can be dominated by
\begin{align}
\big(\eta+C\bar{\epsilon}\big)\sum_{\al'\leq \alpha}\sum_{|\beta'|=|\beta|}\big|\lag v\rag^{(\ell-|\beta'|)}\mathbf{M}^{-q/2}\partial^{\al'}_{\beta'}f\big|_{D}^2+C(\eta)\sum_{\al'\leq \alpha}\sum_{|\beta'|<|\beta|}\big|\lag v\rag^{(\ell-|\beta'|)}\mathbf{M}^{-q/2}\partial^{\al'}_{\beta'}f\big|^2_{D},\notag
\end{align}
where $0<\eta\ll1.$

Now for the estimates involving $\CK_\FM$, we first have from \eqref{AK0} that
\begin{align}
&\big(\partial^{\alpha}_{\beta}\CK_{\mathbf{M}}f,\lag v\rag^{2(\ell-|\beta|)}\mathbf{M}^{-q} \partial^{\alpha}_\beta f\big)\notag\\
=&-\left(\partial_i\left[\lag v\rag^{2(\ell-|\beta|)}\mathbf{M}^{-q} \partial^{\alpha}_{\beta}\left\{\mathbf{M}^{\frac{1}{2}} \Big(\phi^{ij}\ast \mathbf{M}^{\frac{1}{2}}\Big(\partial_jf+\frac{v_j-u_j}{2RT}f\Big)\Big)\right\}\right]
,\partial^{\alpha}_\beta f\right)\notag\\
&+\left(\partial_i\left[\lag v\rag^{2(\ell-|\beta|)}\mathbf{M}^{-q} \right]\partial^{\alpha}_{\beta}\left\{\mathbf{M}^{\frac{1}{2}} \Big(\phi^{ij}\ast \mathbf{M}^{\frac{1}{2}}\Big(\partial_jf+\frac{v_j-u_j}{2RT}f\Big)\Big)\right\}
,\partial^{\alpha}_\beta f\right)\notag\\
&+\left(\lag v\rag^{2(\ell-|\beta|)}\mathbf{M}^{-q} \partial^{\alpha}_{\beta}\left\{\frac{v_i-u_i}{2RT}\mathbf{M}^{\frac{1}{2}} \Big(\phi^{ij}\ast \mathbf{M}^{\frac{1}{2}}\Big(\partial_jf+\frac{v_j-u_j}{2RT}f\Big)\Big)\right\}
,\partial^{\alpha}_\beta f\right).\notag
\end{align}

Since for any finite $k\in\mathbb{N}$ and any multi-indices $\alpha, \alpha', \beta',$ it follows from \eqref{e-sol-es} and \eqref{em-decay} that
$$
\na_v^k\left\{\lag v\rag^{(\ell-|\beta|)}\mathbf{M}^{-q/2}\right\}\FM^{1/2}\leq C\FM^{\frac{1-q}{4}},\
\pa_{\beta'}^{\al'}\left\{\frac{v_i-u_i}{2RT}\mathbf{M}^{\frac{1}{2}}\right\}\leq C\FM^{\frac{1}{4}},
$$
one has by using the same argument as for obtaining the estimates on $K$ in \cite[Lemma 8, pp. 319]{Strain-Guo-ARMA-2008} that
\begin{align}
\big(\partial^{\alpha}_{\beta}&\CK_{\mathbf{M}}f,\lag v\rag^{2(\ell-|\beta|)}\mathbf{M}^{-q} \partial^{\alpha}_\beta f\big)\notag\\
\leq&\left\{\eta\big|\lag v\rag^{(\ell-|\beta|)}\partial^{\alpha}_{\beta}f\big|_{D}+\big|\lag v\rag^{(\ell-|\beta|)}f\big|_{L^2(B_C(\eta))}\right\}\big|\lag v\rag^{(\ell-|\beta|)}\mathbf{M}^{-q/2}\partial^{\alpha}_{\beta}f\big|_{D}.\notag
\end{align}

Putting the above estimates for $\CA_\FM$ and $\CK_\FM$ together, we see that \eqref{wLLM} is true.
\eqref{wLLM0} can be proved similarly as above.

As for \eqref{coLh}, applying \eqref{coL0}, we get
\begin{align*}
&\big\langle \partial^{\alpha}\mathcal{L}_{\mathbf{M}}f,\partial^{\alpha}_\beta f\big\rangle\\
=&\big\langle \mathcal{L}_{\mathbf{M}}\partial^{\alpha}({\bf I}-{\bf P}_{\mathbf{M}})f,({\bf I}-{\bf P}_{\mathbf{M}})\partial^{\alpha}f\big\rangle
+\big\langle \Lbrack\partial^{\alpha},\mathcal{L}_{\mathbf{M}}\Rbrack ({\bf I}-{\bf P}_{\mathbf{M}})f,\partial^{\alpha}f\big\rangle\\
=&\big\langle \mathcal{L}_{\mathbf{M}}\partial^{\alpha}({\bf I}-{\bf P}_{\mathbf{M}})f,\partial^{\alpha}({\bf I}-{\bf P}_{\mathbf{M}})f+\Lbrack\partial^{\alpha},{{\bf P}_{\mathbf{M}}}\Rbrack f\big\rangle
+\big\langle \Lbrack\partial^{\alpha},\mathcal{L}_{\mathbf{M}}\Rbrack ({\bf I}-{\bf P}_{\mathbf{M}})f,\partial^{\alpha}f\big\rangle\\
\geq&\delta\|\partial^{\alpha}({\bf I}-{\bf P}_{\mathbf{M}})f\|_D^2+\big\langle \mathcal{L}_{\mathbf{M}}\partial^{\alpha}({\bf I}-{\bf P}_{\mathbf{M}})f,\Lbrack\partial^{\alpha},{{\bf P}_{\mathbf{M}}}\Rbrack f\big\rangle\\&+\big\langle \Lbrack \partial^{\alpha},\mathcal{L}_{\mathbf{M}}\Rbrack ({\bf I}-{\bf P}_{\mathbf{M}})f,\partial^{\alpha}f\big\rangle\\
\geq&\delta\|\partial^{\alpha}({\bf I}-{\bf P}_{\mathbf{M}})f\|_D^2-C\|\Lbrack\partial^{\alpha},{{\bf P}_{\mathbf{M}}}\Rbrack f\|_D^2+\big\langle \Lbrack \partial^{\alpha},\mathcal{L}_{\mathbf{M}}\Rbrack ({\bf I}-{\bf P}_{\mathbf{M}})f,\partial^{\alpha}f\big\rangle.
\end{align*}
Here $\Lbrack A,B\Rbrack=AB-BA$ is the commutator of $A$ and $B$.

Noting that if $\al>0$, $\partial^{\alpha}$ acts on the local Maxwellian $\mathbf{M}$ at least once in the commutators $\Lbrack\partial^{\alpha},{{\bf P}_{\mathbf{M}}}\Rbrack$, $\Lbrack \partial^{\alpha},\mathcal{L}_{\mathbf{M}}\Rbrack$ and by the definition of the operator ${{\bf P}_{\mathbf{M}}}$, we see that there exists a polynomial $\CU(\cdot)$ with $\CU(0)=0$ such that
\begin{align*}
\|\Lbrack\partial^{\alpha},{{\bf P}_{\mathbf{M}}}\Rbrack f\|_D^2\lesssim&
\CU(\|\nabla_x[\rho,u,T]\|_{W^{|\alpha|-1,\infty}})\|{{\bf P}_{\mathbf{M}}}f\|^2_{H^{|\alpha|-1}}\\
\lesssim&
\CU(\|\nabla_x[\rho,u,T]\|_{W^{|\alpha|-1,\infty}})
\|f\|^2_{H^{|\alpha|-1}}.
\end{align*}

Moreover, by using \eqref{AK0} and Cauchy-Schwarz's inequality with $\vps>0$, one gets that
\begin{align*}
&\big|\big\langle \Lbrack \partial^{\alpha},\mathcal{L}_{\mathbf{M}}\Rbrack ({\bf I}-{\bf P}_{\mathbf{M}})f,\partial^{\alpha}f\big\rangle\big|\\
\lesssim &
\CU(\|\nabla_x[\rho,u,T]\|_{W^{|\alpha|-1,\infty}})\|({\bf I}-{\bf P}_{\mathbf{M}})f\|_{H^{|\alpha|-1}_D}\|\partial^{\alpha}f\|_D\\
\lesssim&
\CU(\|\nabla_x[\rho,u,T]\|_{W^{|\alpha|-1,\infty}})\|({\bf I}-{\bf P}_{\mathbf{M}})f\|_{H^{|\alpha|-1}_D}
\big(\|\partial^{\alpha}({\bf I}-{\bf P}_{\mathbf{M}})f\|_D+\|\partial^{\alpha}{{\bf P}_{\mathbf{M}}}f\|_D\big)\\
\lesssim& \bar{\eps}\|\partial^{\alpha}({\bf I}-{\bf P}_{\mathbf{M}})f\|_D^2+
\CU(\|\nabla_x[\rho,u,T]\|_{W^{|\alpha|-1,\infty}})
\Big( \frac{1}{\varepsilon}\|({\bf I}-{\bf P}_{\mathbf{M}})f\|_{H^{|\alpha|-1}_D}^2+\varepsilon\|f\|^2_{H^{|\alpha|}}\Big).
\end{align*}

Putting the above estimates together gives \eqref{coLh}.
The proof of Lemma \ref{L-co-lem} is finished.
\end{proof}

Recalling that the nonlinear collision operators $\Gamma$ and $\Gamma_{\mathbf{M}}$ are defined as
\begin{align*}
\Gamma(f,g)=\mu^{-\frac{1}{2}} \mathcal{C}( \mu^{\frac{1}{2}} f, \mu^{\frac{1}{2}} g)
\end{align*}
and
\begin{align*}
\Gamma_{\mathbf{M}}(f,g)=&\mathbf{M}^{-\frac{1}{2}} \mathcal{C}( \mathbf{M}^{\frac{1}{2}} f, \mathbf{M}^{\frac{1}{2}} g)
\end{align*}
respectively, we now state the trilinear estimates on these nonlinear collision operators. The first results is concerned with the trilinear estimates of $\Gamma$ and $\Gamma_{\mathbf{M}}$  without velocity weight.
\begin{lemma}\label{GL0} It holds that
\begin{align}\label{GL}
\big|\big(\partial^{\alpha}\Gamma(f,g),\partial^{\alpha} h\big)\big|
\lesssim& \sum_{\al'+\tilde{\alpha}= \alpha}\Big(\big|\partial^{\al'}f\big|_{L^2}\big|\partial^{\tilde{\alpha}}g\big|_D+\big|\partial^{\al'}f\big|_D
\big|\partial^{\tilde{\alpha}}g\big|_{L^2}\Big)\big|\partial^{\alpha}h\big|_D
\end{align}
and
\begin{align}\label{GLM}
&\big|\big(\partial^{\alpha}\Gamma_{\mathbf{M}}(f,g), \partial^{\alpha}h\big)\big|\nonumber\\
\lesssim& \sum_{\al'+\tilde{\alpha}= \alpha}\Big(\big|\partial^{\al'}f\big|_{L^2}\big|\partial^{\tilde{\alpha}}g\big|_D+\big|\partial^{\al'}f\big|_D
\big|\partial^{\tilde{\alpha}}g\big|_{L^2}\Big)\big|\partial^{\alpha}h\big|_D\\
&+\sum_{\al'+\tilde{\alpha}+\alpha'= \alpha} {\bf 1}_{|\alpha'|\geq1}\CU(\|\nabla_x[\rho,u,T]\|_{W^{s,\infty}})\Big(\big|\partial^{\al'}f\big|_{L^2}\big|\partial^{\tilde{\alpha}}g\big|_D+\big|\partial^{\al'}f\big|_D
\big|\partial^{\tilde{\alpha}}g\big|_{L^2}\Big)\big|\partial^{\alpha}h\big|_D.\nonumber
\end{align}
Here and in the sequel, ${\bf 1}_{\Omega}$ is the usual indicator function of the set $\Omega$.
\end{lemma}
\begin{proof} \cite[Lemma 3, pp. 406]{Guo-CMP-2002} with $\ta=0$ and $\beta=0$ gives \eqref{GL}. The proof of \eqref{GLM} is the same as that of \eqref{GGGm} below, we omit the details for brevity. This ends the proof of Lemma \ref{GL0}.
\end{proof}

The following lemma is devoted to the trilinear estimates of the nonlinear operator $\Gamma$ and $\Gamma_{\mathbf{M}}$  with velocity weight and velocity derivatives.
\begin{lemma}\label{wggm}
\begin{itemize}
\item [(i)] For any $\ell\geq0$, it holds that
\begin{align}\label{wGG}
&\big|\big(\partial^{\alpha}\Gamma(f,g), w^2_{|\alpha|}\partial^{\alpha}h\big)\big|\\
\lesssim&  \sum_{\al'\leq\alpha} \Big(\big|\lag v\rag^{\ell-|\al|}\partial^{\al'}f\big|_{L^2}|w_{|\alpha|}\partial^{\al-\al'}g|_D
+\big|\lag v\rag^{\ell-|\al|}\partial^{\al'}f\big|_D\big|w_{|\alpha|}\partial^{\al-\al'}g|_{L^2}\Big)
\big|w_{|\alpha|}\partial^{\alpha}h|_D.\nonumber
\end{align}
\item [(ii)] Let $q\in[0,1)$ and $\ell\geq0$,
it holds that
\begin{align}\label{GGGm}
&\big|\big(\partial^{\alpha}_{\beta}\Gamma_{\mathbf{M}}(f,g), {\lag v\rag}^{2(\ell-|\beta|)}\mathbf{M}^{-q}\partial^{\alpha}_{\beta}h\big)\big|\\
\lesssim& \sum_{\al'+\al''= \alpha\atop{\beta'+\beta''\leq \beta}} \Big(\big|{\lag v\rag}^{(\ell-|\beta'|)}\partial^{\al'}_{\beta'}f\big|_{L^2}\big|{\lag v\rag}^{(\ell-|\beta''|)}\mathbf{M}^{-q/2}\partial^{\al''}_{\beta''}g\big|_D\nonumber\\
&\qquad\qquad+\big|{\lag v\rag}^{(\ell-|\beta'|)}\partial^{\al'}_{\beta'}f\big|_D
\big|{\lag v\rag}^{(\ell-|\tilde{\beta}|)}\mathbf{M}^{-q/2}\partial^{\tilde{\alpha}}_{\tilde{\beta}}g\big|_{L^2}\Big)\big|{\lag v\rag}^{(\ell-|\beta|)}\mathbf{M}^{-q/2}\partial^{\alpha}_{\beta}h\big|_D\nonumber\\
&+\sum_{|\al'+\al''|\leq|\alpha|-m-1\atop{\beta'+\beta''\leq \beta}}{\bf 1}_{m\geq0}\CU(\|\nabla_x[\rho,u,T]\|_{W^{m,\infty}})\Big(\big|{\lag v\rag}^{(\ell-|\beta'|)}\partial^{\al'}_{\beta'}f\big|_{L^2}\big|{\lag v\rag}^{(\ell-|\beta''|)}\mathbf{M}^{-q/2}\partial^{\al''}_{\beta''}g\big|_D\notag\\&\qquad\qquad+\big|{\lag v\rag}^{(\ell-|\beta'|)}\partial^{\al'}_{\beta'}f\big|_D
\big|{\lag v\rag}^{(\ell-|\beta''|)}\mathbf{M}^{-q/2}\partial^{\al''}_{\beta''}g\big|_{L^2}\Big)\big|{\lag v\rag}^{(\ell-|\beta|)}\mathbf{M}^{-q/2}\partial^{\alpha}_{\beta}h\big|_D.\nonumber
\end{align}
\end{itemize}
\end{lemma}
\begin{proof} \eqref{wGG} directly follows from \cite[Lemma 10, pp. 327]{Strain-Guo-ARMA-2008}.
We only prove \eqref{GGGm} here. To do this, we first denote $W={\lag v\rag}^{(\ell-|\beta|)}\mathbf{M}^{-q/2}$
for simplicity, then write
\begin{align}%\label{s-G6}
\big|\big(\partial^{\alpha}_{\beta}&\Gamma_{\mathbf{M}}(f,g),W^2\partial^{\alpha}_{\beta}h\big)\big|=\sum C_\al^{\al',\al''}C_\beta^{\beta'}G_{\al',\al'',\beta'},\notag
\end{align}
where $G_{\al',\al'',\beta'}=\sum\limits_{i=1}^6G_{\al',\al'',\beta'}^i$ with
\begin{align}%\label{G1}
G_{\al',\al'',\beta'}^1=-\left(W^2\left(\phi^{ij}\ast\pa_{\beta'}\left\{\pa^{\al''}
\left\{\mathbf{M}^{\frac{1}{2}}\right\}\pa^{\al'-\al''}f\right\}\right)
\partial_j\pa_{\beta-\beta'}^{\al-\al'}g,\partial_i\partial^{\alpha}_{\beta}h\right),\notag
\end{align}
\begin{align}%\label{G2}
G_{\al',\al'',\beta'}^2=\left(W^2\left(\phi^{ij}\ast\pa_{\beta'}\left\{\pa^{\al''}
\left\{\mathbf{M}^{\frac{1}{2}}\right\}\pa^{\al'-\al''}\partial_jf\right\}\right)
\pa_{\beta-\beta'}^{\al-\al'}g,\partial_i\partial^{\alpha}_{\beta}h\right),\notag
\end{align}
\begin{align}%\label{G3}
G_{\al',\al'',\beta'}^3=-\left(W^2\left(\phi^{ij}\ast\pa_{\beta'}
\left\{\pa^{\al''}\left\{\mathbf{M}^{\frac{1}{2}}\frac{v_i-u_i}{2RT}\right\}\pa^{\al'-\al''}f\right\}\right)
\pa_{\beta-\beta'}^{\al-\al'}\partial_jg,\partial^{\alpha}_{\beta}h\right),\notag
\end{align}
\begin{align}%\label{G4}
G_{\al',\al'',\beta'}^4=\left(W^2\left(\phi^{ij}\ast\pa_{\beta'}
\left\{\pa^{\al''}\left\{\mathbf{M}^{\frac{1}{2}}\frac{v_i-u_i}{2RT}\right\}\pa^{\al'-\al''}\pa_jf\right\}\right)
\pa_{\beta-\beta'}^{\al-\al'}g,\partial^{\alpha}_{\beta}h\right),\notag
\end{align}
\begin{align}%\label{G5}
G_{\al',\al'',\beta'}^5=-\left(\pa_iW^2\left(\phi^{ij}\ast\pa_{\beta'}\left\{\pa^{\al''}
\left\{\mathbf{M}^{\frac{1}{2}}\right\}\pa^{\al'-\al''}f\right\}\right)
\partial_j\pa_{\beta-\beta'}^{\al-\al'}g,\partial^{\alpha}_{\beta}h\right),\notag
\end{align}
\begin{align}%\label{G6}
G_{\al',\al'',\beta'}^6=\left(\partial_iW^2\left(\phi^{ij}\ast\pa_{\beta'}\left\{\pa^{\al''}
\left\{\mathbf{M}^{\frac{1}{2}}\right\}\pa^{\al'-\al''}\partial_jf\right\}\right)
\pa_{\beta-\beta'}^{\al-\al'}g,\partial^{\alpha}_{\beta}h\right).\notag
\end{align}

If $\al''=0$, since it follows from either \eqref{tt0} or \eqref{tt1} that
$$
\rho\sim1,\ u\sim0,\ T\sim T_c,
$$
then performing the same calculation as for proving \cite[Lemma 10, pp. 327]{Strain-Guo-ARMA-2008}, one sees that all of $G_{\al',\al'',\beta'}^i$$(1\leq i\leq 6)$ can be controlled by the first part of the right hand side of \eqref{GGGm}.

If $\al''>0$, using \eqref{tt0} or \eqref{tt1} again, we get
$$
\left|\pa^{\al''}
\mathbf{M}^{\frac{1}{2}}\right|\leq C\CU(\|\nabla_x[\rho,u,T]\|_{W^{|\al''|,\infty}})\FM^{\frac{3}{8}}
$$
and
$$
 \left|\pa^{\al''}\left\{\mathbf{M}^{\frac{1}{2}}\frac{v_i-u_i}{2RT}\right\}\right|\leq C\CU(\|\nabla_x[\rho,u,T]\|_{W^{|\al''|,\infty}})\FM^{\frac{3}{8}}.
$$

From the above estimates and by repeating the similar computations as for deriving \cite[Lemma 10, pp. 327]{Strain-Guo-ARMA-2008} again, one sees that all of $G_{\al',\al'',\beta'}^i$$(1\leq i\leq 6)$ can be bounded by the second part of the right hand side of \eqref{GGGm}.
This ends the proof of Lemma \ref{wggm}.
\end{proof}

Our last lemma in this section is concerned with some estimates related to the difference operator $\CL_d.$
\begin{lemma}\label{WGLL0} Let $\rho(t,x), u(t,x)$ and $T(t,x)$ satisfy \eqref{e-sol-es} or \eqref{em-decay}, and suppose both \eqref{tt0} and \eqref{tt1} are valid. Let the weight function $w_i, 0\leq i\leq s,$ be defined in \eqref{tt 01}, then it holds that
\begin{align}
\big|\big\langle\partial^{\alpha}\mathcal{L}_d[h], w_{|\alpha|}^2 \partial^{\alpha}h\big\rangle\big|\lesssim& \bar{\epsilon}\big\|wh\big\|_{H^{|\alpha|}_D}^2.\label{wGLd}
\end{align}
Moreover, if we set $\sqrt{\FM}f=\sqrt{\mu}h$, then it holds that
\begin{align}
\big|\big\langle\partial^{\alpha}\mathcal{L}_dh, \partial^{\alpha}h\big\rangle\big|\lesssim& \bar{\epsilon}\Big(\big\|({\bf I}-{\bf P})h\|_{H^{|\alpha|}_D}^2+\big\|f\|_{H^{|\alpha|}_D}^2\Big)\label{Ldh}
\end{align}
and
\begin{align}
\big(\partial^{\alpha}\mathcal{L}[h], w_{|\alpha|}^2 \partial^{\alpha}h\big)\geq& \de\big|w_{|\alpha|}\partial^{\alpha}h\big|_{D}^2-C\sum_{\al'\leq \alpha}|\partial^{\al'}f|^2_{L^2}.\label{wLL}
\end{align}
\end{lemma}
\begin{proof}
By \eqref{dL}, we rewrite
\begin{align}\label{dL-ga}
-\CL_df=\Ga\left(\frac{\FM-\mu}{\sqrt{\mu}},f\right)+\Ga\left(f,\frac{\FM-\mu}{\sqrt{\mu}}\right).
\end{align}
On the other hand, from \eqref{tt0} and \eqref{tt1}, it follows that
\begin{align}\label{alphax}
 |\partial^{\alpha}\mathbf{M}|\lesssim \bar{\eps}\mu^{5/6}
 \end{align}
holds for any $\al>0.$ Moreover \eqref{tt0} and \eqref{tt1} together with mean value theorem gives
\begin{align}\label{mvt}
|\mathbf{M}-\mu|=\left|[\na_{[\rho,u,T]}\FM](\bar{\rho},\bar{u},\bar{T})\cdot(\rho-1,u,T-T_c)\right|
\lesssim\bar{\eps}\mu^{5/6},
\end{align}
where $\bar{\rho}\in(\min\{\rho,1\},\max\{\rho,1\})$, $\bar{u}_i\in(\min\{u_i,0\},\max\{0,u_i\})$ with $1\leq i\leq3$ and $\bar{T}\in(T_c,T).$

Thus, \eqref{alphax} and \eqref{mvt} yield
\begin{equation}\label{mmb}
\Big|\partial^{\alpha}\Big[\mu^{-\frac{1}{2}}(\mathbf{M}-\mu)\Big]\Big|\lesssim \bar{\eps}\mu^{\frac{1}{3}},
\end{equation}
and then \eqref{wGLd} follows from \eqref{dL-ga}, \eqref{mmb} and \eqref{wGG} in Lemma \ref{wggm}.
Once \eqref{wGLd} is proved, \eqref{Ldh} follows from
\begin{align}
 h=\mu^{-\frac{1}{2}}\mathbf{M}^{\frac{1}{2}}f\label{fh-rl}
\end{align}
and the fact that
\begin{align}\label{pdhf}
\left\|\partial^{\alpha}{\bf P}\left[h^{\varepsilon}\right]\right\|^2_D\lesssim \left\|\partial^{\alpha}{\bf P}\left[h^{\varepsilon}\right]\right\|^2 \lesssim \left\|f^{\varepsilon}\right\|^2_{H^{|\alpha|}}.
\end{align}

It remains now to prove \eqref{wLL}. To this end, we have by applying \cite[Lemma 9, pp. 323]{Strain-Guo-ARMA-2008} with $q=1/(8\ln(\mathrm{e}+t))$ that
\begin{align*}
\big(\partial^{\alpha}\mathcal{L}h, w_{|\alpha|}^2 \partial^{\alpha}h\big)\geq& \de\big|w_{|\alpha|}\partial^{\alpha}h\big|_{D}^2-C(\eta)\big|\partial^{\alpha}h\big|^2_{L^2(B_C(\eta))}.
\end{align*}
This together with \eqref{fh-rl} gives \eqref{wLL}. This completes the proof of Lemma \ref{WGLL0}.
\end{proof}

\section{Hilbert expansion for the Landau equation}\label{H-LD}

\setcounter{equation}{0}

In this section, we intend to justify  the validity of the Hilbert expansion \eqref{exp-le} for the Cauchy problem \eqref{LE} and \eqref{L-id} of the Landau equation \eqref{LE}. To this end, the main task is to determine the coefficients $F_n$ $(0\leq n\leq2k-1)$ and to solve the remainder $F_R^\vps$ via the equation \eqref{m0F1} together with the initial datum
$$
f^\vps(0,x,v)=f^\vps_0(x,v)=\vps^{-k}\FM^{-\frac{1}{2}}\left\{F^\vps_0(x,v)-\sum\limits_{n=0}^{2k-1}\vps^kF_{n,0}(x,v)\right\}.
$$
Note that $h^\vps(0,x,v)=h^\vps_0(x,v)=\mu^{-\frac{1}{2}}\FM_{[\rho_0,u_0,T_0]}^{\frac{1}{2}}f^\vps_0(x,v).$

%%%%%%%%%%%%%%%%%%%%%%%%%%%%%%%%%%%%%%%%%%%%%%%%%%%%%%%%%%%%%%%%%%%%%%%%%%%%%%%%%%
\subsection{The coefficients} \label{sec-coe}
Recall that $F_0=\FM_{[\rho,u,T]}$ with $[\rho,u,T]$ being given by  Proposition \eqref{e-loc},
thus to obtain the estimates of the coefficients $F_n$ for all $0\leq n\leq2k-1$ in the expansion \eqref{exp-le},
it suffices to show the following lemma.
\begin{lemma}\label{Fn-lem}
Under the condition \eqref{ld-Fn-id}, for $q\in(0,1)$ and $\ell\geq0$, there exists $C>0$ such that
\begin{align}
\sum\limits_{\al_0+|\al|+|\beta|\leq N+4k-2n+2}\left|\pa_t^{\al_0}\pa_\bet^\al \left(\frac{F_n}{\sqrt{\FM}}\right)\right|\leq C\lag v\rag^{-\ell+n+|\beta|}\FM^{\frac{q}{2}},\ 1\leq n\leq2k-1.\label{Fn-es}
\end{align}
%where $C>0$ and depends on $t_e$, $N$ and $q$,
\end{lemma}
\begin{proof}
The proof is much similar to that of Lemma \ref{em-Fn-lem}, we thus omit the details for brevity.
\end{proof}

\begin{remark}
The main purpose of introducing the velocity derivatives in \eqref{Fn-es} is to obtain the $L^\infty$ bound of the coefficients $F_n(t,x,v)$ only, which is further used to prove the nonnegativity of the solutions $F^\vps(t,x,v)$, see the details given in \cite[Appendix, pp. 680]{Guo-CPAM-2006}. Similar consideration will also be presented in Lemma \ref{em-Fn-lem} in Section \ref{h-VML}.
\end{remark}
%%%%%%%%%%%%%%%%%%%%%%%%%%%%%%%%%%%%%%%%%%%%%%%%%%%%%%%%%%%%%%%%%%%%%%%%%%%%%%%%%%
%\section{Energy Estimates of $f^{\varepsilon}$} \label{Sec:Energy-f}
%%%%%%%%%%%%%%%%%%%%%%%%%%%%%%%%%%%%%%%%%%%%%%%%%%%%%%%%%%%%%%%%%%%%%%%%%%%%%%%%%%

%\setcounter{equation}{0}

\subsection{The remainder}
Once the coefficients $F_n$$(0\leq n\leq2k-1)$ are given as above, we can now turn to determine the remainder $f^\vps$
through the following Cauchy problem
\begin{eqnarray}\label{fvps-eq}
\left\{\begin{array}{rll}
&\big(\partial_t+v\cdot\nabla_x\big)f^{\varepsilon}
        +\frac{1}{\varepsilon} \mathcal{L}_{\mathbf{M}}[f^{\varepsilon}]
=-f^{\varepsilon}\mathbf{M}^{-\frac{1}{2}}\big(\partial_t+v\cdot\nabla_x\big)\mathbf{M}^{\frac{1}{2}}
+\varepsilon^{k-1}\Gamma_{\mathbf{M}} ( f^{\varepsilon},
f^{\varepsilon} )\\[2mm]
 &\qquad\qquad\qquad\qquad\quad
 +\sum\limits_{i=1}^{2k-1}\varepsilon^{i-1}\Big[\Gamma_{\mathbf{M}}(\mathbf{M}^{-\frac{1}{2}}F_i,f^{\varepsilon})+\Gamma_{\mathbf{M}}(
f^{\varepsilon}, \mathbf{M}^{-\frac{1}{2}} F_i)\Big]+\CP_0,\\[2mm]
&f^\vps(0,x,v)=f^\vps_0(x,v).
 \end{array}\right.
\end{eqnarray}

As we pointed out in the introduction, in order to deduce the estimates of  $f^{\varepsilon}$, it is necessary to consider $h^\vps$ which is connected to $f^\vps$ by
$\sqrt{\FM}f^\vps=\sqrt{\mu}h^\vps$
and satisfies
\begin{eqnarray}\label{hvps-eq}
\left\{\begin{array}{rll}
&\big(\partial_t +v\cdot\nabla_x\big) h^{\varepsilon}+\frac{1}{\varepsilon}\mathcal{L}[ h^{\varepsilon}]
= -\frac{1}{\varepsilon}\mathcal{L}_d[ h^{\varepsilon}]+\varepsilon^{k-1}\Gamma( h^{\varepsilon}, h^{\varepsilon})\\[2mm]
 &\qquad\qquad\qquad\qquad\qquad\qquad+\sum\limits_{i=1}^{2k-1}\varepsilon^{i-1}\Gamma(\mu^{-\frac{1}{2}}F_i, h^{\varepsilon})
+\mathcal{C}( h^{\varepsilon}, \mu^{-\frac{1}{2}} F_i)]+\CP_1,\\[2mm]
&h^\vps(0,x,v)=h^\vps_0(x,v)=\mu^{-\frac{1}{2}}\sqrt{\FM}f^\vps_0(x,v).
 \end{array}\right.
\end{eqnarray}

In fact, the existence in arbitrary time interval $[0,t_e]$ of \eqref{fvps-eq} and \eqref{hvps-eq} follows from a local existence in a small time interval and the {\it a priori} energy estimate as well as a continuation argument. Here, for the sake of simplicity, we only show the {\it a priori} energy estimate under the {\it a priori} energy assumption
\begin{align}\label{aps-ld}
    \sup_{0\leq t\leq t_e}\mathcal{E}(t)\lesssim \varepsilon^{-\frac{1}{2}},
\end{align}
where
\begin{align}%\label{eg-le}
\mathcal{E}(t)=\sum_{i=0}^2 \varepsilon^i\left(\left\|\nabla_x^if^{\varepsilon}(t)\right\|^2+\varepsilon\left\|\nabla_x^ih^{\varepsilon}(t)\right\|^2\right).\notag
\end{align}

Our first result in this direction is on the energy estimates of $f^\vps$, which can be stated as in the following proposition.
\begin{proposition}\label{H2xf}
Assume that $f^\vps$ and $h^\vps$ are smooth solutions of the system \eqref{fvps-eq} and \eqref{hvps-eq} in time interval $[0,t_e]$ and satisfy \eqref{aps-ld}, then it holds that
\begin{align}\label{H2f-sum}
&\frac{\mathrm{d}}{\mathrm{d}t}\sum\limits_{i=0}^2\varepsilon^i\|\nabla^i_xf^{\varepsilon}\|^2
+\delta \sum\limits_{i=0}^2\varepsilon^{i-1}\|\nabla^i_x({\bf I}-{\bf P}_{\mathbf{M}})[f^{\varepsilon}]\|_D^2 \\
\lesssim& \epsilon_0\|({\bf I}-{\bf P}_{\mathbf{M}})[f^{\varepsilon}]\|_{H^1_D}^2
+\epsilon_0\sum\limits_{i=0}^2\varepsilon^i\|f^{\varepsilon}\|^2_{H^i}
+C_{\epsilon_0}\exp\left(\frac{-\epsilon_0}{8RT^2_c\sqrt{\varepsilon}}\right)
\|h^{\varepsilon}\|^2_{H^2}+\sum\limits_{i=0}^2\varepsilon^{2k+3+i}.\nonumber
\end{align}
\end{proposition}

To prove Proposition \ref{H2xf}, the trouble term is the velocity cubic growth $\mathbf{M}^{-\frac{1}{2}}\big(\partial_t+v\cdot\nabla_x\big)\mathbf{M}^{\frac{1}{2}}$. Our strategy to overcome this difficulty is to use $\FM^{-\frac{1}{2}}\mu^{\frac{1}{2}}$ to absorb such a growth. To illustrate this and make our presentation clear, we conclude the corresponding estimates in the following lemma.
\begin{lemma}\label{dxlm}
Assume that $f^\vps$ and $h^\vps$ are smooth solutions of the system \eqref{fvps-eq} and \eqref{hvps-eq} in time interval $[0,t_e]$, it holds for each $i=0,1,2$ that
\begin{align}\label{growdxi}
   &\Big|\Big\langle \nabla_x^i\Big(f^{\varepsilon}\mathbf{M}^{-\frac{1}{2}}\big(\partial_t+v\cdot\nabla_x\big)\mathbf{M}^{\frac{1}{2}}\Big), \nabla_x^if^{\varepsilon}\Big\rangle\Big|\nonumber\\
    \lesssim&\,\epsilon_0\|f^{\varepsilon}\|^2_{H^i}+\sum_{j=0}^i\frac{\epsilon_0}{\varepsilon^{(i+1-j)}} \left\|\nabla_x^j({\bf I}-{\bf P}_{\mathbf{M}})[f^{\varepsilon}]\right\|_D^2\\ &+C_{\epsilon_0}\exp\left(-\frac{\epsilon_0}{8RT^2_c\sqrt{\varepsilon}}\right)
   \|h^{\varepsilon}\|^2_{H^i}.\nonumber
\end{align}
\end{lemma}
\begin{proof}
First of all, from \eqref{e-sol-es}, it follows that
\begin{align}\label{dxfhi}
   &\Big|\Big\langle \nabla_x^i\Big(f^{\varepsilon}\mathbf{M}^{-\frac{1}{2}}\big(\partial_t+v\cdot\nabla_x\big)\mathbf{M}^{\frac{1}{2}}\Big), \nabla_x^if^{\varepsilon}\Big\rangle\Big|\nonumber\\
    \lesssim& \sum_{j=0}^i\CU(\|\nabla_x[\rho,u,T]\|_{W^{j,\infty}})
    \left\|{\lag v\rag}^{3/2+2j}\nabla_x^jf^{\varepsilon}\right\|\left\|{\lag v\rag}^{3/2}\nabla_x^if^{\varepsilon}\right\|\\
     \lesssim& \sum_{j=0}^i\epsilon_0
    \left\|{\lag v\rag}^{3/2+2(i-j)}\nabla_x^jf^{\varepsilon}\right\|\left\|{\lag v\rag}^{3/2}\nabla_x^if^{\varepsilon}\right\|.\nonumber
\end{align}

In order to eliminate the velocity growth in the right hand side of \eqref{dxfhi}, we divide the domain of the $v-$integration ${\mathbb R}^3_v$ into $\vps-$dependent low velocity part and high velocity part, so that the microscopic  and  macroscopic component of the low velocity part can be bounded by $\frac{1}{\varepsilon}\|\partial_x^i({\bf I}-{\bf P}_{\mathbf{M}})[f^{\varepsilon}]\|_D^2$ and $\|f^{\varepsilon}\|^2_{H^i}$ respectively, while the high velocity part can be bounded by $\|h^{\varepsilon}\|^2_{H^i}$ due to the choice of $T_c$ in the global Maxwellian $\mu$.
More specifically, by \eqref{tt0}, one has for suitably small $\epsilon_0>0$,
\begin{align*}
\mathbf{M}^{-\frac{1}{2}}\mu^{\frac{1}{2}}=&\Big(\frac{1 }{ \rho}\Big)^{\frac{1}{2}}\Big(\frac{T }{ T_c}\Big)^{3/2}
 \exp\left(-\frac{(T-T_c)|v|^2}{4RTT_c}+\frac{v\cdot u}{2RT}+\frac{|u|^2}{4RT}\right)\\
\lesssim& \exp\left(-\frac{(T-T_c)|v|^2}{8RTT_c}\right) \lesssim C_{\epsilon_0}\exp\left(-\frac{\epsilon_0|v|^2}{16RT^2_c}\right),
\end{align*}
which further implies that
 \begin{align}\label{dxihfL}
\left|\nabla_x^if^{\varepsilon}\right|=&\left|\nabla_x^i\Big(\mathbf{M}^{-\frac{1}{2}}\mu^{\frac{1}{2}}h^{\varepsilon}\Big)\right|
\lesssim C_{\epsilon_0}\sum_{j=0}^i\left|\nabla_x^jh^{\varepsilon}\right| \exp\left(-\frac{\epsilon_0|v|^2}{16RT^2_c}\right).
\end{align}
Thus we get from \eqref{dxihfL} that
\begin{align}\label{dxfhi0}
&\sum_{j=0}^i\left\|{\lag v\rag}^{3/2+2(i-j)}\nabla_x^jf^{\varepsilon}\right\|\nonumber\\
\lesssim&\sum_{j=0}^i\int_{{\mathbb R}^3}\int_{{\lag v\rag}^{4(i+1-j)}\leq\varepsilon^{-(i+1-j)}}{\lag v\rag}^{3+4(i-j)}\left|\nabla_x^jf^{\varepsilon}\right|^2\, dv dx\notag\\
&+\sum_{j=0}^i\int_{{\mathbb R}^3}\int_{{\lag v\rag}^{4(i+1-j)}\geq\varepsilon^{-(i+1-j)}}{\lag v\rag}^{3+4(i-j)}\left|\nabla_x^jf^{\varepsilon}\right|^2\, dv dx\nonumber\\
\lesssim& \sum_{j=0}^i\left\|{\lag v\rag}^{3/2+2(i-j)}\nabla_x^j{{\bf P}_{\mathbf{M}}}[f^{\varepsilon}]\right\|^2+\sum_{j=0}^i\frac{1}{\varepsilon^{(i+1-j)}} \left\|\nabla_x^j({\bf I}-{\bf P}_{\mathbf{M}})[f^{\varepsilon}]\right\|_D^2\\
&+\sum_{j=0}^iC_{\epsilon_0}\int_{{\mathbb R}^3}\int_{{\lag v\rag}^{4(i+1-j)}\geq\varepsilon^{-(i+1-j)}} \exp\left(-\frac{\epsilon_0|v|^2}{8T^2_c}\right)|h^{\varepsilon}|^2\, dv dx\nonumber\\
\lesssim&\,\sum_{j=0}^i\|f^{\varepsilon}\|^2_{H^i}+\sum_{j=0}^i\frac{1}{\varepsilon^{(i+1-j)}}\left\|\nabla_x^j({\bf I}-{\bf P}_{\mathbf{M}})[f^{\varepsilon}]\right\|_D^2\notag\\ &+C_{\epsilon_0}\exp\left(-\frac{\epsilon_0}{8RT^2_c\sqrt{\varepsilon}}\right)\|h^{\varepsilon}\|^2_{H^i}.\nonumber
\end{align}

$\left\|{\lag v\rag}^{3/2}\nabla_x^if^{\varepsilon}\right\|$ enjoys the same bound as above. This completes the proof of Lemma \ref{dxlm}.
\end{proof}

Now we turn to prove Proposition \ref{H2xf} with the aid of Lemmas \ref{L-co-lem}, \ref{GL0} and \ref{dxlm}.

\begin{proof}
The proof is divided into the following three steps.

%%%%%%%%%%%%%%%%%%%%%%%%%%%%%%%%%%%%%%%%%%%%%%%%%%%%%%%%%%%%%%%%%%%%%%%%%%%%%%%%%%
%\subsubsection{Basic Energy Estimate of $f^{\varepsilon}$}
%%%%%%%%%%%%%%%%%%%%%%%%%%%%%%%%%%%%%%%%%%%%%%%%%%%%%%%%%%%%%%%%%%%%%%%%%%%%%%%%%%

\noindent\underline{{\it Step 1. Zeroth order energy estimate on $f^\vps$.}} In this step we deduce the zeroth order energy estimate on $f^\vps$. In fact, we can obtain that
\begin{align}\label{L2f}
    \,\frac{\mathrm{d}}{\mathrm{d} t}\|f^{\varepsilon}\|^2
    +\frac{\delta}{\varepsilon}\|({\bf I}-{\bf P}_{\mathbf{M}})[f^{\varepsilon}]\|_D^2    \lesssim\,\epsilon_0\|f^{\varepsilon}\|^2+C_{\epsilon_0}\exp\left(-\frac{\epsilon_0}{8RT^2_c\sqrt{\varepsilon}}\right)\|h^{\varepsilon}\|^2
    +\varepsilon^{2k+3}.
\end{align}

To this end, we take the $L^2$ inner product of $f^{\varepsilon}$ and\eqref{m0F1} and use \eqref{coL0} to get that
\begin{align}\label{L2f1}
\frac{1}{2}\frac{\mathrm{d}}{\mathrm{d} t}\|f^{\varepsilon}\|^2
&+\frac{\delta}{\varepsilon}\|({\bf I}-{\bf P}_{\mathbf{M}})[f^{\varepsilon}]\|^2_D \\
\leq&\;\Big|\Big\langle f^{\varepsilon}\mathbf{M}^{-\frac{1}{2}}\big(\partial_t+v\cdot\nabla_x\big)\mathbf{M}^{\frac{1}{2}}, f^{\varepsilon}\Big\rangle\Big|
    +\varepsilon^{k-1}\big|\big\langle\Gamma_{\mathbf{M}} ( f^{\varepsilon},
    f^{\varepsilon} ), f^{\varepsilon}\big\rangle\big|\nonumber\\
    &+\sum_{i=1}^{2k-1}\varepsilon^{i-1}\big|\big\langle[\Gamma_{\mathbf{M}}(\mathbf{M}^{-\frac{1}{2}}F_i,f^{\varepsilon})+\Gamma_{\mathbf{M}}(
 f^{\varepsilon}, \mathbf{M}^{-\frac{1}{2}} F_i)\Big], f^{\varepsilon}\big\rangle\big|+\big|\big\langle \CP_0,f^{\varepsilon}\big\rangle\big|.\nonumber
\end{align}
We now turn to estimate terms on the R.H.S. of \eqref{L2f1} term by term.

From \eqref{growdxi} in Lemma \ref{dxlm}, one sees that the first term on the R.H.S. of \eqref{L2f1} can be dominated by
\begin{align*}
   \frac{C\epsilon_0}{\varepsilon}\|({\bf I}-{\bf P}_{\mathbf{M}})[f^{\varepsilon}]\|_D^2+C\epsilon_0\|f^{\varepsilon}\|^2
      +C_{\epsilon_0}\exp\left(-\frac{\epsilon_0}{8RT^2_c\sqrt{\varepsilon}}\right)\|h^{\varepsilon}\|^2.\nonumber
\end{align*}

For the second term on the R.H.S. of \eqref{L2f1}, noting that for $k\geq 3$, we can deduce from the {\it a priori} assumption \eqref{aps-ld} that
\begin{align}\label{fh2b}
\varepsilon^{k-1}\|f^{\varepsilon}\|_{H^2}\lesssim \varepsilon^{\frac{1}{2}},
\end{align}
thus we can get by \eqref{GLM} that
\begin{align*}
\varepsilon^{k-1}\big|\big\langle\Gamma_{\mathbf{M}} ( f^{\varepsilon},
    f^{\varepsilon} ), f^{\varepsilon}\big\rangle\big|=&\varepsilon^{k-1}\big|\big\langle\Gamma_{\mathbf{M}}(f^{\varepsilon},f^{\varepsilon}),({\bf I}-{\bf P}_{\mathbf{M}})[f^{\varepsilon}]\big\rangle\big|\\
    \lesssim&\varepsilon^{k-1}\int_{{\mathbb R}^3}|f^{\varepsilon}|_{L^2}|f^{\varepsilon}|_D|({\bf I}-{\bf P}_{\mathbf{M}})[f^{\varepsilon}]|_D\, \nonumber\\
    \lesssim&\varepsilon^{k-1}\|f^{\varepsilon}\|_{H^2}\Big(\|({\bf I}-{\bf P}_{\mathbf{M}})[f^{\varepsilon}]\|_D+\|{{\bf P}_{\mathbf{M}}}[f^{\varepsilon}]\|_D\Big)\|({\bf I}-{\bf P}_{\mathbf{M}})[f^{\varepsilon}]\|_D\nonumber\\
    \lesssim &\,\varepsilon^{\frac{1}{2}}\|({\bf I}-{\bf P}_{\mathbf{M}})[f^{\varepsilon}]\|_D^2+\varepsilon^{\frac{1}{2}}\|{{\bf P}_{\mathbf{M}}}[f^{\varepsilon}]\|^2\\
    \lesssim& \varepsilon^{\frac{1}{2}}\|({\bf I}-{\bf P}_{\mathbf{M}})[f^{\varepsilon}]\|_D^2+\varepsilon^{\frac{1}{2}}\|f^{\varepsilon}\|^2 .\nonumber
\end{align*}

For the third term on the R.H.S. of \eqref{L2f1},
we get from \eqref{Fn-es} that
\begin{align*}
   &\sum_{i=1}^{2k-1}\varepsilon^{i-1}\left|\left
   \langle\left[\Gamma_{\mathbf{M}}(\mathbf{M}^{-\frac{1}{2}}F_i,f^{\varepsilon})+\Gamma_{\mathbf{M}}(
 f^{\varepsilon}, \mathbf{M}^{-\frac{1}{2}} F_i)\right], f^{\varepsilon}\right\rangle\right|\\
    \lesssim&\, \frac{o(1)}{\varepsilon}\|({\bf I}-{\bf P}_{\mathbf{M}})[f^{\varepsilon}]\|_D^2+\varepsilon\big(\|{f^{\varepsilon}}\|_D^2+\varepsilon\|f^{\varepsilon}\|^2\big)\\
 \lesssim& \frac{o(1)}{\varepsilon}\|({\bf I}-{\bf P}_{\mathbf{M}})[f^{\varepsilon}]\|_D^2+\varepsilon\|f^{\varepsilon}\|^2\nonumber.
\end{align*}

For the last term on the R.H.S. of \eqref{L2f1}, using Cauchy's inequality, we have
\begin{align*}
    \big|\big\langle \CP_0,f^{\varepsilon}\big\rangle\big|\lesssim& \frac{o(1)}{\varepsilon}\|({\bf I}-{\bf P}_{\mathbf{M}})[f^{\varepsilon}]\|_D^2+\varepsilon\sum_{\substack{i+j\geq 2k+1\\2\leq i,j\leq2k-1}}\varepsilon^{2(i+j-k)}\|F_i\|^2_{H^2}\|F_j\|_D^2\\
    \lesssim& \frac{o(1)}{\varepsilon}\|({\bf I}-{\bf P}_{\mathbf{M}})[f^{\varepsilon}]\|_D^2+\varepsilon^{2k+3}.\nonumber
\end{align*}
Finally, substituting the above estimates into \eqref{L2f1} gives \eqref{L2f}.

\vskip 0.2cm

\noindent\underline{{\it Step 2. Estimates on the first order derivative of $f^{\varepsilon}$.}}
In this step, we proceed to derive the estimate of $\nabla_xf^{\varepsilon}$. For results in this direction, we have
\begin{align}\label{H1f}
    &\frac{\mathrm{d}}{\mathrm{d} t}\big(\varepsilon\|\nabla_xf^{\varepsilon}\|^2\big)
    +\delta\|\nabla_x({\bf I}-{\bf P}_{\mathbf{M}})[f^{\varepsilon}]\|_D^2\nonumber\\
\lesssim& \frac{\epsilon_0}{\varepsilon}\|({\bf I}-{\bf P}_{\mathbf{M}})[f^{\varepsilon}]\|_D^2+\epsilon_0\Big(\|f^{\varepsilon}\|^2+ \varepsilon\|f^{\varepsilon}\|_{H^1}^2\Big)\\
    &    +C_{\epsilon_0}\exp\left(\frac{-\epsilon_0}{8RT^2_c\sqrt{\varepsilon}}\right)
      \left(\left\|\nabla_xh^{\varepsilon}\right\|^2+\|h^{\varepsilon}\|^2\right)+\varepsilon^{2k+4}.\nonumber
\end{align}

To prove \eqref{H1f}, we have by applying $\partial_x^{\alpha} (1\leq |\alpha|\leq 2)$ to \eqref{m0F1} that
\begin{align}\label{m0F1x}
&\big(\partial_t+v\cdot\nabla_x\big)\partial_x^{\alpha}f^{\varepsilon}
        +\frac{1}{\varepsilon} \partial_x^{\alpha}\mathcal{L}_{\mathbf{M}}[f^{\varepsilon}]\nonumber\\
=&-\partial_x^{\alpha}\Big(f^{\varepsilon}\mathbf{M}^{-\frac{1}{2}}\big(\partial_t+v\cdot\nabla_x\big)\mathbf{M}^{\frac{1}{2}}\Big)
    +\varepsilon^{k-1}\partial_x^{\alpha}\Gamma_{\mathbf{M}} ( f^{\varepsilon},
    f^{\varepsilon} )\\
&+\sum_{i=1}^{2k-1}\varepsilon^{i-1}[\partial_x^{\alpha}\Gamma_{\mathbf{M}}(\mathbf{M}^{-\frac{1}{2}}F_i,f^{\varepsilon})+\partial_x^{\alpha}\Gamma_{\mathbf{M}}(
 f^{\varepsilon}, \mathbf{M}^{-\frac{1}{2}} F_i)\Big]+\partial_x^{\alpha}\CP_0.\nonumber
\end{align}

Letting $|\alpha|=1$, we take the inner product of \eqref{m0F1x} with $\partial^\alpha_xf^{\varepsilon}$ to obtain
\begin{align}\label{H1f1}
    &\frac{1}{2}\frac{\mathrm{d}}{\mathrm{d} t}\left\|\nabla_xf^{\varepsilon}\right\|^2
    +\frac{1}{\varepsilon}\big\langle \nabla_x\mathcal{L}_{\mathbf{M}}[f^{\varepsilon}],\nabla_xf^{\varepsilon}\big\rangle\nonumber\\
    \leq&\;\Big|\Big\langle \nabla_x\Big(f^{\varepsilon}\mathbf{M}^{-\frac{1}{2}}\big(\partial_t+v\cdot\nabla_x\big)\mathbf{M}^{\frac{1}{2}}\Big), \nabla_xf^{\varepsilon}\Big\rangle\Big|
    +\varepsilon^{k-1}\big|\big\langle \nabla_x\Gamma_{\mathbf{M}} ( f^{\varepsilon},
    f^{\varepsilon} ), \nabla_xf^{\varepsilon}\big\rangle\big|\\
    &+\sum_{i=1}^{2k-1}\varepsilon^{i-1}\big|\big\langle[\nabla_x\Gamma_{\mathbf{M}}(\mathbf{M}^{-\frac{1}{2}}F_i,f^{\varepsilon})
    +\nabla_x\Gamma_{\mathbf{M}}(
 f^{\varepsilon}, \mathbf{M}^{-\frac{1}{2}} F_i)\big], \nabla_xf^{\varepsilon}\big\rangle\big|+\big|\big\langle \nabla_x \CP_0, \nabla_xf^{\varepsilon}\big\rangle\big|.\nonumber
\end{align}

For the second term on the L.H.S of \eqref{H1f1}, it follows from
\eqref{coLh} that
\begin{align*}
&\frac{1}{\varepsilon}\big\langle \nabla_x\mathcal{L}_{\mathbf{M}}[f^{\varepsilon}],\nabla_xf^{\varepsilon}\big\rangle\\
\geq&\frac{3\delta}{4\varepsilon}\|\nabla_x({\bf I}-{\bf P}_{\mathbf{M}})[f^{\varepsilon}]\|_D^2-C\epsilon_0\Big(\frac{1}{\varepsilon^2}\|({\bf I}-{\bf P}_{\mathbf{M}})[f^{\varepsilon}]\|_D^2+\|f^{\varepsilon}\|^2_{H^1}+\frac{1}{\varepsilon}\|f^{\varepsilon}\|^2\Big).
\end{align*}
We now control the terms on the R.H.S. of \eqref{H1f1}. For the first term on the R.H.S. of \eqref{H1f1}, from \eqref{growdxi} with $i=1$, it follows
\begin{align*}%\label{growL2}
&\Big|\Big\langle \nabla_x\Big(f^{\varepsilon}\mathbf{M}^{-\frac{1}{2}}\big(\partial_t+v\cdot\nabla_x\big)\mathbf{M}^{\frac{1}{2}}\Big), \nabla_xf^{\varepsilon}\Big\rangle\Big|\notag\\
          \lesssim&\,\epsilon_0\|f^{\varepsilon}\|^2_{H^1}+\frac{\epsilon_0}{\varepsilon}\left\|\nabla_x({\bf I}-{\bf P}_{\mathbf{M}})[f^{\varepsilon}]\right\|_D^2+\frac{\epsilon_0}{\varepsilon^2}\|({\bf I}-{\bf P}_{\mathbf{M}})[f^{\varepsilon}]\|_D^2\\
      &+C_{\epsilon_0}\exp\left(-\frac{\epsilon_0}{8RT^2_c\sqrt{\varepsilon}}\right)
      \big(\|\nabla_xh^{\varepsilon}\|^2+\|h^{\varepsilon}\|^2\big).\nonumber
      \end{align*}

For the second term on the R.H.S. of \eqref{H1f1},  we use \eqref{GLM} and \eqref{fh2b} to obtain
\begin{align*}
&\varepsilon^{k-1}\big|\big\langle \nabla_x\Gamma_{\mathbf{M}} ( f^{\varepsilon},
    f^{\varepsilon} ), \nabla_xf^{\varepsilon}\big\rangle\big|\\
        \lesssim&\varepsilon^{k-1}\|f^{\varepsilon}\|_{H^2}\Big(\|\nabla_x({\bf I}-{\bf P}_{\mathbf{M}})[f^{\varepsilon}]\|_D+\|\nabla_x{{\bf P}_{\mathbf{M}}}[f^{\varepsilon}]\|_D\Big) \Big(\|\nabla_x({\bf I}-{\bf P}_{\mathbf{M}})[f^{\varepsilon}]\|_D+\epsilon_0\|f^{\varepsilon}\|\Big)\nonumber\\
    &+\epsilon_0 \varepsilon^{\frac{1}{2}}\Big(\|({\bf I}-{\bf P}_{\mathbf{M}})[f^{\varepsilon}]\|_D+\|f^{\varepsilon}\|\Big)\Big(\|\nabla_x({\bf I}-{\bf P}_{\mathbf{M}})[f^{\varepsilon}]\|_D+\|f^{\varepsilon}\|_{H^1}\Big)\\
    \lesssim &\,\|\nabla_x({\bf I}-{\bf P}_{\mathbf{M}})[f^{\varepsilon}]\|_D^2+\|({\bf I}-{\bf P}_{\mathbf{M}})[f^{\varepsilon}]\|_D^2+\epsilon_0\Big(\|f^{\varepsilon}\|^2+ \varepsilon\|f^{\varepsilon}\|_{H^1}^2\Big).\nonumber
\end{align*}

For the  third term on the R.H.S. of \eqref{H1f1}, in view of \eqref{GLM} and  \eqref{Fn-es}, we have
\begin{align*}
&\sum_{i=1}^{2k-1}\varepsilon^{i-1}\big|\big\langle[\nabla_x\Gamma_{\mathbf{M}}(\mathbf{M}^{-\frac{1}{2}}F_i,f^{\varepsilon}) +\nabla_x\Gamma_{\mathbf{M}}(
 f^{\varepsilon}, \mathbf{M}^{-\frac{1}{2}} F_i)\big], \nabla_xf^{\varepsilon}\big\rangle\big|\\
 %\leq&\sum_{i=1}^{2k-1}\varepsilon^{i-1}\big|\big\langle\big[\Gamma_{\mathbf{M}}\big(\partial_x\big(\mathbf{M}^{-\frac{1}{2}}F_i\big),f^{\varepsilon}\big)
% +\Gamma_{\mathbf{M}}(
% f^{\varepsilon}, \partial_x\big(\mathbf{M}^{-\frac{1}{2}}F_i\big)\big)\big], ({\bf I}-{\bf P}_{\mathbf{M}})\partial_xf^{\varepsilon}\big\rangle\big|\\
% &+\sum_{i=1}^{2k-1}\varepsilon^{i-1}\big|\big\langle\big[\Gamma_{\mathbf{M}}\big(\mathbf{M}^{-\frac{1}{2}}F_i,\partial_xf^{\varepsilon}\big)
% +\Gamma_{\mathbf{M}}(
% \partial_xf^{\varepsilon}, \mathbf{M}^{-\frac{1}{2}}F_i\big)\big], ({\bf I}-{\bf P}_{\mathbf{M}})\partial_xf^{\varepsilon}\big\rangle\big|\\
% &+\left(\|\nabla_x\rho\|_{\infty}+\|\nabla_xu\|_{\infty}+\|\nabla_xT\|_{\infty}\right)\sum_{i=1}^{2k-1}\varepsilon^{i-1}\int_{{\mathbb R}^3}\big(\|{f^{\varepsilon}}\|+\|{f^{\varepsilon}}\|_D\big)\|\partial_xf^{\varepsilon}\|_D\\
 \lesssim&\big(\|{f^{\varepsilon}}\|_{H^1}+\|{f^{\varepsilon}}\|_{H^1_D}\big)\Big(\|\nabla_x({\bf I}-{\bf P}_{\mathbf{M}})[f^{\varepsilon}]\|_D+\|\Lbrack\nabla_x,{{\bf P}_{\mathbf{M}}}\Rbrack[f^{\varepsilon}]\|_D\Big)\\
 &+\epsilon_0\big(\|{f^{\varepsilon}}\|+\|{f^{\varepsilon}}\|_D\big)\Big(\|\nabla_x({\bf I}-{\bf P}_{\mathbf{M}})[f^{\varepsilon}]\|_D+\|f^{\varepsilon}\|_{H^1}\Big)\\
    \lesssim&\, \frac{o(1)}{\varepsilon}\|\nabla_x({\bf I}-{\bf P}_{\mathbf{M}})[f^{\varepsilon}]\|_D^2+\|({\bf I}-{\bf P}_{\mathbf{M}})[f^{\varepsilon}]\|_D^2+\epsilon_0\|f^{\varepsilon}\|^2_{H^1}\nonumber.
\end{align*}

Similarly, for the last term on the R.H.S. of \eqref{H1f1}, one gets from \eqref{GLM} and  \eqref{Fn-es} that
\begin{align*}
    \big|\big\langle \nabla_x \CP_0, \nabla_xf^{\varepsilon}\big\rangle\big|
    \lesssim& \frac{o(1)}{\varepsilon}\|\nabla_x({\bf I}-{\bf P}_{\mathbf{M}})[f^{\varepsilon}]\|_D^2+\frac{\epsilon_0}{\varepsilon}\|({\bf I}-{\bf P}_{\mathbf{M}})[f^{\varepsilon}]\|_D^2+\frac{\epsilon_0}{\varepsilon}\|f^{\varepsilon}\|^2+\varepsilon^{2k+3}.\nonumber
\end{align*}
%Here we used the integration by parts W.R.T. $x$ when $\partial_x$ hits the local Maxwellian in $\CP_0$.

Finally, plugging the estimates above into \eqref{H1f1} and multiplying the resulting inequality by $\varepsilon$ gives \eqref{H1f}.

\vskip 0.2cm

\noindent\underline{{\it Step 3. Estimates on the second order derivative of $f^{\varepsilon}$.}}
%%%%%%%%%%%%%%%%%%%%%%%%%%%%%%%%%%%%%%%%%%%%%%%%%%%%%%%%%%%%%%%%%%%%%%%%%%%%%%%%%%
In this step, we intend to derive the energy estimates on  $\partial^\alpha f^{\varepsilon}$ for $|\alpha|=2$. For results in this direction, we have
\begin{align}\label{H2f}
    &\frac{\mathrm{d}}{\mathrm{d} t}\big(\varepsilon^2\|\nabla^2_xf^{\varepsilon}\|^2\big)
    +\delta \varepsilon\|\nabla^2_x({\bf I}-{\bf P}_{\mathbf{M}})[f^{\varepsilon}]\|_D^2\nonumber\\
\lesssim& \epsilon_0\|({\bf I}-{\bf P}_{\mathbf{M}})[f^{\varepsilon}]\|_{H^1_D}^2
    +\epsilon_0\Big(\varepsilon^2\|f^{\varepsilon}\|^2_{H^2}+\varepsilon\|f^{\varepsilon}\|^2_{H^1}\Big)\\
    &+C_{\epsilon_0}\exp\left(\frac{-\epsilon_0}{8RT^2_c\sqrt{\varepsilon}}\right)
     \|h^{\varepsilon}\|^2_{H^2}+\varepsilon^{2k+5}.\nonumber
\end{align}

To prove \eqref{H2f}, we first get from \eqref{m0F1x} with $|\alpha|=2$ that
\begin{align}\label{H2f1}
&\frac{1}{2}\frac{\mathrm{d}}{\mathrm{d} t}\|\nabla_x^2f^{\varepsilon}\|^2
    +\frac{1}{\varepsilon}\big\langle \nabla_x^2\mathcal{L}_{\mathbf{M}}[f^{\varepsilon}],\nabla_x^2f^{\varepsilon}\big\rangle\nonumber\\
\leq&\;\Big|\Big\langle \nabla_x^2\Big(f^{\varepsilon}\mathbf{M}^{-\frac{1}{2}}\big(\partial_t+v\cdot\nabla_x\big)\mathbf{M}^{\frac{1}{2}}\Big), \nabla_x^2f^{\varepsilon}\Big\rangle\Big|
    +\varepsilon^{k-1}\big|\big\langle \nabla_x^2\Gamma_{\mathbf{M}} ( f^{\varepsilon},
    f^{\varepsilon} ), \nabla_x^2f^{\varepsilon}\big\rangle\big|\\
    &+\sum_{i=1}^{2k-1}\varepsilon^{i-1}\big|\big\langle[\nabla_x^2\Gamma_{\mathbf{M}}(\mathbf{M}^{-\frac{1}{2}}F_i,f^{\varepsilon}) +\nabla_x^2\Gamma_{\mathbf{M}}(
 f^{\varepsilon}, \mathbf{M}^{-\frac{1}{2}} F_i)\big], \nabla_x^2f^{\varepsilon}\big\rangle\big|+\big|\big\langle \nabla_x^2\CP_0,\nabla_x^2f^{\varepsilon}\big\rangle\big|.\nonumber
\end{align}

By \eqref{coLh} and \eqref{e-sol-es}, one has
\begin{align*}
\frac{1}{\varepsilon}\big\langle \nabla_x^2\mathcal{L}_{\mathbf{M}}[f^{\varepsilon}],\nabla_x^2f^{\varepsilon}\big\rangle
\geq&\frac{3\delta}{4\varepsilon}\|\nabla_x^2({\bf I}-{\bf P}_{\mathbf{M}})[f^{\varepsilon}]\|_D^2\\
&-C\epsilon_0\Big(\frac{1}{\varepsilon^2}\|({\bf I}-{\bf P}_{\mathbf{M}})[f^{\varepsilon}]\|_{H^1_D}^2+\|f^{\varepsilon}\|^2_{H^2}+\frac{1}{\varepsilon}\|f^{\varepsilon}\|^2_{H^1}\Big).
\end{align*}
We now turn to estimate the terms on the R.H.S. of \eqref{H2f1} separately. For the first term on the R.H.S. of \eqref{H2f1},
we get from \eqref{growdxi} with $i=2$ that
\begin{align*}%\label{growL2}
\Big|\Big\langle \nabla_x^2&\Big(f^{\varepsilon}\mathbf{M}^{-\frac{1}{2}}\big(\partial_t+v\cdot\nabla_x\big)\mathbf{M}^{\frac{1}{2}}\Big), \nabla_x^2f^{\varepsilon}\Big\rangle\Big|\\
          \lesssim&\,\epsilon_0\|f^{\varepsilon}\|^2_{H^2}+\frac{\epsilon_0}{\varepsilon}\|\nabla_x^2({\bf I}-{\bf P}_{\mathbf{M}})[f^{\varepsilon}]\|_D^2+\frac{\epsilon_0}{\varepsilon^2}\|\nabla_x({\bf I}-{\bf P}_{\mathbf{M}})[f^{\varepsilon}]\|_D^2\\
      &+\frac{\epsilon_0}{\varepsilon^3}\|({\bf I}-{\bf P}_{\mathbf{M}})[f^{\varepsilon}]\|_D^2+C_{\epsilon_0}\exp\left(-\frac{\epsilon_0}{8RT^2_c\sqrt{\varepsilon}}\right)
    \|h^{\varepsilon}\|^2_{H^2}.\nonumber
\end{align*}

For second term on the R.H.S. of \eqref{H2f1}, we can obtain by using \eqref{GLM} and  \eqref{fh2b} as well as Sobolev's inequalities that
\begin{align*}
&\varepsilon^{k-1}\big|\big\langle \nabla_x^2\Gamma_{\mathbf{M}} ( f^{\varepsilon},
    f^{\varepsilon} ), \nabla_x^2f^{\varepsilon}\big\rangle\big|\\
        \lesssim&\varepsilon^{k-1}\int_{{\mathbb R}^3}\Big(|\nabla_x^2f^{\varepsilon}|_{D}|f^{\varepsilon}|_{L^2}+|\nabla_x^2f^{\varepsilon}|_{L^2}|f^{\varepsilon}|_{D}
    +|\nabla_xf^{\varepsilon}|_{L^2}|\nabla_xf^{\varepsilon}|_{D}\Big)\\
    &\times\Big(\|\nabla_x^2({\bf I}-{\bf P}_{\mathbf{M}})[f^{\varepsilon}]\|_D+\|\Lbrack\nabla_x^2,{{\bf P}_{\mathbf{M}}}\Rbrack[f^{\varepsilon}]\|_D\Big)\\
    &+\epsilon_0 \varepsilon^{k-1}\|f^{\varepsilon}\|_{H^2}\Big(\|({\bf I}-{\bf P}_{\mathbf{M}})[f^{\varepsilon}]\|_{H^1_D}+\|{{\bf P}_{\mathbf{M}}}[f^{\varepsilon}]\|_{H^1_D}\Big)\Big(\|\nabla_x^2({\bf I}-{\bf P}_{\mathbf{M}})[f^{\varepsilon}]\|_D+\|\nabla_x^2{{\bf P}_{\mathbf{M}}}[f^{\varepsilon}]\|_D\Big)\\
    \lesssim&\varepsilon^{k-1}\|f^{\varepsilon}\|_{H^2}\Big(\|\nabla^2_x({\bf I}-{\bf P}_{\mathbf{M}})[f^{\varepsilon}]\|_D+\|\nabla^2_x{{\bf P}_{\mathbf{M}}}[f^{\varepsilon}]\|_D\Big)\Big(\|\nabla_x^2({\bf I}-{\bf P}_{\mathbf{M}})[f^{\varepsilon}]\|_D+\epsilon_0\|f^{\varepsilon}\|_{H^1}\Big)\nonumber\\
    &+\epsilon_0 \varepsilon^{\frac{1}{2}}\Big(\|({\bf I}-{\bf P}_{\mathbf{M}})[f^{\varepsilon}]\|_{H^1_D}+\|f^{\varepsilon}\|_{H^1}\Big)\Big(\|\nabla_x^2({\bf I}-{\bf P}_{\mathbf{M}})[f^{\varepsilon}]\|_D+\|f^{\varepsilon}\|_{H^2}\Big)\\
    \lesssim &\,\|({\bf I}-{\bf P}_{\mathbf{M}})[f^{\varepsilon}]\|_{H^2_D}^2+\epsilon_0\Big(\|f^{\varepsilon}\|^2_{H^1}+ \varepsilon\|f^{\varepsilon}\|_{H^2}^2\Big).\nonumber
\end{align*}

Similarly, for the third and fourth term on the R.H.S. of \eqref{H2f1}, applying
\eqref{GLM} and \eqref{Fn-es}, we  get
\begin{align*}
\sum_{i=1}^{2k-1}\varepsilon^{i-1}&\Big|\Big\langle\Big[\nabla_x^2\Gamma_{\mathbf{M}}(\mathbf{M}^{-\frac{1}{2}}F_i, f^{\varepsilon})+\nabla_x^2\Gamma_{\mathbf{M}}(
 f^{\varepsilon}, \mathbf{M}^{-\frac{1}{2}} F_i)\big], \nabla_x^2f^{\varepsilon}\Big\rangle\Big|\\
  \lesssim&\big(\|{f^{\varepsilon}}\|_{H^2}+\|{f^{\varepsilon}}\|_{H^2_D}\big)\Big(\|\nabla_x^2({\bf I}-{\bf P}_{\mathbf{M}})[f^{\varepsilon}]\|_D+\|\Lbrack\nabla_x^2,{{\bf P}_{\mathbf{M}}}\Rbrack[f^{\varepsilon}]\|_D\Big)\\
 &+\epsilon_0\big(\|{f^{\varepsilon}}\|_{H^1}+\|{f^{\varepsilon}}\|_{H^1_D}\big)\Big(\|\nabla_x^2({\bf I}-{\bf P}_{\mathbf{M}})[f^{\varepsilon}]\|_D+\|f^{\varepsilon}\|_{H^2}\Big)\\
    \lesssim&\, \|\nabla_x^2({\bf I}-{\bf P}_{\mathbf{M}})[f^{\varepsilon}]\|_D^2+\|({\bf I}-{\bf P}_{\mathbf{M}})[f^{\varepsilon}]\|_{H^1_D}^2+\epsilon_0\|f^{\varepsilon}\|^2_{H^2}\nonumber
\end{align*}
and
\begin{align*}
    \big|\big\langle \nabla_x^2\CP_0, \nabla_x^2f^{\varepsilon}\big\rangle\big|
    \lesssim& \frac{o(1)}{\varepsilon}\|\nabla_x^2({\bf I}-{\bf P}_{\mathbf{M}})[f^{\varepsilon}]\|_D^2+\frac{\epsilon_0}{\varepsilon}\|({\bf I}-{\bf P}_{\mathbf{M}})[f^{\varepsilon}]\|^2_{H^1_D}\\
    &+\frac{\epsilon_0}{\varepsilon}\|f^{\varepsilon}\|^2_{H^1}+\varepsilon^{2k+3}.\nonumber
\end{align*}

Putting all the above estimates into \eqref{H2f1} and multiply the resulting inequality by $\varepsilon^2$ gives \eqref{H2f}.

Now \eqref{H2f-sum} follows from \eqref{L2f}, \eqref{H1f} and \eqref{H2f}, this end the proof of Proposition \ref{H2xf}.
\end{proof}

%%%%%%%%%%%%%%%%%%%%%%%%%%%%%%%%%%%%%%%%%%%%%%%%%%%%%%%%%%%%%%%%%%%%%%%%%%%%%%%%%%
%\subsubsection{Energy Estimates of $h^{\varepsilon}$} \label{Sec:Energy-h}
%%%%%%%%%%%%%%%%%%%%%%%%%%%%%%%%%%%%%%%%%%%%%%%%%%%%%%%%%%%%%%%%%%%%%%%%%%%%%%%%%%

%\setcounter{equation}{0}

To close the estimates established in \eqref{H2f-sum}, we now turn to derive the $H^2$ estimates of $h^\vps$, which is the main content of the following lemma.
\begin{proposition}\label{H2xh-prop}
Assume that $f^\vps$ and $h^\vps$ are smooth solutions of the system \eqref{fvps-eq} and \eqref{hvps-eq} in time interval $[0,t_e]$ and satisfy \eqref{aps-ld}, then it holds that
\begin{align}\label{H2h-sum}
\frac{\mathrm{d}}{\mathrm{d} t}\sum\limits_{i=0}^2&\big(\varepsilon^{i+1}\|\nabla^i_xh^{\varepsilon}\|^2\big)
    +\delta\sum\limits_{i=0}^2\varepsilon^i\|\nabla_x({\bf I}-{\bf P})[h^{\varepsilon}]\|_D^2\\    \lesssim&\,\epsilon_0\sum\limits_{i=1}^2\varepsilon^i\|({\bf I}-{\bf P})[h^{\varepsilon}]\|_{H^{i-1}_D}^2+\epsilon_0\sum\limits_{i=0}^2\varepsilon\|f^{\varepsilon}\|^2_{H^i}
    +\sum\limits_{i=0}^2\varepsilon^{2k+3+2i}.\notag
\end{align}
\end{proposition}
\begin{proof}
The proof of Proposition \ref{H2h-sum} is easier than that of Proposition \ref{H2xf}, since the operators we will deal with are perturbations of the global Maxwellian $\mu$. Our proof is divided into two steps.
\vskip 0.2cm

\noindent\underline{{\it Step 1. Basic Energy Estimate of $h^{\varepsilon}$.}} In this step, we deduce the basic $L^2$ estimates of $h^{\varepsilon}$. For results in this direction, we have
\begin{align}\label{L2h}
    \,\frac{\mathrm{d}}{\mathrm{d} t}\big(\varepsilon\|h^{\varepsilon}\|^2\big)
    +\delta\|({\bf I}-{\bf P})[h^{\varepsilon}]\|_D^2    \lesssim\,\epsilon_0\|f^{\varepsilon}\|^2+\varepsilon^{2k+3}.
\end{align}

To prove this, we first take the $L^2$ inner product with $h^{\varepsilon}$ and \eqref{mh} and use \eqref{coL} to obtain
\begin{align}\label{L2h1}
&\frac{1}{2}\frac{\mathrm{d}}{\mathrm{d} t}\|h^{\varepsilon}\|^2
    +\frac{\delta}{\varepsilon}\|({\bf I}-{\bf P})[h^{\varepsilon}]\|^2_D\nonumber\\
    \leq&\;\frac{1}{\varepsilon}\big|\big\langle \mathcal{L}_d[ h^{\varepsilon}], h^{\varepsilon}\big\rangle\big|
    +\varepsilon^{k-1}\big|\big\langle\Gamma ( h^{\varepsilon},
    h^{\varepsilon} ), h^{\varepsilon}\big\rangle\big|\\
    &+\sum_{i=1}^{2k-1}\varepsilon^{i-1}\big|\big\langle[\Gamma(\mu^{-\frac{1}{2}}F_i,h^{\varepsilon})+\Gamma(
 h^{\varepsilon}, \mu^{-\frac{1}{2}} F_i)], h^{\varepsilon}\big\rangle\big|+\big|\big\langle \CP_1, h^{\varepsilon}\big\rangle\big|.\nonumber
\end{align}

For the first term on the R.H.S. of \eqref{L2h1}, by \eqref{Ldh}, we have
\begin{align*}
\frac{1}{\varepsilon}\big|\big\langle \mathcal{L}_d[ h^{\varepsilon}], h^{\varepsilon}\big\rangle\big|\lesssim \frac{\epsilon_0}{\varepsilon} \big(\|({\bf I}-{\bf P})[h^{\varepsilon}]\|^2_D +\|f^{\varepsilon}\|^2\big).
\end{align*}
Next, note that for $k\geq 3$,
\begin{align}\label{hf2b}
\varepsilon^{k-1}\|h^{\varepsilon}\|_{H^2}\lesssim 1
\end{align}
by \eqref{aps-ld}, using this and \eqref{GL}, we see that the second term on the R.H.S. of \eqref{L2h1} can be dominated by
\begin{align*}
\varepsilon^{k-1}\big|\big\langle\Gamma ( h^{\varepsilon},
    h^{\varepsilon} ), h^{\varepsilon}\big\rangle\big|
    \lesssim&\varepsilon^{k-1}\int_{{\mathbb R}^3}|h^{\varepsilon}|_{L^2}|h^{\varepsilon}|_D|({\bf I}-{\bf P})[h^{\varepsilon}]|_D\, \nonumber\\
    \lesssim&\varepsilon^{k-1}\|h^{\varepsilon}\|_{H^2}\Big(\|({\bf I}-{\bf P})[h^{\varepsilon}]\|_D+\|{{\bf P}}[h^{\varepsilon}]\|_D\Big)\|({\bf I}-{\bf P})[h^{\varepsilon}]\|_D\nonumber\\
    \lesssim &\,\|({\bf I}-{\bf P})[h^{\varepsilon}]\|_D^2+\|{{\bf P}}[h^{\varepsilon}]\|^2\lesssim \|({\bf I}-{\bf P})[h^{\varepsilon}]\|_D^2+\|f^{\varepsilon}\|^2 .\nonumber
\end{align*}
It should be pointed out that we always bound ${{\bf P}}[h^{\varepsilon}]$ by $f^{\varepsilon}$ so that the $\vps-$singularity of ${{\bf P}}[h^{\varepsilon}]$ can be avoid.

For the third term on the R.H.S. of \eqref{L2h1}, from \eqref{GL} and \eqref{Fn-es}, it follows that
\begin{align*}
   &\sum_{i=1}^{2k-1}\varepsilon^{i-1}\big|\big\langle[\Gamma(\mu^{-\frac{1}{2}}F_i,h^{\varepsilon})+\Gamma(
 h^{\varepsilon}, \mu^{-\frac{1}{2}} F_i)], h^{\varepsilon}\big\rangle\big|\\
    \lesssim&\, \|({\bf I}-{\bf P})[h^{\varepsilon}]\|_D^2+\varepsilon\big(\|{f^{\varepsilon}}\|_D^2+\|{f^{\varepsilon}}\|^2\big)\\
    \lesssim& \frac{o(1)}{\varepsilon}\|({\bf I}-{\bf P})[h^{\varepsilon}]\|_D^2+\varepsilon\|f^{\varepsilon}\|^2\nonumber.
\end{align*}

Finally, for the last term on the R.H.S. of \eqref{L2h1},
using Cauchy's inequality, we have
\begin{align*}
    \big|\big\langle \CP_1, h^{\varepsilon}\big\rangle\big|\lesssim& \frac{o(1)}{\varepsilon}\|({\bf I}-{\bf P})[h^{\varepsilon}]\|_D^2+\varepsilon\sum_{\substack{i+j\geq 2k+1\\2\leq i,j\leq2k-1}}\varepsilon^{2(i+j-k)}\|F_i\|^2_{H^2}\|F_j\|_D^2\\
    \lesssim& \frac{o(1)}{\varepsilon}\|({\bf I}-{\bf P})[h^{\varepsilon}]\|_D^2+\varepsilon^{2k+3}.\nonumber
\end{align*}

Substituting the above estimates into \eqref{L2h1} and multiplying the resulting inequality by $\varepsilon$ give \eqref{L2h}.

\vskip 0.2cm
\noindent\underline{{\it Step 2. Higher order energy estimates of $h^{\varepsilon}$.}} In this step, we shall deduce the $H^2$ estimate of $h^\vps$, our results in this direction can be summarized as follows
\begin{align}\label{H1h}
    \,\frac{\mathrm{d}}{\mathrm{d} t}\big(\varepsilon^2\|\nabla_xh^{\varepsilon}\|^2\big)
    +\varepsilon\delta\|\nabla_x({\bf I}-{\bf P})[h^{\varepsilon}]\|_D^2    \lesssim\,\epsilon_0\varepsilon\|({\bf I}-{\bf P})[h^{\varepsilon}]\|_D^2+\epsilon_0\varepsilon\|f^{\varepsilon}\|^2_{H^1}+\varepsilon^{2k+5}
\end{align}
and
\begin{align}\label{H2h}
    \,\frac{\mathrm{d}}{\mathrm{d} t}\big(\varepsilon^3\|\nabla_x^2h^{\varepsilon}\|^2\big)
    +\varepsilon^2\delta\|\nabla_x^2({\bf I}-{\bf P})[h^{\varepsilon}]\|_D^2    \lesssim\,\epsilon_0\varepsilon^2\Big(\|({\bf I}-{\bf P})[h^{\varepsilon}]\|^2_{H^1_D}+\|f^{\varepsilon}\|^2_{H^2}\Big)+\varepsilon^{2k+6}.
\end{align}

For brevity, we only prove \eqref{H2h}, since \eqref{H1h} can be shown similarly. To do this,
applying $\partial_x^{\alpha} $ with $|\al|=2$ to \eqref{mh} and using \eqref{coL}, one has
\begin{align}\label{H2h1}
&\frac{1}{2}\frac{\mathrm{d}}{\mathrm{d} t}\|\nabla_x^2h^{\varepsilon}\|^2
    +\frac{\delta}{\varepsilon}\|\nabla_x^2({\bf I}-{\bf P})[h^{\varepsilon}]\|^2_D \nonumber\\
    \leq&\;\frac{1}{\varepsilon}\big|\big\langle \nabla_x^2\mathcal{L}_d[ h^{\varepsilon}], \nabla_x^2h^{\varepsilon}\big\rangle\big|
    +\varepsilon^{k-1}\big|\big\langle \nabla^2_x\Gamma ( h^{\varepsilon},
    h^{\varepsilon} ), \nabla^2_xh^{\varepsilon}\big\rangle\big|\\
    &+\sum_{i=1}^{2k-1}\varepsilon^{i-1}\big|\big\langle[\nabla_x^2\Gamma(\mu^{-\frac{1}{2}}F_i,h^{\varepsilon}) +\nabla_x^2\Gamma(h^{\varepsilon}, \mu^{-\frac{1}{2}} F_i)], \nabla_x^2h^{\varepsilon}\big\rangle\big|+\big|\big\langle \nabla_x^2\CP_1, \nabla_x^2h^{\varepsilon}\big\rangle\big|.\nonumber
\end{align}

We now turn to compute the the R.H.S. of \eqref{H2h1} term by term.
For the first term on the R.H.S. of \eqref{H2h1},
in light of \eqref{Ldh}, one has
\begin{align*}
\frac{1}{\varepsilon}\big|\big\langle \nabla_x^2\mathcal{L}_d[ h^{\varepsilon}], \nabla_x^2h^{\varepsilon}\big\rangle\big|
\lesssim& \frac{\epsilon_0}{\varepsilon} \big(\|({\bf I}-{\bf P})[h^{\varepsilon}]\|^2_{H^2_D} +\|f^{\varepsilon}\|^2_{H^2}\big).
\end{align*}

For second term on the R.H.S. of \eqref{H2h1},
by  \eqref{pdhf}, \eqref{GL}, \eqref{hf2b} and Sobolev's inequalities, we have
\begin{align*}
&\varepsilon^{k-1}\big|\big\langle\nabla_x^2\Gamma ( h^{\varepsilon},
    h^{\varepsilon} ), \nabla_x^2h^{\varepsilon}\big\rangle\big|\\
           \lesssim&\varepsilon^{k-1}\int_{{\mathbb R}^3}\Big(|\nabla_x^2h^{\varepsilon}|_{L^2}|h^{\varepsilon}|_D|+|h^{\varepsilon}|_{L^2}|\nabla_x^2h^{\varepsilon}|_D|
        +|\nabla_xh^{\varepsilon}|_{L^2}|\nabla_xh^{\varepsilon}|_D\Big)|({\bf I}-{\bf P})[\nabla_x^2h^{\varepsilon}]|_D\, \nonumber\\
    \lesssim&\varepsilon^{k-1}\|h^{\varepsilon}\|_{H^2}\|\nabla_x^2h^{\varepsilon}\|_D\|({\bf I}-{\bf P})[\nabla_x^2h^{\varepsilon}]\|_D\nonumber\\
    \lesssim &\,\|({\bf I}-{\bf P})[\nabla_x^2h^{\varepsilon}]\|_D^2+\|{{\bf P}}[\nabla_x^2h^{\varepsilon}]\|^2\lesssim \|\nabla_x^2({\bf I}-{\bf P})[h^{\varepsilon}]\|_D^2+\|f^{\varepsilon}\|^2_{H^2} .\nonumber
\end{align*}

Similarly, for the third and fourth term on the R.H.S. of \eqref{H2h1},
by \eqref{GL} and \eqref{Fn-es}, we obtain
\begin{align*}
\sum_{i=1}^{2k-1}\varepsilon^{i-1}\big|\big\langle&[\nabla_x^2\Gamma(\mu^{-\frac{1}{2}}F_i,h^{\varepsilon})+\nabla_x^2\Gamma(
 h^{\varepsilon}, \mu^{-\frac{1}{2}} F_i)], \nabla_x^2h^{\varepsilon}\big\rangle\big|\\
     \lesssim&\, \frac{o(1)}{\varepsilon}\|({\bf I}-{\bf P})[ \nabla_x^2h^{\varepsilon}]\|_D^2+\varepsilon\big(\|{h^{\varepsilon}}\|_{H^2_D}^2+\|{h^{\varepsilon}}\|^2_{H^2}\big)\\
      \lesssim& \frac{o(1)}{\varepsilon}\|({\bf I}-{\bf P})[\nabla_x^2h^{\varepsilon}]\|_D^2+\varepsilon\|({\bf I}-{\bf P})[h^{\varepsilon}]\|_{H^1_D}^2+\varepsilon\|f^{\varepsilon}\|^2_{H^2}\nonumber
\end{align*}
and
\begin{align*}
    \big|\big\langle \nabla_x^2\CP_1, \nabla_x^2h^{\varepsilon}\big\rangle\big|\lesssim& \frac{o(1)}{\varepsilon}\|\nabla_x^2({\bf I}-{\bf P})[h^{\varepsilon}]\|_D^2+\varepsilon^{2k+3}.\nonumber
\end{align*}

Plugging the above estimates into \eqref{H2h1} and multiplying the resulting inequality by $\varepsilon^3$ give \eqref{H2h}.

Finally, \eqref{H2h-sum} follows from \eqref{L2h}, \eqref{H1h} and \eqref{H2h}, this finishes the proof of Proposition \ref{H2xh-prop}.
\end{proof}

%%%%%%%%%%%%%%%%%%%%%%%%%%%%%%%%%%%%%%%%%%%%%%%%%%%%%%%%%%%%%%%%%%%%%%%%%%%%%%%%%%
\subsection{The Proof of the Theorem \ref{resultLB}} \label{Sec:thmLB}
%%%%%%%%%%%%%%%%%%%%%%%%%%%%%%%%%%%%%%%%%%%%%%%%%%%%%%%%%%%%%%%%%%%%%%%%%%%%%%%%%%
We are now in a position to complete the Proof of Theorem~\ref{resultLB}. We will only show that the energy estimates \eqref{thm1} holds and omit the details of proof of the non-negativity of $F^{\varepsilon}$ since it is the same as the proof in \cite{Ouyang-Wu-Xiao-arxiv-2022-rL, Ouyang-Wu-Xiao-arxiv-2022-rLM}.

To this end, we first get from Propositions \ref{H2xf} and \ref{H2xh-prop} that
\begin{align}\label{eg-ld-sum}
\frac{\mathrm{d}}{\mathrm{d} t}\mathcal{E}(t)
&+\frac{\delta}{2}\sum_{i=0}^2 \varepsilon^i\Big(\frac{1}{\varepsilon}\|\nabla_x^i({\bf I}-{\bf P}_{\mathbf{M}})[f^{\varepsilon}]\|^2_D+\|\nabla_x^i({\bf I}-{\bf P})[h^{\varepsilon}]\|^2_D\Big)\\
\lesssim & \epsilon_0\Big(\varepsilon^2\|f^{\varepsilon}\|^2_{H^2}+\varepsilon\|f^{\varepsilon}\|^2_{H^1}+\|f^{\varepsilon}\|^2\Big)
+C_{\epsilon_0}\exp\left(\frac{-\epsilon_0}{8RT^2_c\sqrt{\varepsilon}}\right)
     \|h^{\varepsilon}\|^2_{H^2}+\varepsilon^{2k+3},\notag
\end{align}
where $\mathcal{E}(t)$ is defined in \eqref{eg-le}.

Next, by choosing $\varepsilon_0>0$ sufficiently small such that
$$C_{\epsilon_0}\exp\left(\frac{-\epsilon_0}{8RT^2_c\sqrt{\varepsilon_0}}\right)\lesssim \varepsilon_0^3,$$
then for $\varepsilon\leq \varepsilon_0$ and $t\in[0,t_e]$, we can get from Gronwall's inequality and \eqref{eg-ld-sum} that
\begin{align*}
&\mathcal{E}(t)+\int_0^t\mathcal{D}(s)\, d s\lesssim \mathcal{E}(0)+\varepsilon^{2k+3},\ t\in[0,t_e],
\end{align*}
where $\mathcal{D}(t)$ is given in \eqref{dn-le}. The proof of Theorem~\ref{resultLB} is then complete.

%%%%%%%%%%%%%%%%%%%%%%%%%%%%%%%%%%%%%%%%%%%%%%%%%%%%%%%%%%%%%%%%%%%%%%%%%%%%%%%%%%
\section{Hilbert expansion for VML system} \label{h-VML}
In this section, we shall justify the validity of the Hilbert expansion \eqref{expan} for the VML system \eqref{main1}.
As discussed in Section \ref{H-LD}, the proof relies on two ingredients, the first one is to determine the coefficients $[F_n,E_n,B_n]$, which is given by the iteration equations \eqref{expan2}, the other is to prove the wellposedness of the remainder $[f^\vps,E^\vps_R,B^\vps_R]$ via the equations \eqref{VMLf} with prescribed initial data
\begin{align*}
f^\vps(0,x,v)=&f^\vps_0(x,v)=\vps^{-k}\FM^{-\frac{1}{2}}\left\{F^\vps_0(x,v)-\sum\limits_{n=0}^{2k-1}\vps^kF_{n,0}(x,v)\right\},\\
E_R^\vps(0,x)=&E^\vps_{R,0}(x)=\vps^{-k}\left\{E^\vps_0(x)-\sum\limits_{n=0}^{2k-1}\vps^kE_{n,0}(x)\right\},\\
B_R^\vps(0,x)=&B^\vps_{R,0}(x)=\vps^{-k}\left\{B^\vps_0(x)-\sum\limits_{n=0}^{2k-1}\vps^kB_{n,0}(x)\right\}.
\end{align*}
Recall that $h^\vps(0,x,v)=h^\vps_0(x,v)=\mu^{-\frac{1}{2}}\FM_{[\rho_0,u_0,T_0]}^{\frac{1}{2}}f^\vps_0(x,v).$

%%%%%%%%%%%%%%%%%%%%%%%%%%%%%%%%%%%%%%%%%%%%%%%%%%%%%%%%%%%%%%%%%%%%%%%%%%%%%%%%%%

%\setcounter{equation}{0}
%%%%%%%%%%%%%%%%%%%%%%%%%%%%%%%%%%%%%%%%%%%%%%%%%%%%%%%%%%%%%%%%%%%%%%%%%%%%%%%%%%
\subsection{The coefficients}
This subsection is devoted to determining the coefficients $F_n$ for all $1\leq n\leq2k-1$ in the expansion \eqref{expan}.
Our results are the following
\begin{lemma}\label{em-Fn-lem}
Under the condition \eqref{Fn-id}, for $q\in(0,1)$, $\ell\geq0$ and $1\leq n\leq 2k-1$, there exists $C>0$ such that
\begin{align}
\sum\limits_{\al_0+|\al|+|\beta|\leq N+2k-2n+2}\left|\pa_t^{\al_0}\pa_\bet^\al \left(\frac{F_n}{\sqrt{\FM}}\right)\right|\leq C\eps_1\lag v\rag^{-\ell+n+|\beta|}\FM^{\frac{q}{2}}(1+t)^{n}\label{em-Fn-es}
\end{align}
and
\begin{align}
\sum\limits_{\al_0+|\al|\leq N+2k-2n+2}|\pa_t^{\al_0}\pa^\al [E_n,B_n]|\leq C\eps_1(1+t)^{n}.\label{em-EB-es}
\end{align}
%where $C>0$ and depends on $N$ and $q$.
\end{lemma}
\begin{proof}
 The main argument applied here is originally used in
\cite{Guo-CPAM-2006} for the case of cutoff Boltzmann equation and Landau equation around global Maxwellian.
Setting $F_n=\sqrt{\FM}f_n$ and recalling the definition \eqref{fn-mac-def}, we write
$$
{\bf P}_{\mathbf{M}}[f_n]=\frac{\rho_n}{\sqrt{\rho}} \chi_0+\sum\limits_{i=1}^3\frac{1}{\sqrt{R\rho T}}u_n^i\chi_i+\frac{T_n}{\sqrt{6\rho}}\chi_4.
$$

Next, from \eqref{expan2}, one sees that $f_n$ satisfies the following iterative equations
\begin{align}\notag
L_{\FM} f_1=-\FM^{-\frac{1}{2}}\{\pa_t\FM+v\cdot\na_x\FM\}+\FM^{-\frac{1}{2}}(E+v\times B)\cdot\na_v \FM,%\label{em-mi-f1}\
\end{align}
\begin{align}
&\pa_tf_1+v\cdot\na_xf_1+f_1\FM^{-\frac{1}{2}}(\pa_t+v\cdot\na_x)\FM^{\frac{1}{2}}
-(E+v\times B)\cdot\na_v f_1\notag\\&-f_1\FM^{-\frac{1}{2}}(E+v\times B)\cdot\na_v\FM^{\frac{1}{2}}
-\FM^{-\frac{1}{2}}(E_1+v\times B_1)\cdot\na_v\FM
+\CL_\FM f_2=\Ga_\FM(f_1,f_1),\notag
\end{align}
\begin{align}
\cdots\cdots\cdots\cdots\cdots,\notag
\end{align}
\begin{align}\label{em-fn-eq}
&\pa_tf_n+v\cdot\na_xf_n+f_n\FM^{-\frac{1}{2}}(\pa_t+v\cdot\na_x)\FM^{\frac{1}{2}}
-(E+v\times B)\cdot\na_v f_n\notag\\&-f_n\FM^{-\frac{1}{2}}(E+v\times B)\cdot\na_v\FM^{\frac{1}{2}}
-\FM^{-\frac{1}{2}}(E_n+v\times B_n)\cdot\na_v\FM
\notag\\&-{\bf 1}_{n>1}\sum\limits_{i+j=n\atop{i,j\geq1}}(E_i+v\times B_i)\cdot\na_v f_j
-{\bf 1}_{n>1}\sum\limits_{i+j=n\atop{i,j\geq1}}f_j\FM^{-\frac{1}{2}}(E_i+v\times B_i)\cdot\na_v\FM^{\frac{1}{2}}
+\CL_\FM f_{n+1}\notag\\=&\sum\limits_{i+j=n+1\atop{i,j\geq1,i\neq j}}\{\Ga_{\FM}(f_i,f_j)+\Ga_{\FM}(f_j,f_i)\}
+{\bf 1}_{n+1=2\Z}\Ga_\FM(f_{\frac{n+1}{2}},f_{\frac{n+1}{2}}),
\end{align}
\begin{align}
\cdots\cdots\cdots\cdots\cdots,\notag
\end{align}
\begin{align*}
&\pa_tf_{2k-2}+v\cdot\na_xf_{2k-2}+f_{2k-2}\FM^{-\frac{1}{2}}(\pa_t+v\cdot\na_x)\FM^{\frac{1}{2}}
-(E+v\times B)\cdot\na_v f_{2k-2}\notag\\&-f_{2k-2}\FM^{-\frac{1}{2}}(E+v\times B)\cdot\na_v\FM^{\frac{1}{2}}
-\FM^{-\frac{1}{2}}(E_{2k-2}+v\times B_{2k-2})\cdot\na_v\FM
\notag\\&-\sum\limits_{i+j={2k-2}\atop{i,j\geq1}}(E_i+v\times B_i)\cdot\na_v f_j
-\sum\limits_{i+j={2k-2}\atop{i,j\geq1}}f_j\FM^{-\frac{1}{2}}(E_i+v\times B_i)\cdot\na_v\FM^{\frac{1}{2}}
+\CL_\FM f_{2k-1}
\notag\\=&\sum\limits_{i+j=2k-1\atop{i,j\geq1}}\Ga_{\FM}(f_i,f_j),
\end{align*}
\begin{align}
\pa_tf_{2k-1}&+v\cdot\na_xf_{2k-1}+f_{2k-1}\FM^{-\frac{1}{2}}(\pa_t+v\cdot\na_x)\FM^{\frac{1}{2}}
-(E+v\times B)\cdot\na_v f_{2k-1}\notag\\&-f_{2k-1}\FM^{-\frac{1}{2}}(E+v\times B)\cdot\na_v\FM^{\frac{1}{2}}
-\FM^{-\frac{1}{2}}(E_{2k-1}+v\times B_{2k-1})\cdot\na_v\FM
\notag\\&-\sum\limits_{i+j={2k-1}\atop{i,j\geq1}}(E_i+v\times B_i)\cdot\na_v f_j
-\sum\limits_{i+j={2k-1}\atop{i,j\geq1}}f_j\FM^{-\frac{1}{2}}(E_i+v\times B_i)\cdot\na_v\FM^{\frac{1}{2}}
\notag\\=&\sum\limits_{i+j=2k}\Ga_\FM(f_i,f_j).\notag
\end{align}

To prove \eqref{em-Fn-es} and \eqref{em-EB-es}, our purpose now is to show that
\begin{align}
&\sum\limits_{|\al|+|\beta|\leq N+4k-2n-2}\left|\lag v\rag^{\ell-n-|\beta|}\FM^{-\frac{q}{2}}\pa_\bet^\al f_n\right|
+\sum\limits_{\al_0+|\al|\leq N+2k-2n-2}\left|\pa_t^{\al_0}\pa^\al [E_n,B_n]\right|\nonumber\\
\leq& C\eps_{1}(1+t)^{n},\label{em-fn-es}
\end{align}
holds for $N\geq2$.

In fact, \eqref{em-fn-es} follows from
\begin{align}
&\sum\limits_{|\al|+|\beta|\leq N+4k-2n+2}\left\|\lag v\rag^{\ell-n-|\beta|}\FM^{-\frac{q'}{2}}\pa_{\bet}^{\al} f_n\right\|_D
+\sum\limits_{\al_0+|\al|\leq N+2k-2n+2}\left|\pa_t^{\al_0}\pa^\al [E_n,B_n]\right|\nonumber\\
\leq& C\eps_1(1+t)^n\label{em-w-fn-es}
\end{align}
with $0<q<q'<1$, because for $0<\eta\ll 1$, we have the following estimates
\begin{align}
\left\|\FM^{-\frac{q'-\eta}{2}}\pa_{\bet}^{\al} f_n\right\|\leq C\left\|\FM^{-\frac{q'}{2}}\pa_{\bet}^{\al} f_n\right\|_D\notag
\end{align}
and
\begin{align*}
&\sum\limits_{|\al|+|\beta|\leq m}\left|\lag v\rag^{\ell-n-|\beta|}\FM^{-\frac{q}{2}}\pa_\bet^\al f_n\right|^2\\
\leq& C\sum\limits_{|\al|+|\beta|\leq m+4}\left\|\pa_\bet^\al \{\lag v\rag^{\ell-n-|\beta|}\FM^{-\frac{q}{2}}f_n\}\right\|^2
\notag\\
&+C\sum\limits_{|\al|+|\beta|\leq m+4\atop{\bar{\al}+\bar{\beta}>0}}\left\|\pa_{\bar{\beta}}^{\bar{\al}}\{\lag v\rag^{\ell-n-|\beta|}\FM^{-\frac{q}{2}}\}\pa_{\bet-\bar{\bet}}^{\al-\bar{\al}} f_n\right\|^2
\notag\\
\leq&C\sum\limits_{|\al|+|\beta|\leq m+4}\left\|\lag v\rag^{\ell-n-|\beta|}\FM^{-\frac{q'-\eta}{2}}\pa_{\bet}^{\al} f_n\right\|^2\notag
\end{align*}
for $l\geq|\beta|$ and $q<q'-\eta.$

%For brevity, in the rest of this section we still denote $[q',\beta',\al']$ by $[q,\beta,\al]$.
We now verify \eqref{em-w-fn-es} by the method of induction. To do so, the macroscopic and microscopic parts of $f_n$ should be estimated separately.
For the macroscopic component ${\bf P}_{\mathbf{M}}[f_n]$, taking the following velocity moments
 $$\FM^{1/2},(v-u)\FM^{1/2},\left(\frac{|v-u|^2}{RT}-3\right)\FM^{1/2}$$
with {$1\leq i,j\leq 3$}  for the equation \eqref{em-fn-eq}, respectively, one has
\begin{align}\label{em-abc-eq}
\left\{
\begin{array}{lll}
\pa_t\rho_n+\na_x\cdot u_n+\na_x\cdot(\rho_n u)=0,\\[2mm]
\pa_tu_n+\na_x[RT(\rho_n+2T_n)]+\na_x\cdot(u\otimes u_n)+(\pa_t u+u\cdot\na_x  u)\rho_n+\na_xu \cdot u_n
\\[2mm]\qquad+\rho_n E+u_n\times B+\rho_n u\times B+\rho E_n+\rho u\times B_n\\[2mm]\qquad
+\sum\limits_{i+j=n\atop{i,j\geq1}}(\rho_i E_j+u_i\times B_j+\rho_i u\times B_j)
\\[2mm]\quad=-\na_x\cdot\left(\left[(v-u)\otimes(v-u)-\frac{|v-u|^2}{2}\FI\right]\FM^{1/2},({\bf I-P}_{\mathbf{M}})[f_n]\right),\\[2mm]
\pa_tT_n+2\na_x\cdot u_n+\na_x\cdot(uT_n)+\frac{\pa_tT}{T}(3\rho_n+T_n)+\frac{5\na_xT\cdot u_n}{T}
+\frac{\na_xT\cdot u (3\rho_n+T_n)}{T}\\[2mm]
\quad=-\na_x\cdot\left((v-u)\left(\frac{|v-u|^2}{RT}-5\right)\FM^{1/2},({\bf I-P}_{\mathbf{M}})[f_n]\right),\\[2mm]
\pa_t E_n-\na_x\times B_n=u_n+\rho_n u,\\[2mm]
\pa_t B_n+\na_x\times E_n=0,\\[2mm]
\na_x\cdot E_n=-\rho_n,\ \na_x\cdot B_n=0,
\end{array}\right.
\end{align}
where $1\leq n\leq2k-1$ and $\FI$ stands for the unit $3\times 3$ matrix.
Then, by standard energy estimates, it follows from \eqref{em-abc-eq} and \eqref{em-decay} that
\begin{align}\label{em-ma-es}
&\frac{d}{dt}\sum\limits_{\al_0+|\al|\leq m}\|\pa_t^{\al_0}\pa^\al[\rho_n,u_n,T_n,E_n,B_n]\|^2\notag\\
\leq& C(1+t)^{-p_0}\sum\limits_{\al_0+|\al|\leq m}\|\pa_t^{\al_0}\pa^\al[\rho_n,u_n,T_n,E_n,B_n]\|^2
\notag\\&+C\sum\limits_{\al_0+|\al|\leq m+1}\|\pa_t^{\al_0}\pa^\al({\bf I-P}_{\mathbf{M}})[f_n]\|_D\sum\limits_{|\al|\leq m}\|\pa_t^{\al_0}\pa^\al[\rho_n,u_n,T_n,E_n,B_n]\|
%+C\sup\limits_{0\leq t\leq t_e}\sum\limits_{|\al|\leq m+1}\|\pa^\al\FP^\FM_1f_n\|^2_D,
\end{align}
where $m>0$ is finite.

To close our estimate, we now turn to estimate the microscopic part $({\bf I-P}_{\mathbf{M}})[f_n]$.
For $n=1$, in view of the first equation for $\varepsilon^0$ in \eqref{expan2}, one has for $q'\in(0,1)$
\begin{align}
&\left\lag\FM^{-q'/2}\pa_t^{\al_0}\pa_\beta^\al L_{\FM}[({\bf I-P}_{\mathbf{M}})[f_1]],\lag v\rag^{2(\ell-1-|\beta|)}\FM^{-q'/2}\pa_t^{\al_0}\pa_\beta^\al({\bf I-P}_{\mathbf{M}})[f_1]\right\rag\notag\\
=& -\left\lag\FM^{-q'/2}\pa_t^{\al_0}\pa_\beta^\al
\left[\FM^{-\frac{1}{2}}\{\pa_t\FM+v\cdot\na_x\FM\}\right],\lag v\rag^{2(\ell-1-|\beta|)}\FM^{-q'/2}\pa_t^{\al_0}\pa_\beta^\al({\bf I-P}_{\mathbf{M}})[f_1]\right\rag,\notag
\end{align}
from this and \eqref{wLLM} as well as \eqref{em-decay}, we get for $\beta>0$ that
\begin{align}%\label{em-cof1-1}
&\left\|\lag v\rag^{\ell-1-|\beta|}\mathbf{M}^{-q'/2}\pa_t^{\al_0}\partial^{\alpha}_{\beta}({\bf I-P}_{\mathbf{M}})[f_1]\right\|_{D}^2
\notag\\
&-\left(\eta+C\epsilon_1\right)\sum_{\bar{\al}_0+|\bar{\alpha}|\leq \al_0+|\alpha|}\sum_{|\bar{\beta}|=|\beta|}\left\|\lag v\rag^{\ell-1-|\bar{\beta}|}\mathbf{M}^{-q'/2}\pa_t^{\bar{\al}_0}\partial^{\bar{\alpha}}_{\bar{\beta}}({\bf I-P}_{\mathbf{M}})[f_1]\right\|_{D}^2\nonumber\\
&-C(\eta)\sum\limits_{\bar{\al}_0\leq\al_0}\sum_{\bar{\alpha}\leq \alpha}\sum_{|\bar{\beta}|<|\beta|}
\left\|\lag v\rag^{\ell-1-|\bar{\beta}|}\mathbf{M}^{-q'/2}\pa_t^{\bar{\al}_0}\partial^{\bar{\alpha}}_{\bar{\beta}}({\bf I-P}_{\mathbf{M}})[f_1]\right\|^2_{D}
\leq C\eps_1^2,\notag
\end{align}
and for $\beta=0$
\begin{align}%\label{em-cof1-2}
&\left\|\lag v\rag^{\ell-1}\FM^{-q'/2}\pa_t^{\al_0}\pa^\al({\bf I-P}_{\mathbf{M}})[f_1]\right\|_D^2\notag\\&
-C\epsilon_1{\bf 1}_{\al_0+|\al|>0}\sum\limits_{\bar{\al}_0\leq\al_0,\bar{\alpha}\leq \alpha,|\bar{\al}_0|+|\bar{\alpha}|<|\al_0|+|\alpha|}\left\|\lag v\rag^{\ell-1}\mathbf{M}^{-q'/2}\pa_t^{\bar{\al}_0}\partial^{\bar{\alpha}}({\bf I-P}_{\mathbf{M}})[f_1]\right\|_{D}^2
\notag\\
&-C\left\|\lag v\rag^{\ell-1}\FM^{-q'/2}\pa_t^{\al_0}\pa^{\al}({\bf I-P}_{\mathbf{M}})[f_1]\right\|_{L^2(\R^3\times B_C(\eta))}
\leq C\eps_1^2,\notag
\end{align}
and
\begin{align*}%\label{em-cof1-3}
\left\|\pa_t^{\al_0}\pa^\al({\bf I-P}_{\mathbf{M}})[f_1]\right\|_D^2
\leq C\eps_1^2.
%,\\
%{\bf 1}_{\al_0+|\al|>0}\sum\limits_{\bar{\al}_0\leq\al_0,\bar{\alpha}\leq \alpha,|\bar{\al}_0|+|\bar{\alpha}|<|\al_0|+|\alpha|}\left\|\lag v\rag^{\ell-1}\mathbf{M}^{-q'/2}\pa_t^{\bar{\al}_0}\partial^{\bar{\alpha}}({\bf I-P}_{\mathbf{M}})[f_1]\right\|_{D}^2
%\leq C\eps_1^2.
\end{align*}

Putting the above estimates together, we conclude that
\begin{align}\label{em-cof1-s}
\sum\limits_{\al_0+|\al|+|\beta|\leq m+4k-3}\left\|\lag v\rag^{\ell-1-|\beta|}\FM^{-q'/2}\pa_t^{\al_0}\pa_\beta^\al({\bf I-P}_{\mathbf{M}})[f_1]\right\|_D^2
\leq C\eps_1^2.
\end{align}
%for any $m\geq0.$
Thus \eqref{em-cof1-s} and \eqref{em-ma-es} with $n=1$ gives
\begin{align}\label{em-f1-es}
&\sum\limits_{\al_0+|\al|+|\beta|\leq m+4k-4}\left\|\lag v\rag^{\ell-1-|\beta|}\FM^{-q'/2}\pa_t^{\al_0}\pa_\beta^\al f_1(t)\right\|_D
+\sum\limits_{\al_0+|\al|\leq m+4k-4}\left|\pa_t^{\al_0}\pa^\al [E_1,B_1]\right|\notag\\
\leq& C\sum\limits_{\al_0+|\al|\leq m+4k-4}\left\|\pa_t^{\al_0}\pa^\al[\rho_1,u_1,T_1,E_1,B_1](0,x)\right\|+C\eps_1(1+t).
\end{align}

Now we assume that
\begin{align}%\label{em-fn-es-ass}
&\sum\limits_{\al_0+|\al|+|\beta|\leq m+4k-2n-2}\left\|\lag v\rag^{\ell-n-|\beta|}\FM^{-q'/2}\pa_t^{\al_0}\pa_\beta^\al f_n(t)\right\|_D\nonumber\\
&+\sum\limits_{\al_0+|\al|\leq m+4k-2n-2}\left|\pa_t^{\al_0}\pa^\al [E_n,B_n]\right|\notag\\
\leq& C\sum\limits_{\al_0+|\al|\leq m+4k-2n-2}\left\|\pa_t^{\al_0}\pa^\al[\rho_n,u_n,T_n,E_n,B_n](0,x)\right\|+C\eps_1(1+t)^n\notag
\end{align}
holds for $1\leq n\leq2k-2 $, then by repeating the argument used to deduce \eqref{em-cof1-s}, one has
\begin{align}%\label{em-cofn-s}
\sum\limits_{\al_0+|\al|+|\beta|\leq m+4k-2n-3}\left\|\lag v\rag^{\ell-n-|\beta|-1}\FM^{-q'/2}\pa_t^{\al_0}\pa_\beta^\al({\bf I-P}_{\mathbf{M}})[f_{n+1}](t)\right\|_D
\leq C\eps_1(1+t)^n,\notag
\end{align}
this together with \eqref{em-ma-es} yields
\begin{align}\label{em-f1-es}
&\sum\limits_{\al_0+|\al|+|\beta|\leq m+4k-2n-4}\left\|\lag v\rag^{\ell-n-|\beta|-1}\FM^{-q'/2}\pa_t^{\al_0}\pa_\beta^\al f_{n+1}(t)\right\|_D\nonumber\\
&+\sum\limits_{\al_0+|\al|\leq m+4k-2n-4}\left|\pa_t^{\al_0}\pa^\al [E_{n+1},B_{n+1}]\right|\\
\leq& C\eps_1(1+t)^{n+1}.\nonumber
\end{align}

Finally, by taking $m=N+4$, \eqref{em-f1-es} with $0\leq n\leq 2k-2$ gives \eqref{em-w-fn-es}. Thus the proof of Lemma \ref{em-Fn-lem} is complete.
\end{proof}

\subsection{The remainder}
In this subsection, we intend to construct the wellposedness of the remainders $f^{\varepsilon}, E_R^{\varepsilon}, B_R^{\varepsilon}$ and $h^{\varepsilon}$ which satisfy the following equations
\begin{align}
\partial_tf^{\varepsilon}&+v\cdot\nabla_xf^{\varepsilon}+\frac{\big(E_R^{\varepsilon}+v \times B_R^{\varepsilon} \big) }{ RT}\cdot \big(v-u\big)\mathbf{M}^{\frac{1}{2}}\nonumber\\
&+\Big(E+v \times B \Big)\cdot\frac{v-u }{ 2RT}f^{\varepsilon}-\Big(E+v \times B \Big)\cdot\nabla_vf^{\varepsilon}+\frac{\mathcal{L}_{\mathbf{M}}[f^{\varepsilon}]}{\varepsilon}\nonumber\\
=&-\mathbf{M}^{-\frac{1}{2}}f^{\varepsilon}\Big[\partial_t+v\cdot\nabla_x-\Big(E+v \times B \Big)\cdot\nabla_v\Big]\mathbf{M}^{\frac{1}{2}}+\varepsilon^{k-1}\Gamma_{\mathbf{M}}(f^{\varepsilon},f^{\varepsilon})\nonumber\\
 &+\sum_{i=1}^{2k-1}
 \varepsilon^{i-1}\Big[\Gamma_{\mathbf{M}}(\mathbf{M}^{-\frac{1}{2}}F_i, f^{\varepsilon})+\Gamma_{\mathbf{M}}(f^{\varepsilon}, \mathbf{M}^{-\frac{1}{2}} F_i)\Big]+\varepsilon^k \Big(E_R^{\varepsilon}+v \times B_R^{\varepsilon}\Big)\cdot\nabla_vf^{\varepsilon}\nonumber\\
 &-\varepsilon^k \Big(E_R^{\varepsilon}+v \times B_R^{\varepsilon}\Big) \cdot\frac{v-u }{ 2RT}f^{\varepsilon}\label{VMLf-2}\\
 &+\sum_{i=1}^{2k-1}\varepsilon^i\Big[\Big(E_i+v \times B_i \Big)\cdot\nabla_vf^{\varepsilon}+\Big(E_R^{\varepsilon}+v \times B_R^{\varepsilon} \Big)\cdot\nabla_v F_i\Big]\nonumber\\
 &-\sum_{i=1}^{2k-1}\varepsilon^i\Big[\Big(E_i+v \times B_i \Big)\cdot\frac{\big(v-u\big)}{ 2RT}f^{\varepsilon}\Big]+\varepsilon^{k}\CQ_0,
\nonumber
\end{align}
\begin{align}
&\partial_tE_R^{\varepsilon}-\nabla_x \times B_R^{\varepsilon}=\int_{\mathbb R^3} v\mathbf{M}^{\frac{1}{2}}f^{\varepsilon} dv, \nonumber\\
 &\partial_t B_R^{\varepsilon}+ \nabla_x \times E_R^{\varepsilon}=0,\label{fM-2}\\
& \nabla_x\cdot E_R^{\varepsilon}=-\int_{\mathbb R^3}  \mathbf{M}^{\frac{1}{2}} f^{\varepsilon} dv, \qquad \nabla_x\cdot  B_R^{\varepsilon}=0,\nonumber
\end{align}
and
\begin{align}
\partial_th^{\varepsilon}&+v\cdot\nabla_xh^{\varepsilon}+\frac{\big(E_R^{\varepsilon}+v \times B_R^{\varepsilon} \big) }{ RT}\cdot \big(v-u\big)\mu^{-\frac{1}{2}}\mathbf{M}\nonumber\\
&+\frac{E\cdot v }{ 2RT_c}h^{\varepsilon}-\Big(E+v \times B \Big)\cdot\nabla_vh^{\varepsilon}+\frac{\mathcal{L}[h^{\varepsilon}]}{\varepsilon}\nonumber\\
 =&-\frac{\mathcal{L}_d[h^{\varepsilon}]}{\varepsilon}+\varepsilon^{k-1}\Gamma(h^{\varepsilon},h^{\varepsilon})+\sum_{i=1}^{2k-1}\varepsilon^{i-1}[\Gamma(\mu^{-\frac{1}{2}}F_i, h^{\varepsilon})+\Gamma(h^{\varepsilon}, \mu^{-\frac{1}{2}}F_i)]\nonumber\\
 &+\varepsilon^k \Big(E_R^{\varepsilon}+v \times B_R^{\varepsilon}\Big)\cdot\nabla_vh^{\varepsilon}
 -\varepsilon^k \frac{E_R^{\varepsilon}\cdot v}{ 2RT_c}h^{\varepsilon}\label{VMLh-2}\\
 &+\sum_{i=1}^{2k-1}\varepsilon^i\Big[\Big(E_i+v \times B_i \Big)\cdot\nabla_vh^{\varepsilon}+\Big(E_R^{\varepsilon}+v \times B_R^{\varepsilon} \Big)\cdot\nabla_v F_i\Big]-\sum_{i=1}^{2k-1}\varepsilon^i\frac{E_i \cdot v}{ 2RT_c}h^{\varepsilon}+\varepsilon^{k}\CQ_1,
\nonumber
\end{align}
together with the initial datum
\begin{align}
\left[f^\vps(0,x,v),E_R^\vps(0,x),B_R^\vps(0,x),h^\vps(0,x,v)\right]=\left[f^\vps_0(x,v),E^\vps_{R,0}(x),B^\vps_{R,0}(x), h^\vps_0(x,v)\right].\label{VML-id-pt}
\end{align}

The global existence of \eqref{VMLf-2}, \eqref{fM-2}, \eqref{VMLh-2} and \eqref{VML-id-pt} can be established by the local-in-time existence and the {\it a priori} energy estimate as well as the continuation argument. For the sake of simplicity, we only prove the {\it a priori} energy estimate \eqref{TVML1} under the {\it a priori} assumption
\begin{align}\label{aps-vml}
    \sup_{0\leq t\leq \vps^{-\frac{1}{3}}}\mathcal{E}(t)\lesssim \varepsilon^{-\frac{1}{2}}.
\end{align}

Due to the presence of the Lorentz force, the energy estimates for the VML system \eqref{main1} are much more complicate than those of Section \ref{H-LD} for the Landau equation \eqref{LE}. Before deducing the {\it a priori} energy estimate \eqref{TVML1}, we first give some key estimates in the following lemmas. The first one is concerned with the velocity growth caused by local Maxwellian $\FM_{[\rho,u,T]}$ and electromagnetic field without velocity derivatives.
\begin{lemma}\label{dxlm VML}
Assume that $f^{\varepsilon}, E_R^{\varepsilon}, B_R^{\varepsilon}$ and $h^{\varepsilon}$ are smooth solutions of the Cauchy problem \eqref{VMLf-2}, \eqref{fM-2}, \eqref{VMLh-2} and \eqref{VML-id-pt} for the VML system \eqref{main1} and satisfy \eqref{aps-vml}, then for $i= 1, 2$,  it holds that
\begin{align}\label{growdxi VML}
&\Big|\Big\langle \nabla_x^i\Big(f^{\varepsilon}\mathbf{M}^{-\frac{1}{2}}\big(\partial_t+v\cdot\nabla_x\big)\mathbf{M}^{\frac{1}{2}}\Big), 4\pi RT\nabla_x^if^{\varepsilon}\Big\rangle\Big|\\
    \lesssim&\,\epsilon_1(1+t)^{-p_0}\Big(\|f^{\varepsilon}\|^2_{H^i}+\sum_{j=0}^i\frac{1}{\varepsilon^{(i+1-j)}} \left\|\nabla_x^j({\bf I}-{\bf P}_{\mathbf{M}})[f^{\varepsilon}]\right\|_D^2 +C_{\epsilon_1}\exp\left(-\frac{\epsilon_1}{8RT^2_c\sqrt{\varepsilon}}\right)\|h^{\varepsilon}\|^2_{H^i}\Big),\nonumber
   \end{align}
\begin{align}\label{EBdxi}
   &\Big|\Big\langle \nabla_x^i\Big[\Big(E+v \times B \Big)\cdot\frac{v-u }{ T}f^{\varepsilon}\Big], 2\pi T \nabla_x^if^{\varepsilon}\Big\rangle\Big|\\
    \lesssim&\,\epsilon_1(1+t)^{-p_0}\Big(\|f^{\varepsilon}\|^2_{H^i}+\sum_{j=0}^i\frac{1}{\varepsilon^{(i+1-j)}} \left\|\nabla_x^j({\bf I}-{\bf P}_{\mathbf{M}})[f^{\varepsilon}]\right\|_D^2 +C_{\epsilon_1}\exp\left(-\frac{\epsilon_1}{8RT^2_c\sqrt{\varepsilon}}\right) \|h^{\varepsilon}\|^2_{H^i}\Big),\nonumber
   \end{align}
  \begin{align}\label{EBRdxi}
  &\varepsilon^k\Big|\Big\langle \nabla_x^i\Big[\big(E_R^{\varepsilon}+v \times B_R^{\varepsilon}\big) \cdot\frac{(v-u)}{ T}f^{\varepsilon}\Big], 2\pi T \nabla_x^if^{\varepsilon}\Big\rangle\Big|\\
    \lesssim&\,\varepsilon\Big(\|f^{\varepsilon}\|^2_{H^i}+\sum_{j=0}^i\frac{1}{\varepsilon^{(i+1-j)}} \left\|\nabla_x^j({\bf I}-{\bf P}_{\mathbf{M}})[f^{\varepsilon}]\right\|_D^2 +C_{\epsilon_1}\exp\left(-\frac{\epsilon_1}{8RT^2_c\sqrt{\varepsilon}}\right)\|h^{\varepsilon}\|^2_{H^i}\Big),\nonumber
   \end{align}
  and for $0\leq t\leq \vps^{-\frac{1}{3}}$,
   \begin{align}\label{EBndxi}
  &\sum_{n=1}^{2k-1}\varepsilon^n\Big|\Big\langle \nabla^i_x\Big[\Big(E_n+v \times B_n \Big)\cdot\frac{(v-u)}{ T}f^{\varepsilon}\Big], 2\pi T \nabla_x^if^{\varepsilon}\Big\rangle\Big|\\
    \lesssim&\,\varepsilon^{\frac{2}{3}}\Big(\|f^{\varepsilon}\|^2_{H^i}+\sum_{j=0}^i\frac{1}{\varepsilon^{(i+1-j)}} \left\|\nabla_x^j({\bf I}-{\bf P}_{\mathbf{M}})[f^{\varepsilon}]\right\|_D^2 +C_{\epsilon_1}\exp\left(-\frac{\epsilon_1}{8RT^2_c\sqrt{\varepsilon}}\right)\|h^{\varepsilon}\|^2_{H^i}\Big).\nonumber
   \end{align}
\end{lemma}
\begin{proof}  We only prove \eqref{EBndxi}, since \eqref{growdxi VML}, \eqref{EBdxi} and \eqref{EBRdxi} can be proved similarly. To this end, for $0\leq t\leq \vps^{-\frac{1}{3}}$, we get from \eqref{em-fn-es}  that
\begin{align*}
&\sum_{n=1}^{2k-1}\varepsilon^n\Big|\Big\langle \nabla_x^i\Big[\Big(E_n+v \times B_n \Big)\cdot\frac{(v-u)}{ T}f^{\varepsilon}\Big], 2\pi T \nabla_x^if^{\varepsilon}\Big\rangle\Big|\\
   \lesssim& \sum_{n=1}^{2k-1}\varepsilon^n(1+t)^{n}\|{\lag v\rag}f^{\varepsilon}\|^2_{H^i}\lesssim
    \varepsilon^{\frac{2}{3}}\|{\lag v\rag}f^{\varepsilon}\|^2_{H^i}.
   \end{align*}
This together with \eqref{growdxi} then gives \eqref{EBndxi}. This ends the proof of Lemma \ref{dxlm VML}.
\end{proof}

The following lemma is devoted to the convection terms also caused by Lorentz force, in particular, the extra velocity derivatives are involved in, which makes the estimates more delicate.
\begin{lemma}\label{dxvlm VML}
Assume that $f^{\varepsilon}, E_R^{\varepsilon}, B_R^{\varepsilon}$ and $h^{\varepsilon}$ are smooth solutions of the Cauchy problem \eqref{VMLf-2}, \eqref{fM-2}, \eqref{VMLh-2} and \eqref{VML-id-pt} for the VML system \eqref{main1} and satisfy \eqref{aps-vml}, then for $i= 1, 2$,  it holds that
  \begin{align}\label{EBvdxi}
   &\Big|\Big\langle \nabla_x^i\Big[\Big(E+v \times B \Big)\cdot\nabla_vf^{\varepsilon}\Big], 4\pi R T \nabla_x^if^{\varepsilon}\Big\rangle\Big|\nonumber\\
    \lesssim&\,\epsilon_1(1+t)^{-p_0}\Big[\|f^{\varepsilon}\|^2_{H^i} +\frac{1}{\varepsilon}\left\|\nabla_x^{i-1}({\bf I}-{\bf P}_{\mathbf{M}})[f^{\varepsilon}]\right\|_D^2\\
&+C_{\epsilon_1}\exp\left(-\frac{\epsilon_1}{8RT^2_c\sqrt{\varepsilon}}\right)
    \Big(\|h^{\varepsilon}\|^2_{H^{i-1}}+\|h^{\varepsilon}\|^2_{H^{i-1}_D}\Big)\Big],\nonumber
   \end{align}
   \begin{align}\label{EBRvdxi}
   &\varepsilon^k\Big|\Big\langle \nabla_x^i\Big[\big(E_R^{\varepsilon}+v \times B_R^{\varepsilon}\big) \cdot\nabla_vf^{\varepsilon}\Big], 4\pi RT \nabla_x^if^{\varepsilon}\Big\rangle\Big|\nonumber\\
    \lesssim&\,\varepsilon\Big[\|f^{\varepsilon}\|^2_{H^i}+\frac{1}{\varepsilon}\left\|\nabla_x^{i}({\bf I}-{\bf P}_{\mathbf{M}})[f^{\varepsilon}]\right\|_D^2\\
&+C_{\epsilon_1}\exp\left(-\frac{\epsilon_1}{8RT^2_c\sqrt{\varepsilon}}\right)
    \Big(\|h^{\varepsilon}\|^2_{H^{i}}+\|h^{\varepsilon}\|^2_{H^{i}_D}\Big)\Big],\nonumber
   \end{align}
  and
   \begin{align}\label{EBnvdxi}
   &\sum_{n=1}^{2k-1}\varepsilon^n\Big|\Big\langle \nabla_x^i\Big[\Big(E_n+v \times B_n \Big)\cdot\nabla_vf^{\varepsilon}+\Big(E_R^{\varepsilon}+v \times B_R^{\varepsilon} \Big)\cdot\nabla_v F_n\Big], 4\pi RT \nabla_x^if^{\varepsilon}\Big\rangle\Big|\nonumber\\
    \lesssim&\,\varepsilon^{\frac{2}{3}}\Big[\|f^{\varepsilon}\|^2_{H^i}+\|E_R^{\varepsilon}\|^2_{H^i} +\|B_R^{\varepsilon}\|^2_{H^i}+\frac{1}{\varepsilon}\left\|\nabla_x^{i-1}({\bf I}-{\bf P}_{\mathbf{M}})[f^{\varepsilon}]\right\|_D^2\\
&+C_{\epsilon_1}\exp\left(-\frac{\epsilon_1}{8RT^2_c\sqrt{\varepsilon}}\right)
    \Big(\|h^{\varepsilon}\|^2_{H^{i-1}}+\|h^{\varepsilon}\|^2_{H^{i-1}_D}\Big)\Big].\nonumber
   \end{align}
\end{lemma}
\begin{proof} We only prove \eqref{EBvdxi} and \eqref{EBRvdxi}, since \eqref{EBnvdxi} can be proved in the same way. By \eqref{em-decay}, Sobolev's inequalities and \eqref{aps-vml}, one has
\begin{align} \label{preEB}
&\Big|\Big\langle \nabla_x^i\left[\big(E+v \times B \big) \cdot \nabla_v f^{\varepsilon}\right],4\pi R T \nabla_x^if^{\varepsilon}\Big\rangle\Big|\notag\\
\lesssim& \Big(\|E\|_{W^{i,\infty}}+\|B\|_{W^{i,\infty}}\Big)\Big(\left\|{\lag v\rag}\nabla_v \nabla_x^{i-1}f^{\varepsilon}\right\|^2+\left\|\nabla_x^if^{\varepsilon}\right\|^2\Big)\\
\lesssim&\epsilon_1(1+t)^{-p_0}\Big(\left\|{\lag v\rag}\nabla_v \nabla_x^{i-1}f^{\varepsilon}\right\|^2+\left\|\nabla_x^if^{\varepsilon}\right\|^2\Big)\nonumber
\end{align}
and
\begin{align} \label{preEBR}
&\varepsilon^k\Big|\Big\langle \nabla_x^i\Big[\big(E_R^{\varepsilon}+v \times B_R^{\varepsilon}\big) \cdot\nabla_vf^{\varepsilon}\Big], 4\pi RT \nabla_x^if^{\varepsilon}\Big\rangle\Big|\notag\\
\lesssim& \varepsilon^k\Big(\|E_R^{\varepsilon}\|_{H^2}+\|B_R^{\varepsilon}\|_{H^2}\Big)\Big(\left\|{\lag v\rag}\nabla_v \nabla_x^{i}f^{\varepsilon}\right\|^2+\left\|\nabla_x^if^{\varepsilon}\right\|^2\Big)\\
\lesssim&\varepsilon\Big(\left\|{\lag v\rag}\nabla_v \nabla_x^{i}f^{\varepsilon}\right\|^2+\left\|\nabla_x^if^{\varepsilon}\right\|^2\Big).\nonumber
\end{align}
Then it suffices to control $\|{\lag v\rag}\nabla_v \nabla_x^if^{\varepsilon}\|^2$ with $i=0, 1, 2$.
To see this, similar to argument to derive \eqref{dxihfL}, one has
\begin{align*}
\left|\nabla_v\nabla_x^if^{\varepsilon}\right|\lesssim& \sum_{j=0}^i\exp\left(-\frac{\epsilon_1|v|^2}{8RTT_c}\right) \left[{\lag v\rag}^{2(i-j)+1}|h^{\varepsilon}|+{\lag v\rag}^{2(i-j)}|\nabla_v\nabla_x^jh^{\varepsilon}|\right]\\
\lesssim& \sum_{j=0}^iC_{\epsilon_1}\exp\left(-\frac{\epsilon_1|v|^2}{12RT^2_c}\right) \big[|h^{\varepsilon}|+|\nabla_v\nabla_x^jh^{\varepsilon}|\big].
\end{align*}

Next, similar to  \eqref{dxfhi0}, we use Lemma \ref{lower norm} to obtain
  \begin{align}\label{dxfhi0 VML}
  &\left\|{\lag v\rag}\nabla_v \nabla_x^if^{\varepsilon}\right\|^2\nonumber\\
  \lesssim& \int_{{\mathbb R}^3}\int_{{\lag v\rag}^4\leq\frac{1}{\varepsilon}}{\lag v\rag}^2\left|\nabla_v\nabla_x^if^{\varepsilon}\right|^2\, dv dx+\int_{{\mathbb R}^3}\int_{{\lag v\rag}^4\geq\frac{1}{\varepsilon}}{\lag v\rag}^2\left|\nabla_v\nabla_x^if^{\varepsilon}\right|^2\, dv dx\nonumber\\
    \lesssim&\,\left\|\nabla_x^if^{\varepsilon}\right\|^2+\frac{1}{\varepsilon}\left\|\nabla_x^i({\bf I}-{\bf P}_{\mathbf{M}})[f^{\varepsilon}]\right\|_D^2\\
    &+C_{\epsilon_1}\int_{{\mathbb R}^3}\int_{{\lag v\rag}^4\geq\frac{1}{\varepsilon}}\exp\left(-\frac{\epsilon_1|v|^2}{8RT^2_c}\right) \sum_{j=0}^i\big(|\nabla_v\nabla_x^jh^{\varepsilon}|^2+|\nabla_x^jh^{\varepsilon}|^2\big)\, dv dx\nonumber\\
    \lesssim&\,\|\nabla_x^if^{\varepsilon}\|^2+\frac{1}{\varepsilon}\left\|\nabla_x^i({\bf I}-{\bf P}_{\mathbf{M}})[f^{\varepsilon}]\right\|_D^2+C_{\epsilon_1}\exp\left(-\frac{\epsilon_1}{8RT^2_c\sqrt{\varepsilon}}\right)
    \Big(\|h^{\varepsilon}\|^2_{H^i}+\|h^{\varepsilon}\|^2_{H^i_D}\Big).\nonumber
\end{align}

Finally, combing \eqref{dxfhi0 VML} and \eqref{preEB}, \eqref{preEBR} gives the desired estimates. This ends the proof of Lemma \ref{dxvlm VML}.
\end{proof}

We are ready to deduce the following proposition.
\begin{proposition}\label{f-eng-vml}
Assume that $f^{\varepsilon}, E_R^{\varepsilon}, B_R^{\varepsilon}$ and $h^{\varepsilon}$ are smooth solutions of the Cauchy problem \eqref{VMLf-2}, \eqref{fM-2}, \eqref{VMLh-2} and \eqref{VML-id-pt} for the VML system \eqref{main1} and satisfy \eqref{aps-vml}, then it holds that
\begin{align}\label{f-eng-vml-sum}
&\frac{\mathrm{d}}{\mathrm{d} t}\sum\limits_{i=0}^2\vps^i\Big(\|\sqrt{4\pi RT}\na_x^if^{\varepsilon}\|^2+\|\na_x^iE_R^{\varepsilon}\|^2+\|\na_x^iB_R^{\varepsilon}\|^2\Big)\nonumber\\
    &+\delta\sum\limits_{i=0}^2\vps^{i-1}\|\na_x^i({\bf I}-{\bf P}_{\mathbf{M}})[f^{\varepsilon}]\|^2_D   \\
     \lesssim&\sum\limits_{i=0}^2\vps^{i}\Big[(1+t)^{-p_0}+\varepsilon^{\frac{2}{3}}\Big]
     \Big(\|\na_x^if^{\varepsilon}\|^2+\left\|[E_R^{\varepsilon},B_R^{\varepsilon}]\right\|_{H^i}^2
     +C_{\epsilon_0}\exp\left(-\frac{\epsilon_1}{8RT^2_c\sqrt{\varepsilon}}\right)(\|h^{\varepsilon}\|_{H^i}^2+\|h^{\varepsilon}\|_{H^i_D}^2)\Big)
     \nonumber\\
     &+\sum\limits_{i=1}^2\vps^{i-1}\Big[(1+t)^{-p_0}+\varepsilon^{\frac{2}{3}}\Big]
\|f^\vps\|_{H^{i-1}}
     +\sum\limits_{i=0}^2\left\{\varepsilon^{2k+1+i}(1+t)^{4k+2}+\varepsilon^{k+i}(1+t)^{2k}\|\na_x^if^{\varepsilon}\|\right\}.\nonumber
\end{align}
\end{proposition}
\begin{proof}The proof is divided into three steps.
\vskip 0.2cm

\noindent\underline{{\it Step 1. Basic energy estimate of the remainders.}} In this step, we derive the $L^2$ estimates on $[f^{\varepsilon}, E_R^{\varepsilon}, B_R^{\varepsilon}]$. Taking the $L^2$ inner product of \eqref{VMLf-2} with $4\pi RT f^{\varepsilon}$ and by applying \eqref{coL0}, one has
\begin{align}\label{L2f1 VML}
&\frac{1}{2}\frac{\mathrm{d}}{\mathrm{d} t}\Big(\|\sqrt{4\pi RT}f^{\varepsilon}\|^2+\|E_R^{\varepsilon}\|^2+\|B_R^{\varepsilon}\|^2\Big)
    +\frac{\delta}{\varepsilon}\|({\bf I}-{\bf P}_{\mathbf{M}})[f^{\varepsilon}]\|^2_D \\
    \leq&\;\frac{1}{2}\big|\big\langle \partial_tT f^{\varepsilon}, 4\pi R f^{\varepsilon}\big\rangle\big|+\Big|\Big\langle \big(E_R^{\varepsilon}+v \times B_R^{\varepsilon} \big) \cdot u\mathbf{M}^{\frac{1}{2}},4\pi  f^{\varepsilon}\Big\rangle\Big|\nonumber\\
    &+\Big|\Big\langle \Big(E+v \times B \Big)\cdot(v-u )f^{\varepsilon}, 2\pi  f^{\varepsilon}\Big\rangle\Big|\nonumber\\
    &+    \Big|\Big\langle f^{\varepsilon}\mathbf{M}^{-\frac{1}{2}}\big(\partial_t+v\cdot\nabla_x\big)\mathbf{M}^{\frac{1}{2}}, 4\pi RT f^{\varepsilon}\Big\rangle\Big|
    +\varepsilon^{k-1}\big|\big\langle\Gamma_{\mathbf{M}} ( f^{\varepsilon},
    f^{\varepsilon} ), 4\pi RT f^{\varepsilon}\big\rangle\big|\nonumber\\
    &+\sum_{i=1}^{2k-1}\varepsilon^{i-1}\big|\big\langle[\Gamma_{\mathbf{M}}(\mathbf{M}^{-\frac{1}{2}}F_i,f^{\varepsilon})+\Gamma_{\mathbf{M}}(
 f^{\varepsilon}, \mathbf{M}^{-\frac{1}{2}} F_i)\big], 4\pi RT f^{\varepsilon}\big\rangle\big|\nonumber\\
 &+\varepsilon^k \big|\big\langle \big(E_R^{\varepsilon}+v \times B_R^{\varepsilon}\big) \cdot(v-u)f^{\varepsilon},2\pi  f^{\varepsilon}\big\rangle\big|\nonumber\\
 &+\sum_{i=1}^{2k-1}\varepsilon^i\Big|\Big\langle \Big(E_R^{\varepsilon}+v \times B_R^{\varepsilon} \Big)\cdot\nabla_v F_i,4\pi RT f^{\varepsilon}\Big\rangle\Big|\nonumber\\
 &+\sum_{i=1}^{2k-1}\varepsilon^i\Big|\Big\langle \Big(E_i+v \times B_i \Big)\cdot(v-u)f^{\varepsilon},2\pi f^{\varepsilon}\Big\rangle\Big|
  +\big|\big\langle \CQ_0,4\pi RT f^{\varepsilon}\big\rangle\big|.\nonumber
\end{align}
Here we have used the following identity
\begin{align*}%\label{L2max}
\frac{1}{2}\frac{d}{dt}\left(\|E_R^{\varepsilon}(t)\|^2+\|B_R^{\varepsilon}(t)\|^2\right)=4\pi \Big\langle v\cdot E_R^{\varepsilon}\mathbf{M}^{\frac{1}{2}}, f^{\varepsilon}\Big\rangle
\end{align*}
from \eqref{fM-2}.

Now we turn to estimate the terms on the R.H.S. of \eqref{L2f1 VML} individually.
By \eqref{em-decay}, we see that the first and second terms on the R.H.S. of \eqref{L2f1 VML} can be dominated by
\begin{align*}
&C\|\partial_t T\|_{\infty}\|f^{\varepsilon}\|^2+C\|u\|_{\infty}\Big(\|f^{\varepsilon}\|^2+\|E_R^{\varepsilon}\|^2
+\|B_R^{\varepsilon}\|^2\Big)\\
\lesssim&\big(\|\nabla_x\rho\|_{\infty}+\|u\|_{W^{1,\infty}}+\|\nabla_xT\|_{\infty}\big)\Big(\|f^{\varepsilon}\|^2+\|E_R^{\varepsilon}\|^2
+\|B_R^{\varepsilon}\|^2\Big)\\
\lesssim& (1+t)^{-p_0}\epsilon_1 \Big(\|f^{\varepsilon}\|^2+\|E_R^{\varepsilon}\|^2
+\|B_R^{\varepsilon}\|^2\Big).
\end{align*}

In view of Lemma \ref{dxlm VML} and Lemma \ref{dxvlm VML}, the 3rd, 4th, 7th, 8th and 9th terms on the R.H.S. of \eqref{L2f1 VML} can be controlled by
\begin{align*}%\label{growL2}
 C\Big[\epsilon_1(1+t)^{-p_0}&+\varepsilon^{\frac{2}{3}}\Big]\Big(\|f^{\varepsilon}\|^2+\|E_R^{\varepsilon}\|^2
+\|B_R^{\varepsilon}\|^2\\
&+\frac{1}{\varepsilon}\|({\bf I}-{\bf P}_{\mathbf{M}})[f^{\varepsilon}]\|_D^2+C_{\epsilon_0}\exp\left(-\frac{\epsilon_0}{8RT^2_c\sqrt{\varepsilon}}\right)\|h^{\varepsilon}\|^2\Big).
\end{align*}

For the 5th term on the R.H.S. of \eqref{L2f1 VML}, one has by the {\it a priori} assumption \eqref{aps-vml} that
\begin{align}\label{fbd-vml}
\varepsilon^{k-1}\|f^{\varepsilon}\|_{H^2}\lesssim \varepsilon^{\frac{1}{2}}.
\end{align}
From which, we get from \eqref{GLM}, Sobolev's inequality and \eqref{fbd-vml} that
\begin{align*}
&\varepsilon^{k-1}\big|\big\langle\Gamma_{\mathbf{M}} ( f^{\varepsilon},
    f^{\varepsilon} ), 4\pi RT f^{\varepsilon}\big\rangle\big|\\
    =&\varepsilon^{k-1}\big|\big\langle\Gamma_{\mathbf{M}}(f^{\varepsilon},f^{\varepsilon}), 4\pi RT ({\bf I}-{\bf P}_{\mathbf{M}})[f^{\varepsilon}]\big\rangle\big|\\
    \lesssim&\varepsilon^{k-1}\int_{{\mathbb R}^3}|f^{\varepsilon}|_{L^2}|f^{\varepsilon}|_D|({\bf I}-{\bf P}_{\mathbf{M}})[f^{\varepsilon}]|_D\, \nonumber\\
    \lesssim&\varepsilon^{k-1}\|f^{\varepsilon}\|_{H^2}\Big(\|({\bf I}-{\bf P}_{\mathbf{M}})[f^{\varepsilon}]\|_D+\|{{\bf P}_{\mathbf{M}}}[f^{\varepsilon}]\|_D\Big)\|({\bf I}-{\bf P}_{\mathbf{M}})[f^{\varepsilon}]\|_D\nonumber\\
    \lesssim &\,\|({\bf I}-{\bf P}_{\mathbf{M}})[f^{\varepsilon}]\|_D^2+\varepsilon\|{{\bf P}_{\mathbf{M}}}[f^{\varepsilon}]\|^2\lesssim \|({\bf I}-{\bf P}_{\mathbf{M}})[f^{\varepsilon}]\|_D^2+\varepsilon\|f^{\varepsilon}\|^2 .\nonumber
\end{align*}

For the 6th term on the R.H.S. of \eqref{L2f1 VML}, from \eqref{GLM} and \eqref{em-Fn-es}, it follows that
\begin{align*}
&\sum_{i=1}^{2k-1}\varepsilon^{i-1}\big|\big\langle[\Gamma_{\mathbf{M}}(\mathbf{M}^{-\frac{1}{2}}F_i,f^{\varepsilon})+\Gamma_{\mathbf{M}}(
 f^{\varepsilon}, \mathbf{M}^{-\frac{1}{2}} F_i)\big], 4\pi RT f^{\varepsilon}\big\rangle\big|\\
 \lesssim&\sum_{i=1}^{2k-1}\varepsilon^{i-1}(1+t)^i\|f^{\varepsilon}\|_{H^2}\Big(\|({\bf I}-{\bf P}_{\mathbf{M}})[f^{\varepsilon}]\|_D+\|{{\bf P}_{\mathbf{M}}}[f^{\varepsilon}]\|_D\Big)\|({\bf I}-{\bf P}_{\mathbf{M}})[f^{\varepsilon}]\|_D\nonumber\\
    \lesssim&\, \frac{o(1)}{\varepsilon}\|({\bf I}-{\bf P}_{\mathbf{M}})[f^{\varepsilon}]\|_D^2+\varepsilon(1+t)^2\|{{\bf P}_{\mathbf{M}}}[f^{\varepsilon}]\|_D^2\\
    \lesssim& \frac{  o(1)}{\varepsilon}\|({\bf I}-{\bf P}_{\mathbf{M}})[f^{\varepsilon}]\|_D^2+\varepsilon^{\frac{1}{3}}\|f^{\varepsilon}\|^2\nonumber,
\end{align*}
where we also used the fact that $0\leq t\leq \varepsilon^{-1/3}$.

For the last term on the R.H.S. of \eqref{L2f1 VML},
using Cauchy's inequality, \eqref{em-Fn-es} and \eqref{em-EB-es}, we have
\begin{align*}
    \big|\big\langle \CQ_0, 4\pi RT f^{\varepsilon}\big\rangle\big|\lesssim& \frac{o(1)}{\varepsilon}\|({\bf I}-{\bf P}_{\mathbf{M}})[f^{\varepsilon}]\|_D^2+\varepsilon
    \sum_{\substack{i+j\geq 2k+1\\2\leq i,j\leq2k-1}}\varepsilon^{2(i+j-k-1)}\|F_i\|^2_{H^2}\|F_j\|_D^2\\
    &+\sum_{\substack{i+j\geq 2k\\1\leq i,j\leq2k-1}}\varepsilon^{i+j-k}\Big(\|E_i\|_{L^\infty}+\|B_i\|_{L^\infty}\Big)\big\|(1+v) \mathbf{M}^{-\frac{1}{2}}\nabla_vF_j\big\|\|f^{\varepsilon}\|\\
    \lesssim& \frac{o(1)}{\varepsilon}\|({\bf I}-{\bf P}_{\mathbf{M}})[f^{\varepsilon}]\|_D^2+\varepsilon^{2k+1}(1+t)^{4k+2}+\varepsilon^{k}(1+t)^{2k}\|f^{\varepsilon}\|.\nonumber
\end{align*}

Substituting the above estimates in \eqref{L2f1 VML} leads to
\begin{align}\label{L2f VML}
&\frac{\mathrm{d}}{\mathrm{d} t}\Big(\|\sqrt{4\pi RT}f^{\varepsilon}\|^2+\|E_R^{\varepsilon}\|^2+\|B_R^{\varepsilon}\|^2\Big)
    +\frac{\delta}{\varepsilon}\|({\bf I}-{\bf P}_{\mathbf{M}})[f^{\varepsilon}]\|^2_D\nonumber\\
     \lesssim&\,\Big[(1+t)^{-p_0}+\varepsilon^{\frac{1}{3}}\Big]\Big(\|f^{\varepsilon}\|^2+\|E_R^{\varepsilon}\|^2+\|B_R^{\varepsilon}\|^2
     +C_{\epsilon_0}\exp\left(-\frac{\epsilon_1}{8RT^2_c\sqrt{\varepsilon}}\right)\|h^{\varepsilon}\|^2\Big)\\
     &+\varepsilon^{2k+1}(1+t)^{4k+2}+\varepsilon^{k}(1+t)^{2k}\|f^{\varepsilon}\|.\nonumber
\end{align}
\vskip 0.2cm
\noindent\underline{{\it Step 2. Estimates on the first order derivative of the remainders.}}
In this step, we proceed to deduce the estimate of $\nabla_xf^{\varepsilon}, \nabla_x E^{\varepsilon}$ and $\nabla_x B_R^{\varepsilon}$. For this, applying $\partial_x^{\alpha} (1\leq |\alpha|\leq 2)$ to \eqref{VMLf-2} gives
\begin{align}\label{m0xVML}
\big(\partial_t&+v\cdot\nabla_x\big)\partial_x^{\alpha}f^{\varepsilon}
        +\partial_x^{\alpha}\Big[\frac{\big(E_R^{\varepsilon}+v \times B_R^{\varepsilon} \big) }{ RT}\cdot \big(v-u\big)\mathbf{M}^{\frac{1}{2}}\Big]\nonumber\\
&+\partial_x^{\alpha}\Big[\Big(E+v \times B \Big)\cdot\frac{v-u }{ 2RT}f^{\varepsilon}\Big]-\partial_x^{\alpha}\Big[\Big(E+v \times B \Big)\cdot\nabla_vf^{\varepsilon}\Big]+\frac{\partial_x^{\alpha}\mathcal{L}_{\mathbf{M}}[f^{\varepsilon}]}{\varepsilon}\nonumber\\
 =&-\partial_x^{\alpha}\Big[\mathbf{M}^{-\frac{1}{2}}f^{\varepsilon}\Big[\partial_t+\hat{p}\cdot\nabla_x-\Big(E+v \times B \Big)\cdot\nabla_v\Big]\mathbf{M}^{\frac{1}{2}}\Big]+\varepsilon^{k-1}\partial_x^{\alpha}\Gamma_{\mathbf{M}}(f^{\varepsilon},f^{\varepsilon})\nonumber\\
 &+\sum_{i=1}^{2k-1}\varepsilon^{i-1}[\partial_x^{\alpha}\Gamma_{\mathbf{M}}(\mathbf{M}^{-\frac{1}{2}}F_i, f^{\varepsilon})+\partial_x^{\alpha}\Gamma_{\mathbf{M}}(f^{\varepsilon}, \mathbf{M}^{-\frac{1}{2}} F_i)\big]+\varepsilon^k \partial_x^{\alpha}\Big[\Big(E_R^{\varepsilon}+v \times B_R^{\varepsilon}\Big)\cdot\nabla_vf^{\varepsilon}\Big]\nonumber\\
 &-\varepsilon^k \partial_x^{\alpha}\Big[\Big(E_R^{\varepsilon}+v \times B_R^{\varepsilon}\Big) \cdot\frac{v-u }{ 2RT}f^{\varepsilon}\Big]\\
 &+\sum_{i=1}^{2k-1}\varepsilon^i\partial_x^{\alpha}\Big[\Big(E_i+v \times B_i \Big)\cdot\nabla_vf^{\varepsilon}+\Big(E_R^{\varepsilon}+v \times B_R^{\varepsilon} \Big)\cdot\nabla_v F_i\Big]\nonumber\\
 &-\sum_{i=1}^{2k-1}\varepsilon^i\partial_x^{\alpha}\Big[\Big(E_i+v \times B_i \Big)\cdot\frac{\big(v-u\big)}{ 2RT}f^{\varepsilon}\Big]+\varepsilon^{k}\partial_x^{\alpha}\CQ_0.\nonumber
\end{align}

Furthermore, taking the inner product of $4\pi RT\partial^\alpha f^{\varepsilon}$ and \eqref{m0xVML} with $|\al|=1$, one gets
\begin{align}\label{H1f1 VML}
&\frac{1}{2}\frac{\mathrm{d}}{\mathrm{d} t}\Big(\|\sqrt{4\pi RT}\nabla_xf^{\varepsilon}\|^2+\|[\nabla_xE_R^{\varepsilon},\nabla_xB_R^{\varepsilon}]\|^2\Big)
    +\frac{1}{\varepsilon}\big\langle \nabla_x\mathcal{L}_{\mathbf{M}}[f^{\varepsilon}],4\pi RT\nabla_xf^{\varepsilon}\big\rangle \\
    \lesssim&\;\big|\big\langle \partial_tT \nabla_xf^{\varepsilon},   \nabla_xf^{\varepsilon}\big\rangle\big|+\sum_{|\alpha+\alpha'|=1}\Big|\Big\langle \big(E_R^{\varepsilon}+v \times \partial^\alpha B_R^{\varepsilon} \big) \cdot \partial^{\alpha'}\Big[\frac{v-u}{ T}\mathbf{M}^{\frac{1}{2}}\Big],T \nabla_xf^{\varepsilon}\Big\rangle\Big|\nonumber\\
    &+\sum\limits_{|\alpha|=1}\big|\big\langle v\partial^\alpha \mathbf{M}^{-\frac{1}{2}}f^{\varepsilon},   \partial^\alpha E_R^{\varepsilon}\big\rangle\big|+\Big|\Big\langle \nabla_x\left[\big(E+v \times B \big) \cdot \nabla_v f^{\varepsilon}\right], T \nabla_xf^{\varepsilon}\Big\rangle\Big|\nonumber\\
    &+\Big|\Big\langle \nabla_x\Big[\Big(E+v \times B \Big)\cdot\frac{v-u }{ 2T}f^{\varepsilon}\Big],T \nabla_xf^{\varepsilon}\Big\rangle\Big|\nonumber\\
    &+    \Big|\Big\langle \nabla_x\Big[f^{\varepsilon}\mathbf{M}^{-\frac{1}{2}}\big(\partial_t+v\cdot\nabla_x\big)\mathbf{M}^{\frac{1}{2}}\Big], T \nabla_xf^{\varepsilon}\Big\rangle\Big|
    +\varepsilon^{k-1}\big|\big\langle\nabla_x\Gamma_{\mathbf{M}} ( f^{\varepsilon},
    f^{\varepsilon} ), T \nabla_xf^{\varepsilon}\big\rangle\big|\nonumber\\
    &+\sum_{i=1}^{2k-1}\varepsilon^{i-1}\big|\big\langle[\nabla_x\Gamma_{\mathbf{M}}(\mathbf{M}^{-\frac{1}{2}}F_i, f^{\varepsilon})+\nabla_x\Gamma_{\mathbf{M}}(
 f^{\varepsilon}, \mathbf{M}^{-\frac{1}{2}} F_i)\big], T \nabla_xf^{\varepsilon}\big\rangle\big|\nonumber\\
 &+\varepsilon^k \big|\big\langle \nabla_x\Big[\big(E_R^{\varepsilon}+v \times B_R^{\varepsilon}\big) \cdot\nabla_vf^{\varepsilon}\Big],T  \nabla_xf^{\varepsilon}\big\rangle\big|\nonumber\\
 &+\varepsilon^k \big|\big\langle \nabla_x\Big[\big(E_R^{\varepsilon}+v \times B_R^{\varepsilon}\big) \cdot\frac{(v-u)}{ T}f^{\varepsilon}\Big],T \nabla_xf^{\varepsilon}\big\rangle\big|\nonumber\\
 &+\sum_{i=1}^{2k-1}\varepsilon^i\Big|\Big\langle \nabla_x\Big[\big(E_i+v \times B_i\big) \cdot\nabla_vf^{\varepsilon}+\Big(E_R^{\varepsilon}+v \times B_R^{\varepsilon} \Big)\cdot\nabla_v F_i\Big],T \nabla_xf^{\varepsilon}\Big\rangle\Big|\nonumber\\
 &+\sum_{i=1}^{2k-1}\varepsilon^i\Big|\Big\langle \nabla_x\Big[\Big(E_i+v \times B_i \Big)\cdot\frac{(v-u)}{ T}f^{\varepsilon}\Big], T\nabla_xf^{\varepsilon}\Big\rangle\Big|
  +\big|\big\langle \nabla_x\CQ_0,T \nabla_xf^{\varepsilon}\big\rangle\big|.\nonumber
\end{align}

For the second term on the L.H.S. of \eqref{H1f1 VML}, applying \eqref{coLh} and \eqref{em-decay}, we have
\begin{align*}
\frac{1}{\varepsilon}\big\langle \nabla_x\mathcal{L}_{\mathbf{M}}[f^{\varepsilon}],4\pi RT \nabla_xf^{\varepsilon}\big\rangle
\geq&\frac{3\delta}{4\varepsilon}\|\nabla_x({\bf I}-{\bf P}_{\mathbf{M}})[f^{\varepsilon}]\|_D^2-\frac{C\epsilon_1}{\varepsilon^2}\|({\bf I}-{\bf P}_{\mathbf{M}})[f^{\varepsilon}]\|_D^2\\
&-\frac{C\epsilon_1}{\varepsilon}(1+t)^{-p_0}\big(\varepsilon\|f^{\varepsilon}\|^2_{H^1}+\|f^{\varepsilon}\|^2\big).
\end{align*}

We now turn to estimate the R.H.S. of \eqref{H1f1 VML} separately.
Using \eqref{em-decay}, one sees that
the first three terms on the R.H.S. of \eqref{H1f1 VML} can be bounded by
\begin{align*}
 C\Big(\|\nabla_x\rho\|_{\infty}&+\|\nabla_xu\|_{\infty}+\|\nabla_xT\|_{\infty}\Big)\Big(\|f^{\varepsilon}\|^2_{H^1}+\|E_R^{\varepsilon}\|^2_{H^1}
+\|B_R^{\varepsilon}\|^2_{H^1}\Big)\\
\lesssim &\epsilon_1(1+t)^{-p_0}\Big(\|f^{\varepsilon}\|^2_{H^1}+\|E_R^{\varepsilon}\|^2_{H^1}
+\|B_R^{\varepsilon}\|^2_{H^1}\Big).
\end{align*}
It follows from Lemma \ref{dxvlm VML} that the 4th, 9th and 11th terms on the R.H.S. of \eqref{H1f1 VML} can be dominated by
\begin{align*}
C\Big[\epsilon_1(1+t)^{-p_0}&+\varepsilon^{\frac{2}{3}}\Big]\Big[\|f^{\varepsilon}\|^2_{H^1}+\|E_R^{\varepsilon}\|^2_{H^1} +\|B_R^{\varepsilon}\|^2_{H^1}+\frac{1}{\varepsilon}\|({\bf I}-{\bf P}_{\mathbf{M}})[f^{\varepsilon}]\|_{H^1_D}^2\nonumber\\
&+C_{\epsilon_1}\exp\left(-\frac{\epsilon_1}{8RT^2_c\sqrt{\varepsilon}}\right)
    \Big(\|h^{\varepsilon}\|^2_{H^{1}}+\|h^{\varepsilon}\|^2_{H^{1}_D}\Big)\Big].
\end{align*}
By Lemma \ref{dxlm VML}, we see that the 5th, 6th, 10th and 12th  terms on the R.H.S. of \eqref{H1f1 VML} can be controlled by
\begin{align*}%\label{growL2}
C\Big[(1+t)^{-p_0}+\varepsilon^{\frac{2}{3}}\Big]\Big(\|f^{\varepsilon}\|^2_{H^1}+\sum_{j=0}^1\frac{1}{\varepsilon^{(2-j)}} \|\nabla_x^j({\bf I}-{\bf P}_{\mathbf{M}})[f^{\varepsilon}]\|_D^2+\exp\left(-\frac{\epsilon_1}{8RT^2_c\sqrt{\varepsilon}}\right)\|h^{\varepsilon}\|^2_{H^1}\Big).
\end{align*}
For the 7th term on the R.H.S. of \eqref{H1f1 VML},
using \eqref{GLM}, \eqref{fh2b} and Sobolev's inequalities, we have
\begin{align*}
    &\varepsilon^{k-1}\big|\big\langle \nabla_x\Gamma_{\mathbf{M}} ( f^{\varepsilon},
    f^{\varepsilon} ), T \nabla_xf^{\varepsilon}\big\rangle\big|\\
        \lesssim&\varepsilon^{k-1}\int_{{\mathbb R}^3}\Big(|\nabla_xf^{\varepsilon}|_{D}|f^{\varepsilon}|_{L^2}+|\nabla_xf^{\varepsilon}|_{L^2}|f^{\varepsilon}|_{D}\Big) \Big(|\nabla_x({\bf I}-{\bf P}_{\mathbf{M}})[f^{\varepsilon}]|_D+|\Lbrack\nabla_x,{{\bf P}_{\mathbf{M}}}\Rbrack[f^{\varepsilon}]|_D\Big)\\
    &+\epsilon_1 \varepsilon^{k-1}(1+t)^{-p_0}\|f^{\varepsilon}\|_{H^2}\Big(\|({\bf I}-{\bf P}_{\mathbf{M}})[f^{\varepsilon}]\|_D+\|{{\bf P}_{\mathbf{M}}}[f^{\varepsilon}]\|_D\Big)\\
    &\times\Big(\|\nabla_x({\bf I}-{\bf P}_{\mathbf{M}})[f^{\varepsilon}]\|_D+\|\nabla_x{{\bf P}_{\mathbf{M}}}[f^{\varepsilon}]\|_D\Big)\\
    \lesssim&\varepsilon^{k-1}\|f^{\varepsilon}\|_{H^2}\Big(\|\nabla_x({\bf I}-{\bf P}_{\mathbf{M}})[f^{\varepsilon}]\|_D+\|\nabla_x{{\bf P}_{\mathbf{M}}}[f^{\varepsilon}]\|_D\Big)\Big(\|\nabla_x({\bf I}-{\bf P}_{\mathbf{M}})[f^{\varepsilon}]\|_D+\epsilon_1\|f^{\varepsilon}\|\Big)\nonumber\\
    &+\epsilon_1 \varepsilon^{\frac{1}{2}}(1+t)^{-p_0}\Big(\|({\bf I}-{\bf P}_{\mathbf{M}})[f^{\varepsilon}]\|_D+\|f^{\varepsilon}\|\Big)\Big(\|\nabla_x({\bf I}-{\bf P}_{\mathbf{M}})[f^{\varepsilon}]\|_D+\|f^{\varepsilon}\|_{H^1}\Big)\\
    \lesssim &\,\|\nabla_x({\bf I}-{\bf P}_{\mathbf{M}})[f^{\varepsilon}]\|_D^2+\|({\bf I}-{\bf P}_{\mathbf{M}})[f^{\varepsilon}]\|_D^2+\epsilon_1(1+t)^{-p_0}\Big(\|f^{\varepsilon}\|^2+ \varepsilon\|f^{\varepsilon}\|_{H^1}^2\Big).\nonumber
\end{align*}
For the 8th term on the R.H.S. of \eqref{H1f1 VML}, note that $0\leq t\leq\varepsilon^{-1/3}$,
by virtue of \eqref{em-Fn-es}, \eqref{GLM} and Cauchy's inequalities, we get
\begin{align*}
&\sum_{i=1}^{2k-1}\varepsilon^{i-1}\big|\big\langle[\nabla_x\Gamma_{\mathbf{M}}(\mathbf{M}^{-\frac{1}{2}}F_i, f^{\varepsilon})+\nabla_x\Gamma_{\mathbf{M}}(
 f^{\varepsilon}, \mathbf{M}^{-\frac{1}{2}} F_i)\big], T \nabla_xf^{\varepsilon}\big\rangle\big|\\
  \lesssim&(1+t)\big(\|{f^{\varepsilon}}\|_{H^1}+\|{f^{\varepsilon}}\|_{H^1_D}\big)\Big(\|\nabla_x({\bf I}-{\bf P}_{\mathbf{M}})[f^{\varepsilon}]\|_D+\|\Lbrack\nabla_x,{{\bf P}_{\mathbf{M}}}\Rbrack[f^{\varepsilon}]\|_D\Big)\\
 &+\epsilon_1(1+t)^{1-p_0}\big(\|{f^{\varepsilon}}\|+\|{f^{\varepsilon}}\|_D\big)\Big(\|\nabla_x({\bf I}-{\bf P}_{\mathbf{M}})[f^{\varepsilon}]\|_D+\|f^{\varepsilon}\|_{H^1}\Big)\\
    \lesssim&\, \frac{o(1)}{\varepsilon}\|\nabla_x({\bf I}-{\bf P}_{\mathbf{M}})[f^{\varepsilon}]\|_D^2+\frac{\epsilon_1}{\varepsilon}\|({\bf I}-{\bf P}_{\mathbf{M}})[f^{\varepsilon}]\|_D^2 +\frac{\epsilon_1}{\varepsilon}(1+t)^{-p_0}\|f^{\varepsilon}\|^2+\varepsilon^{\frac{2}{3}}\|f^{\varepsilon}\|^2_{H^1}\nonumber.
\end{align*}
For the last term on the R.H.S. of \eqref{H1f1 VML},
\eqref{em-Fn-es} gives
\begin{align*}
    \big|\big\langle \nabla_x\CQ_0, T \nabla_xf^{\varepsilon}\big\rangle\big|
    \lesssim& \frac{o(1)}{\varepsilon}\|\nabla_x({\bf I}-{\bf P}_{\mathbf{M}})[f^{\varepsilon}]\|_D^2+\frac{\epsilon_1}{\varepsilon}\|({\bf I}-{\bf P}_{\mathbf{M}})[f^{\varepsilon}]\|_D^2\\
    &+\frac{\epsilon_1}{\varepsilon}(1+t)^{-p_0}\|f^{\varepsilon}\|^2 +\varepsilon^{2k+1}(1+t)^{4k+2}+\varepsilon^{k}(1+t)^{2k}\|\nabla_xf^{\varepsilon}\|.\nonumber
\end{align*}
Finally, plugging the above estimates into \eqref{H1f1 VML} and multiplying the resulting inequality by $\varepsilon$, we arrive at
\begin{align}\label{H1f VML}
\frac{\mathrm{d}}{\mathrm{d} t}\Big[\varepsilon\Big(&\|\sqrt{4\pi RT}\nabla_xf^{\varepsilon}\|^2+\|\nabla_xE_R^{\varepsilon}\|^2+\|\nabla_xB_R^{\varepsilon}\|^2\Big)\Big]
    +\delta\|\nabla_x({\bf I}-{\bf P}_{\mathbf{M}})[f^{\varepsilon}]\|^2_D   \\
     \lesssim&\, \varepsilon\Big((1+t)^{-p_0}+\varepsilon^{\frac{2}{3}}\Big)\Big[\|\nabla_xf^{\varepsilon}\|^2+\frac{1}{\varepsilon}\|f^{\varepsilon}\|^2
     +\|E_R^{\varepsilon}\|^2_{H^1}+\|B_R^{\varepsilon}\|^2_{H^1}
     \nonumber\\
     &+ C_{\epsilon_1}\exp\left(-\frac{\epsilon_1}{8RT^2_c\sqrt{\varepsilon}}\right)
    \Big(\|h^{\varepsilon}\|^2_{H^1}+\|h^{\varepsilon}\|^2_{H^1_D}\Big)\Big]+\frac{\epsilon_1}{\varepsilon}\|({\bf I}-{\bf P}_{\mathbf{M}})[f^{\varepsilon}]\|^2_D\nonumber\\
    &+\varepsilon^{2k+2}(1+t)^{4k+2}
     +\varepsilon^{k+1}(1+t)^{2k}\|\nabla_xf^{\varepsilon}\|.\nonumber
\end{align}
\vskip 0.2cm
\noindent\underline{{\it Step 3. Estimates on the second order derivative of the remainders.}}
%%%%%%%%%%%%%%%%%%%%%%%%%%%%%%%%%%%%%%%%%%%%%%%%%%%%%%%%%%%%%%%%%%%%%%%%%%%%%%%%%%
In this step, we proceed to derive the estimate of $\|[\nabla^2_x f^{\varepsilon},\nabla_x^2 E_R^{\varepsilon}, \nabla_x^2 B_R^{\varepsilon}]\|$. For results in this direction, we have
\begin{align}\label{H2f VML}
\frac{\mathrm{d}}{\mathrm{d} t}\Big[\varepsilon^2&\Big(\|\sqrt{4\pi RT}\nabla_x^2f^{\varepsilon}\|^2+\|\nabla_x^2E_R^{\varepsilon}\|^2+\|\nabla_x^2B_R^{\varepsilon}\|^2\Big)\Big]
    +\delta\varepsilon\|\nabla_x^2({\bf I}-{\bf P}_{\mathbf{M}})[f^{\varepsilon}]\|^2_D   \\
     \lesssim&\,\varepsilon^2\Big((1+t)^{-p_0}+\varepsilon^{\frac{2}{3}}\Big)\Big[\|\nabla_x^2f^{\varepsilon}\|^2 +\frac{1}{\varepsilon}\|f^{\varepsilon}\|^2_{H^1}
     +\|E_R^{\varepsilon}\|^2_{H^2}+\|B_R^{\varepsilon}\|^2_{H^2}\Big)
     \nonumber\\
     &+ C_{\epsilon_1}\exp\left(-\frac{\epsilon_1}{8RT^2_c}\varepsilon^{-\frac{1}{6}}\right)
    \Big(\|h^{\varepsilon}\|^2_{H^2}+\|h^{\varepsilon}\|^2_{H^2_D}\Big)\Big]+\epsilon_1\|\nabla_x({\bf I}-{\bf P}_{\mathbf{M}})[f^{\varepsilon}]\|^2_{D}\nonumber\\
    &+\epsilon_1\|({\bf I}-{\bf P}_{\mathbf{M}})[f^{\varepsilon}]\|^2_D+\varepsilon^{2k+3}(1+t)^{4k+2}
     +\varepsilon^{k+2}(1+t)^{2k}\|\nabla_x^2f^{\varepsilon}\|.\nonumber
\end{align}

To prove \eqref{H2f VML}, we first take the inner product of $4\pi RT\partial_x^\alpha f^{\varepsilon}$ and \eqref{m0xVML} with $|\alpha|=2$ to obtain
\begin{align}\label{H2f1 VML}
&\frac{1}{2}\frac{\mathrm{d}}{\mathrm{d} t}\Big(\|\sqrt{4\pi RT}\nabla_x^2f^{\varepsilon}\|^2+\|\nabla_x^2E_R^{\varepsilon}\|^2+\|\nabla_x^2B_R^{\varepsilon}\|^2\Big)
    +\frac{1}{\varepsilon}\big\langle \nabla_x^2\mathcal{L}_{\mathbf{M}}[f^{\varepsilon}],4\pi RT\nabla_x^2f^{\varepsilon}\big\rangle \\
    \lesssim&\;\big|\big\langle \partial_tT \nabla_x^2f^{\varepsilon},  \nabla_x^2f^{\varepsilon}\big\rangle\big|+\sum_{|\alpha+\alpha'|=2}\Big|\Big\langle \partial^\alpha\big(E_R^{\varepsilon}+v \times B_R^{\varepsilon} \big) \cdot \partial^{\alpha'}\Big[\frac{v-u}{ T}\mathbf{M}^{\frac{1}{2}}\Big], T \nabla_x^2f^{\varepsilon}\Big\rangle\Big|\nonumber\\
    &+\sum_{|\alpha+\alpha'|=2}\big|\big\langle v\partial^\alpha\big(\mathbf{M}^{-\frac{1}{2}}\big)\partial^{\alpha'}f^{\varepsilon},   \nabla_x^2E_R^{\varepsilon}\big\rangle\big|+\Big|\Big\langle \nabla_x^2\Big[\big(E+v \times B \big) \cdot \nabla_v f^{\varepsilon}\Big], T \nabla_x^2f^{\varepsilon}\Big\rangle\Big|\nonumber\\
    &+\Big|\Big\langle \nabla_x^2\Big[\Big(E+v \times B \Big)\cdot\frac{v-u }{ 2T}f^{\varepsilon}\Big],T \nabla_x^2f^{\varepsilon}\Big\rangle\Big|\nonumber\\
    &+    \Big|\Big\langle \nabla_x^2\Big[f^{\varepsilon}\mathbf{M}^{-\frac{1}{2}}\big(\partial_t+v\cdot\nabla_x\big)\mathbf{M}^{\frac{1}{2}}\Big], T \nabla_x^2f^{\varepsilon}\Big\rangle\Big|
    +\varepsilon^{k-1}\big|\big\langle\nabla_x^2\Gamma_{\mathbf{M}} ( f^{\varepsilon},
    f^{\varepsilon} ), T \nabla_x^2f^{\varepsilon}\big\rangle\big|\nonumber\\
    &+\sum_{i=1}^{2k-1}\varepsilon^{i-1}\big|\big\langle[\nabla_x^2\Gamma_{\mathbf{M}}(\mathbf{M}^{-\frac{1}{2}}F_i, f^{\varepsilon})+\nabla_x^2\Gamma_{\mathbf{M}}(
 f^{\varepsilon}, \mathbf{M}^{-\frac{1}{2}} F_i)\big], T \nabla_x^2f^{\varepsilon}\big\rangle\big|\nonumber\\
 &+\varepsilon^k \big|\big\langle \nabla_x^2\Big[\big(E_R^{\varepsilon}+v \times B_R^{\varepsilon}\big) \cdot\nabla_vf^{\varepsilon}\Big],T  \nabla_x^2f^{\varepsilon}\big\rangle\big|\nonumber\\
 &+\varepsilon^k \big|\big\langle \nabla_x^2\Big[\big(E_R^{\varepsilon}+v \times B_R^{\varepsilon}\big) \cdot\frac{(v-u)}{ T}f^{\varepsilon}\Big],T \nabla_x^2f^{\varepsilon}\big\rangle\big|\nonumber\\
 &+\sum_{i=1}^{2k-1}\varepsilon^i\Big|\Big\langle \nabla_x^2\Big[\big(E_i+v \times B_i\big) \cdot\nabla_vf^{\varepsilon}+\Big(E_R^{\varepsilon}+v \times B_R^{\varepsilon} \Big)\cdot\nabla_v F_i\Big],T \nabla_x^2f^{\varepsilon}\Big\rangle\Big|\nonumber\\
 &+\sum_{i=1}^{2k-1}\varepsilon^i\Big|\Big\langle \nabla_x^2\Big[\Big(E_i+v \times B_i \Big)\cdot\frac{(v-u)}{ T}f^{\varepsilon}\Big], T\nabla_x^2f^{\varepsilon}\Big\rangle\Big|
  +\big|\big\langle \nabla_x^2\CQ_0,T \nabla_x^2f^{\varepsilon}\big\rangle\big|.\nonumber
\end{align}

Using \eqref{em-decay} and \eqref{coLh} again, we see that the 2nd term on the L.H.S. of \eqref{H2f1 VML} can be bounded by
\begin{align*}
\frac{1}{\varepsilon}\big\langle \nabla_x^2\mathcal{L}_{\mathbf{M}}[f^{\varepsilon}],4\pi RT\nabla_x^2f^{\varepsilon}\big\rangle
\geq&\frac{3\delta}{4\varepsilon}\|\nabla_x^2({\bf I}-{\bf P}_{\mathbf{M}})[f^{\varepsilon}]\|_D^2-\frac{C\epsilon_1}{\varepsilon^2}\|({\bf I}-{\bf P}_{\mathbf{M}})[f^{\varepsilon}]\|_{H^1_D}^2\\
&-\frac{C\epsilon_1}{\varepsilon}(1+t)^{-p_0}\Big(\varepsilon\|f^{\varepsilon}\|^2_{H^2}+\|f^{\varepsilon}\|^2_{H^1}\Big).
\end{align*}
It follows from \eqref{em-decay} that the first three terms on the R.H.S. of \eqref{H2f1 VML} can be bounded by
\begin{align*}
 C\Big(\|\nabla_x\rho\|_{W^{1,\infty}}&+\|\nabla_xu\|_{W^{1,\infty}}+\|\nabla_xT\|_{W^{1,\infty}}\Big)\Big(\|f^{\varepsilon}\|^2_{H^2}
 +\|E_R^{\varepsilon}\|^2_{H^2}
+\|B_R^{\varepsilon}\|^2_{H^2}\Big)\\
\lesssim &\epsilon_1(1+t)^{-p_0}\Big(\|f^{\varepsilon}\|^2_{H^2}+\|E_R^{\varepsilon}\|^2_{H^2}
+\|B_R^{\varepsilon}\|^2_{H^2}\Big).
\end{align*}
In view of Lemma \ref{dxvlm VML}, one sees that the 4th, 9th and 11th terms on the R.H.S. of \eqref{H2f1 VML} are no more than
\begin{align*}
C\Big[\epsilon_1(1+t)^{-p_0}&+\varepsilon^{\frac{2}{3}}\Big]\Big[\|f^{\varepsilon}\|^2_{H^2}+\|E_R^{\varepsilon}\|^2_{H^2} +\|B_R^{\varepsilon}\|^2_{H^2}+\frac{1}{\varepsilon}\|({\bf I}-{\bf P}_{\mathbf{M}})[f^{\varepsilon}]\|_{H^2_D}^2\nonumber\\
&+C_{\epsilon_1}\exp\left(-\frac{\epsilon_1}{8RT^2_c\sqrt{\varepsilon}}\right)
    \Big(\|h^{\varepsilon}\|^2_{H^{2}}+\|h^{\varepsilon}\|^2_{H^{2}_D}\Big)\Big].
\end{align*}
By Lemma \ref{dxlm VML},
we see that the 5th, 6th, 10th and 12th term on the R.H.S. of \eqref{H2f1 VML} can be controlled by
\begin{align*}%\label{growL2}
  C\Big[\epsilon_1(1+t)^{-p_0}&+\varepsilon^{\frac{2}{3}}\Big]\Big(\|f^{\varepsilon}\|^2_{H^1} +\sum_{j=0}^2\frac{1}{\varepsilon^{(3-j)}}\|\nabla_x^j({\bf I}-{\bf P}_{\mathbf{M}})[f^{\varepsilon}]\|_D^2+\exp\left(-\frac{\epsilon_1}{8RT^2_c\sqrt{\varepsilon}}\right)
  \|h^{\varepsilon}\|^2_{H^2}\Big).
\end{align*}
For the 7th term on the R.H.S. of \eqref{H2f1 VML}, applying \eqref{GL}, \eqref{fbd-vml} and Sobolev's inequalities, we get that
\begin{align*}
&\varepsilon^{k-1}\big|\big\langle
\nabla_x^2\Gamma_{\mathbf{M}} ( f^{\varepsilon},
    f^{\varepsilon} ), T\nabla_x^2f^{\varepsilon}\big\rangle\big|\\
        \lesssim&\varepsilon^{k-1}\int_{{\mathbb R}^3}\Big(|\nabla_x^2f^{\varepsilon}|_{D}|f^{\varepsilon}|_{L^2}+|\nabla_x^2f^{\varepsilon}|_{L^2}|f^{\varepsilon}|_{D}
    +|\nabla_xf^{\varepsilon}|_{L^2}|\nabla_xf^{\varepsilon}|_{D}\Big)\\
    &\times\Big(\|\nabla_x^2({\bf I}-{\bf P}_{\mathbf{M}})[f^{\varepsilon}]\|_D+\|\Lbrack\nabla_x^2,{{\bf P}_{\mathbf{M}}}\Rbrack[f^{\varepsilon}]\|_D\Big)
    +\epsilon_1 \varepsilon^{k-1}(1+t)^{-p_0}\|f^{\varepsilon}\|_{H^2}\\
    &\times\Big(\|({\bf I}-{\bf P}_{\mathbf{M}})[f^{\varepsilon}]\|_{H^1_D}+\|{{\bf P}_{\mathbf{M}}}[f^{\varepsilon}]\|_{H^1_D}\Big)\Big(\|\nabla_x^2({\bf I}-{\bf P}_{\mathbf{M}})[f^{\varepsilon}]\|_D+\|\nabla_x^2{{\bf P}_{\mathbf{M}}}[f^{\varepsilon}]\|_D\Big)\\
    \lesssim&\varepsilon^{k-1}\|f^{\varepsilon}\|_{H^2}\Big(\|\nabla^2_x({\bf I}-{\bf P}_{\mathbf{M}})[f^{\varepsilon}]\|_D+\|\nabla^2_x{{\bf P}_{\mathbf{M}}}[f^{\varepsilon}]\|_D\Big)\Big(\|\nabla_x^2({\bf I}-{\bf P}_{\mathbf{M}})[f^{\varepsilon}]\|_D+\epsilon_1\|f^{\varepsilon}\|_{H^1}\Big)\nonumber\\
    &+\epsilon_1 \varepsilon^{\frac{1}{2}}(1+t)^{-p_0}\Big(\|({\bf I}-{\bf P}_{\mathbf{M}})[f^{\varepsilon}]\|_{H^1_D}+\|f^{\varepsilon}\|_{H^1}\Big)\Big(\|\nabla_x^2({\bf I}-{\bf P}_{\mathbf{M}})[f^{\varepsilon}]\|_D+\|f^{\varepsilon}\|_{H^2}\Big)\\
    \lesssim &\,\|({\bf I}-{\bf P}_{\mathbf{M}})[f^{\varepsilon}]\|_{H^2_D}^2+\epsilon_1(1+t)^{-p_0}\Big(\|f^{\varepsilon}\|^2_{H^1}+ \varepsilon\|f^{\varepsilon}\|_{H^2}^2\Big).\nonumber
\end{align*}
For the 8th term on the R.H.S. of \eqref{H2f1 VML},
\eqref{GLM} and \eqref{em-Fn-es} lead us to
\begin{align*}
&\sum_{i=1}^{2k-1}\varepsilon^{i-1}\big|\big\langle[\nabla_x^2\Gamma_{\mathbf{M}}(\mathbf{M}^{-\frac{1}{2}}F_i,f^{\varepsilon}) +\nabla_x^2\Gamma_{\mathbf{M}}(
 f^{\varepsilon}, \mathbf{M}^{-\frac{1}{2}} F_i)\big], T\nabla_x^2f^{\varepsilon}\big\rangle\big|\\
  \lesssim&(1+t)\big(\|{f^{\varepsilon}}\|_{H^2}+\|{f^{\varepsilon}}\|_{H^2_D}\big)\Big(\|\nabla_x^2({\bf I}-{\bf P}_{\mathbf{M}})[f^{\varepsilon}]\|_D+\|\Lbrack\nabla_x^2,{{\bf P}_{\mathbf{M}}}\Rbrack[f^{\varepsilon}]\|_D\Big)\\
 &+\epsilon_1(1+t)^{1-p_0}\big(\|{f^{\varepsilon}}\|_{H^1}+\|{f^{\varepsilon}}\|_{H^1_D}\big)\Big(\|\nabla_x^2({\bf I}-{\bf P}_{\mathbf{M}})[f^{\varepsilon}]\|_D+\|f^{\varepsilon}\|_{H^2}\Big)\\
    \lesssim&\, \frac{o(1)}{\varepsilon}\|\nabla_x^2({\bf I}-{\bf P}_{\mathbf{M}})[f^{\varepsilon}]\|_D^2+\frac{\epsilon_1}{\varepsilon}\|({\bf I}-{\bf P}_{\mathbf{M}})[f^{\varepsilon}]\|_{H^1_D}^2+\frac{\epsilon_1}{\varepsilon}(1+t)^{-p_0}\|f^{\varepsilon}\|_{H^1}^2+\varepsilon^{\frac{2}{3}}\|f^{\varepsilon}\|^2_{H^2}\nonumber.
\end{align*}
For the last term on the R.H.S. of \eqref{H2f1 VML}, one gets from \eqref{em-Fn-es} that
\begin{align*}
    \big|\big\langle \nabla_x^2\CQ_0, T \nabla_x^2f^{\varepsilon}\big\rangle\big|
    \lesssim& \frac{o(1)}{\varepsilon}\|\nabla_x^2({\bf I}-{\bf P}_{\mathbf{M}})[f^{\varepsilon}]\|_D^2+\frac{\epsilon_1}{\varepsilon}\|\nabla_x({\bf I}-{\bf P}_{\mathbf{M}})[f^{\varepsilon}]\|_D^2\\
    &+\frac{\epsilon_1}{\varepsilon}(1+t)^{-p_0}\|f^{\varepsilon}\|^2_{H^1} +\varepsilon^{2k+1}(1+t)^{4k+2}+\varepsilon^{k}(1+t)^{2k}\|\nabla_x^2f^{\varepsilon}\|.\nonumber
\end{align*}
Then Substituting the above estimates in \eqref{H2f1 VML} and multiplying the resulting inequality by $\varepsilon^2$ give \eqref{H2f VML}.

Finally, \eqref{f-eng-vml-sum} follows from \eqref{L2f VML}, \eqref{H1f VML} and \eqref{H2f VML}, this ends the proof of Proposition \ref{f-eng-vml}.
\end{proof}

%%%%%%%%%%%%%%%%%%%%%%%%%%%%%%%%%%%%%%%%%%%%%%%%%%%%%%%%%%%%%%%%%%%%%%%%%%%%%%%%%%
%\subsection{Weighted Energy Estimates of $h^{\varepsilon}$} \label{Sec:Energy-h}
%%%%%%%%%%%%%%%%%%%%%%%%%%%%%%%%%%%%%%%%%%%%%%%%%%%%%%%%%%%%%%%%%%%%%%%%%%%%%%%%%%

%\setcounter{equation}{0}

%In this section, we derive the $L^2$ energy estimates for the remainders $h^{\varepsilon}$.

As what we have done in Section \ref{H-LD}, in order to close the estimate \eqref{f-eng-vml-sum} constructed in Proposition \ref{f-eng-vml}, one further needs to deduce the $H^2$ estimates of $h^\vps$ which is determined by \eqref{VMLh-2} and \eqref{VML-id-pt}. However, compared to the case of the Landau equation \eqref{LE}, the estimate of $h^{\varepsilon}$ for the VML system \eqref{main1} is more complicated. Firstly, the velocity growth terms such as $\frac{E\cdot v }{ 2RT_c}h^{\varepsilon}$ cannot be dominated by the degenerate dissipation of the linearized Landau collision operator at large velocity, to overcome this difficulty, a time-dependent exponential weight $\exp\left(\frac{(1+|v|^2)}{8RT_c\ln(\mathrm{e}+t)}\right)$, which leads to extra dissipation term, is introduced so that these velocity growth can be eliminated. Secondly, in order to handle the terms like  $-\big(E+v \times B \big)\cdot\nabla_vh^{\varepsilon}$, we also design a polynomial velocity weight function $\lag v\rag^{\ell-i}$ with $i$ being  the order  of $x-$derivatives to compensate the weak dissipation of the linearized Landau operator. Nevertheless, this  polynomial velocity weight also cause velocity growth while estimating the nonlinear collision operator $\varepsilon^{k-1}\Gamma ( h^{\varepsilon}, h^{\varepsilon} )$. Fortunately, these growths can be absorbed by the exponential weight in the sense that
\begin{align}\label{pevi}
\sup_{v\in {\mathbb R}^3}\Big[{\lag v\rag}^i\exp\left(\frac{-\epsilon_1(1+|v|^2)}{8RT_c\ln(\mathrm{e}+t)}\right)\Big]\lesssim \Big(\frac{\ln(\mathrm{e}+t)}{\epsilon_1}\Big)^{\frac{i}{2}}.
\end{align}
% and transform these polynomial weighted norms into the  $\|wh^{\varepsilon}\|_{H^2}$ norm or $\|wh^{\varepsilon}\|_{H^2_D}$ norm
%to balance their weights and orders  of $x-$ derivatives of $h^{\varepsilon}$.

Before estimating $h^{\varepsilon}$, we give some useful lemmas which will be used later, the first one is concerned with the velocity growth caused by the electric field.
\begin{lemma}\label{dxlmh VML}
Assume that $f^{\varepsilon}, E_R^{\varepsilon}, B_R^{\varepsilon}$ and $h^{\varepsilon}$ are smooth solutions of the Cauchy problem \eqref{VMLf-2}, \eqref{fM-2}, \eqref{VMLh-2} and \eqref{VML-id-pt} for the VML system \eqref{main1} and satisfy \eqref{aps-vml}, then
for $i=0, 1, 2$,  it holds that
  \begin{align}%\label{EBdxih}
   \Big|\Big\langle \nabla_x^i\Big[\frac{E\cdot v }{ 2RT_c}h^{\varepsilon}\Big],  w_i^2 \nabla_x^ih^{\varepsilon}\Big\rangle\Big|\lesssim \epsilon_1Y(t)\|{\lag v\rag}w_i \nabla_x^ih^{\varepsilon}\|^2+\epsilon_1(1+t)^{-p_0}{\bf 1}_{i\geq1}\|w h^{\varepsilon}\|^2_{H^{i-1}},\notag
\end{align}
   \begin{align}%\label{EBRdxih}
   \varepsilon^k\Big|\Big\langle \nabla_x^i\Big[\frac{E_R^{\varepsilon}\cdot v }{ 2RT_c}h^{\varepsilon}\Big],  w_i^2 \nabla_x^ih^{\varepsilon}\Big\rangle\Big|
    \lesssim\,\varepsilon^{\frac{1}{2}}Y(t)\big\|{\lag v\rag}w_i \nabla_x^i h^{\varepsilon}\big\|^2+{\bf 1}_{i\geq1}\varepsilon\|w h^{\varepsilon}\|^2_{H^i},\notag
   \end{align}
  and
   \begin{align}\label{EBndxih}
   &\sum_{n=1}^{2k-1}\varepsilon^n\Big|\Big\langle \nabla_x^i\Big[\frac{E_n\cdot v }{ 2RT_c}h^{\varepsilon}\Big],  w_i^2 \nabla_x^ih^{\varepsilon}\Big\rangle\Big|
    \lesssim\,\varepsilon^{\frac{1}{6}}Y(t)\|{\lag v\rag}w_i \nabla_x^ih^{\varepsilon}\|^2+\varepsilon^{\frac{2}{3}}{\bf 1}_{i\geq1}\|w h^{\varepsilon}\|^2_{H^{i-1}}.
   \end{align}
\end{lemma}
\begin{proof} For brevity,  we only verify \eqref{EBndxih}, the other two inequalities can be treated similarly.
To this end, by \eqref{em-decay}, one has for $t\leq \varepsilon^{-\frac{1}{3}}$ that
\begin{align*}
\sum_{n=1}^{2k-1}\varepsilon^n\Big|\Big\langle \nabla_x^i\Big[\frac{E_n\cdot v }{ 2RT_c}h^{\varepsilon}\Big],  w_i^2 \nabla_x^ih^{\varepsilon}\Big\rangle\Big|
\lesssim& \sum_{n=1}^{2k-1}\varepsilon^n(1+t)^{n}\Big(\|{\lag v\rag}w_i \nabla_x^ih^{\varepsilon}\|^2+{\bf I}_{i\geq1}\|w h^{\varepsilon}\|^2_{H^{i-1}}\Big)\\
\lesssim&\,\varepsilon(1+t)\Big(\|{\lag v\rag}w_i \nabla_x^ih^{\varepsilon}\|^2+{\bf I}_{i\geq1}\|w h^{\varepsilon}\|^2_{H^{i-1}}\Big)\\
\lesssim&\varepsilon^{\frac{1}{6}}Y(t)\|{\lag v\rag}w_i \nabla_x^ih^{\varepsilon}\|^2+\varepsilon^{\frac{2}{3}}{\bf I}_{i\geq1}\|w h^{\varepsilon}\|^2_{H^{i-1}}.
\end{align*}
Thus the proof of Lemma \eqref{dxlmh VML} is finished.
\end{proof}

The following Lemma is devote to the control of the velocity growth caused by Lorentz force.
\begin{lemma}\label{key-VML-2} Assume that $f^{\varepsilon}, E_R^{\varepsilon}, B_R^{\varepsilon}$ and $h^{\varepsilon}$ are smooth solutions of the Cauchy problem \eqref{VMLf-2}, \eqref{fM-2}, \eqref{VMLh-2} and \eqref{VML-id-pt} for the VML system \eqref{main1} and satisfy \eqref{aps-vml}, then
for $i=0, 1, 2$, it holds that
  \begin{align}\label{EBvdxih}
   &\Big|\Big\langle \nabla_x^i\Big[\Big(E+v \times B \Big)\cdot\nabla_vh^{\varepsilon}\Big],  w_i^2 \nabla_x^ih^{\varepsilon}\Big\rangle\Big|\lesssim\epsilon_1\Big(Y(t)\|{\lag v\rag}w_i \nabla_x^ih^{\varepsilon}\|^2+{\bf I}_{i\geq1}\|w h^{\varepsilon}\|^2_{H^{i-1}_D}\Big),
   \end{align}
   \begin{align}\label{EBRvdxih}
   &\varepsilon^k\Big|\Big\langle \nabla_x^i\Big[\big(E_R^{\varepsilon}+v \times B_R^{\varepsilon}\big) \cdot\nabla_vh^{\varepsilon}\Big], w_i^2 \nabla_x^ih^{\varepsilon}\Big\rangle\Big|
    \lesssim\,\varepsilon^{\frac{1}{2}}Y(t)\big\|{\lag v\rag}w_i \nabla_x^i h^{\varepsilon}\big\|^2+\varepsilon\|w h^{\varepsilon}\|^2_{H^i_D},
   \end{align}
  and
   \begin{align}\label{EBnvdxih}
   \sum_{n=1}^{2k-1}\varepsilon^n\Big|\Big\langle \nabla_x^i\big[\big(E_n&+v \times B_n \big)\cdot\nabla_vh^{\varepsilon}+\big(E_R^{\varepsilon}+v \times B_R^{\varepsilon} \big)\cdot\nabla_v F_n\big],  w_i^2 \nabla_x^ih^{\varepsilon}\Big\rangle\Big|\\
       \lesssim&\,\varepsilon^{\frac{1}{6}}Y(t)\|{\lag v\rag}w_i \nabla_x^ih^{\varepsilon}\|^2+\varepsilon^{\frac{2}{3}}\Big(\|E_R^{\varepsilon}\|_{H^i}^2+\|B_R^{\varepsilon}\|_{H^i}^2+{\bf I}_{i\geq1}\|w h^{\varepsilon}\|^2_{H^{i-1}_D}\Big).\nonumber
   \end{align}
\end{lemma}
\begin{proof} We will prove \eqref{EBvdxih} and \eqref{EBRvdxih} and omit the proof of \eqref{EBnvdxih} for brevity.
To prove \eqref{EBvdxih}, from \eqref{em-decay} and Lemma \ref{lower norm}, it follows that
\begin{align*}
&\big|\big\langle \nabla_x^i\big[\big(E+v \times B \big)\cdot\nabla_vh^{\varepsilon}\big],w^2_i \nabla_x^i h^{\varepsilon}\big\rangle\big| \\
\lesssim& \sum_{|\alpha+\alpha'|=i}{\bf I}_{i\geq1}\big|\big\langle \partial^\alpha E\cdot\nabla_v\partial^{\alpha'}h^{\varepsilon},w^2_i \nabla_x^i h^{\varepsilon}\big\rangle\big|
+\big|\big\langle E \cdot\nabla_v\big(w^2_i\big) \nabla_x^i h^{\varepsilon},  \nabla_x^ih^{\varepsilon} \big\rangle\big|\\
&+\sum_{|\alpha+\alpha'|=i}{\bf I}_{i\geq1}\big|\big\langle\big( v \times \partial^{\alpha}B \big)\cdot\big({\bf I}-{\bf P}_v)[\nabla_v\partial^{\alpha'}h^{\varepsilon}], w^2_i \nabla_x^i h^{\varepsilon}\big\rangle\big|\\
\lesssim & \Big(\|E\|_{W^{i,\infty}}+\|B\|_{W^{i,\infty}}\Big)\Big[\big\|\sqrt{1+|v|}w_i\nabla_x^i h^{\varepsilon}\big\|^2
+{\bf I}_{i\geq1}\|{\lag v\rag}w_i \nabla_x^ih^{\varepsilon}\|\\
&\times\Big(\big\|{\lag v\rag}^{-1}w_i\nabla_v h^{\varepsilon}\big\|_{H^{i-1}}+\Big\|\frac{|v|}{{\lag v\rag}}w_i\big({\bf I}-{\bf P}_v)[\nabla_v h^{\varepsilon}]\Big\|_{H^{i-1}}\Big)\Big]\\
\lesssim& \epsilon_1(1+t)^{-p_0}\Big(\|{\lag v\rag}w_i \nabla_x^ih^{\varepsilon}\|^2+\|w h^{\varepsilon}\|^2_{H^{i-1}_D}\Big)\\
\lesssim&
\epsilon_1\Big(Y(t)\|{\lag v\rag}w_i \nabla_x^ih^{\varepsilon}\|^2 +{\bf I}_{i\geq1}\|w h^{\varepsilon}\|^2_{H^{i-1}_D}\Big).
\end{align*}
For \eqref{EBRvdxih}, one has
\begin{align*}
&\varepsilon^k \Big|\Big\langle \nabla_x^i\Big[\Big(E_R^{\varepsilon}+v \times B_R^{\varepsilon} \Big)\cdot\nabla_vh^{\varepsilon}\Big], w^2_i \nabla_x^i h^{\varepsilon}\Big\rangle\Big|\\
    \lesssim&\,\varepsilon^k\sum_{|\alpha+\alpha'|=i}{\bf I}_{i\geq1}\big|\big\langle \partial^{\alpha}E_R^{\varepsilon}\cdot\nabla_v\partial^{\alpha'}h^{\varepsilon},w^2_i \nabla_x ^i h^{\varepsilon}\big\rangle\big|
+\varepsilon^k\big|\big\langle E_R^{\varepsilon} \cdot\nabla_v\big(w^2_i\big)  \nabla_x^i h^{\varepsilon},  \nabla_x^i h^{\varepsilon}\big\rangle\big|\\
&+\varepsilon^k\sum_{|\alpha+\alpha'|=i}{\bf I}_{i\geq1}\big|\big\langle\big( v \times \partial^{\alpha}B_R^{\varepsilon}\big)\cdot\big({\bf I}-{\bf P}_v)[\nabla_v\partial^{\alpha'}h^{\varepsilon}], w^2_i \nabla_x^i h^{\varepsilon}\big\rangle\big|\\
\lesssim &\varepsilon^k\sum_{|\alpha+\alpha'|=i}{\bf I}_{i\geq1}\|\partial^{\alpha}E_R^{\varepsilon}\|_{L^6} \big\|{\lag v\rag}^{-1}w_1\nabla_v\partial^{\alpha'}h^{\varepsilon}\big\|_{L^2_vL^3_x}\|{\lag v\rag}w_i\nabla_x^i h^{\varepsilon}\|\\
&+\varepsilon^k\|E_R^{\varepsilon}\|_{H^2} \big\|{\lag v\rag}w_i \nabla_x^i h^{\varepsilon}\big\|^2\\
+&\varepsilon^k\sum_{|\alpha+\alpha'|=i}{\bf I}_{i\geq1}\|\partial^{\alpha}B_R^{\varepsilon}\|_{L^6} \Big\|\frac{|v|}{{\lag v\rag}}w_i\big({\bf I}-{\bf P}_v)[\nabla_v \partial^{\alpha'}h^{\varepsilon}]\Big\|_{L^2_vL^3_x}\|{\lag v\rag}w_i\nabla_x^i h^{\varepsilon}\|.
\end{align*}
Furthermore, by Lemma \ref{lower norm}, for $|\alpha'|<i$ and $1\leq i\leq 2$, we can get that
\begin{align*}
\big\|{\lag v\rag}^{-1}w_i\nabla_v\partial^{\alpha'}h^{\varepsilon}\big\|_{L^2_vL^3_x}^2=&\int_{{\mathbb R}^3_v} {\lag v\rag}^{-2}w_i^2|\nabla_v\partial^{\alpha'}h^{\varepsilon}|^2_{L^3_x}dv\\
\leq& \int_{{\mathbb R}^3_v} \Big({\lag v\rag}^{-\frac{3}{2}}w_i|\nabla_v\partial^{\alpha'}h^{\varepsilon}|_{L^6_x}\Big)\Big({\lag v\rag}^{-\frac{1}{2}}w_i|\nabla_v\partial^{\alpha'}h^{\varepsilon}|_{L^2_x}\Big)\\
\lesssim& \left\|w_{i-1}\nabla_x^{|\alpha'|} h^{\varepsilon}\right\|_D \left\|w_i\nabla_x^{|\alpha'|+1}h^{\varepsilon}\right\|_D,
\end{align*}
and similarly,
\begin{align*}
\Big\|\frac{|v|}{{\lag v\rag}}&w_i\big({\bf I}-{\bf P}_v)[\nabla_v \partial^{\alpha'}h^{\varepsilon}]\Big\|_{L^2_vL^3_x}^2\leq\int_{{\mathbb R}^3_v} w_i^2|\big({\bf I}-{\bf P}_v)[\nabla_v \partial^{\alpha'}h^{\varepsilon}]|^2_{L^3_x}dv\\
\leq& \int_{{\mathbb R}^3_v} \Big({\lag v\rag}^{-\frac{1}{2}}w_i|\big({\bf I}-{\bf P}_v)[\nabla_v\partial^{\alpha'} h^{\varepsilon}]|_{L^6_x}\Big)\Big({\lag v\rag}^{-\frac{1}{2}}w_{i-1}|\big({\bf I}-{\bf P}_v)[\nabla_v \partial^{\alpha'}h^{\varepsilon}]|_{L^2_x}\Big)dv\\
\lesssim& \left\|w_{i-1} \nabla_x^{|\alpha'|}h^{\varepsilon}\right\|_D \left\|w_i\nabla_x^{|\alpha'|+1}h^{\varepsilon}\right\|_D.
\end{align*}

Noting that $(1+t)^{-\beta}\lesssim Y(t)$ and $\varepsilon\lesssim \varepsilon^{\frac{1}{2}}Y(t)$ for $t\leq \varepsilon^{-1/3}$, we thus conclude that
\begin{align*}
 &\varepsilon^k \Big|\Big\langle \nabla_x^i\Big[\Big(E_R^{\varepsilon}+v \times B_R^{\varepsilon} \Big)\cdot\nabla_vh^{\varepsilon}\Big], w^2_i \nabla_x^i h^{\varepsilon}\Big\rangle\Big|\\
\lesssim & \varepsilon^k\Big(\|E_R^{\varepsilon}\|_{H^2}+\|B_R^{\varepsilon}\|_{H^2}\Big)\Big( \big\|{\lag v\rag}w_i \nabla_x^i h^{\varepsilon}\big\|^2+ \|w_i\nabla_x^ih^{\varepsilon}\|_D^2+{\bf I}_{i\geq1}\|w h^{\varepsilon}\|_{H^{i-1}_D}^2\Big)\\
\lesssim&
\varepsilon^{\frac{1}{2}}Y(t)\big\|{\lag v\rag}w_i \nabla_x^i h^{\varepsilon}\big\|^2+\varepsilon\|w h^{\varepsilon}\|^2_{H^i_D}.
\end{align*}
This ends the proof of Lemma \ref{key-VML-2}.
\end{proof}

We now state our main result on the weighted energy estimates on $h^\varepsilon$ as follows.
\begin{proposition}\label{h-vml-eng-prop}
Assume that $f^{\varepsilon}, E_R^{\varepsilon}, B_R^{\varepsilon}$ and $h^{\varepsilon}$ are smooth solutions of the Cauchy problem \eqref{VMLf-2}, \eqref{fM-2}, \eqref{VMLh-2} and \eqref{VML-id-pt} for the VML system \eqref{main1} and satisfy \eqref{aps-vml}, then it holds that
\begin{align}\label{h-vml-eng}
\frac{\mathrm{d}}{\mathrm{d} t}\sum\limits_{i=0}^2&\big(\varepsilon^{\frac{4}{3}+i}\|w_i\nabla_x^ih^{\varepsilon}\|^2\big)
    +\sum\limits_{i=0}^2\varepsilon^{\frac{4}{3}+i}Y(t)\big\|{\lag v\rag}w_i\nabla_x^ih^{\varepsilon}\big\|^2
    +\de\sum\limits_{i=0}^2\varepsilon^{\frac{1}{3}+i}\|w_i\nabla_x^ih^{\varepsilon}\|_D^2 \\
    \lesssim&\sum\limits_{i=0}^2\varepsilon^{\frac{1}{3}+i}\|f^{\varepsilon}\|_{H^i}^2
    +\sum\limits_{i=0}^2\varepsilon^{\frac{5}{3}+i}\Big(\|E_R^{\varepsilon}\|_{H^i}^2+\|B_R^{\varepsilon}\|_{H^i}^2
    +\|wh^{\varepsilon}\|_{H^i}^2\Big)+\sum\limits_{i=0}^2\varepsilon^{\frac{10}{3}+i}\|wh^{\varepsilon}\|_{H^2_D}^2
   \nonumber\\
    &+\epsilon_1\sum\limits_{i=0}^1\varepsilon^{\frac{7}{3}+i}\Big((1+t)^{-p_0}\|w h^{\varepsilon}\|_{H^i}^2+\|w h^{\varepsilon}\|_{H^i_D}^2\Big)
    \nonumber\\
    &+\sum\limits_{i=0}^2\varepsilon^{2k+\frac{7}{3}+i}(1+t)^{4k+2}
    +\sum\limits_{i=0}^2\varepsilon^{k+\frac{4}{3}+i}(1+t)^{2k}\|w_i\na_x^ih^{\varepsilon}\|.\nonumber
\end{align}
\end{proposition}

\begin{proof} The proof is divided into three steps.
\vskip 0.2cm
\noindent{\underline{\it Step 1. Basic weighted energy estimate of $h^{\varepsilon}$.}} In this step, we deduce the weighted $L^2$ estimates of $h^{\varepsilon}$. For results in this direction, we have
\begin{align}\label{L2h VML}
    \frac{\mathrm{d}}{\mathrm{d} t}\big(\varepsilon^{\frac{4}{3}}&\|w_0h^{\varepsilon}\|^2\big)+\varepsilon^{\frac{4}{3}}Y\big\|{\lag v\rag}w_0h^{\varepsilon}\big\|^2
    +\delta\varepsilon^{\frac{1}{3}}\|w_0h^{\varepsilon}\|_D^2 \\
    \lesssim&\,\varepsilon^{\frac{1}{3}}\|f^{\varepsilon}\|^2+\varepsilon^{\frac{5}{3}}\Big(\|E_R^{\varepsilon}\|^2+\|B_R^{\varepsilon}\|^2
    +\|w_0h^{\varepsilon}\|^2\Big)+\varepsilon^{\frac{10}{3}}\|wh^{\varepsilon}\|_{H^2_D}^2\nonumber\\
    &+\varepsilon^{2k+\frac{7}{3}}(1+t)^{4k+2}+\varepsilon^{k+\frac{4}{3}}(1+t)^{2k}\|w_0h^{\varepsilon}\|.\nonumber
\end{align}
To prove this,
we first take the $L^2$ inner product of $w^2_0 h^{\varepsilon}$ with \eqref{VMLh-2} and use \eqref{wLL}, to obtain
\begin{align}\label{L2h1 VML}
    \frac{1}{2}\frac{\mathrm{d}}{\mathrm{d} t}\|w_0h^{\varepsilon}\|^2&+Y\|{\lag v\rag}w_0h^{\varepsilon}\|^2
    +\frac{\delta}{\varepsilon}\|w_0h^{\varepsilon}]\|^2_D \\
    \leq&\;\frac{C}{\varepsilon}\|f^{\varepsilon}\|^2+\frac{1}{\varepsilon}\big|\big\langle \mathcal{L}_d[ h^{\varepsilon}], w^2_0h^{\varepsilon}\big\rangle\big|+\Big|\Big\langle \big(E_R^{\varepsilon}+v \times B_R^{\varepsilon} \big) \cdot \frac{v-u}{ RT}\mu^{-\frac{1}{2}}\mathbf{M},w^2_0 h^{\varepsilon}\Big\rangle\Big| \nonumber\\
    & +\Big|\Big\langle \big(E+v \times B \big) \cdot \nabla_v h^{\varepsilon}, w^2_0 h^{\varepsilon}\Big\rangle\Big|  +\Big|\Big\langle \frac{E\cdot v }{ 2RT_c}h^{\varepsilon}, w^2_0 h^{\varepsilon}\Big\rangle\Big|\nonumber\\
    &+\varepsilon^{k-1}\big|\big\langle\Gamma ( h^{\varepsilon},
    h^{\varepsilon} ), w^2_0 h^{\varepsilon}\big\rangle\big|+\sum_{i=1}^{2k-1}\varepsilon^{i-1}\big|\big\langle[\Gamma(\mu^{-\frac{1}{2}}F_i, h^{\varepsilon})
    +\Gamma(
 h^{\varepsilon}, \mu^{-\frac{1}{2}} F_i)], w^2_0 h^{\varepsilon}\big\rangle\big|\nonumber\\
 &+\varepsilon^k \big|\big\langle \Big(E_R^{\varepsilon}+v \times B_R^{\varepsilon} \Big)\cdot\nabla_vh^{\varepsilon},w^2_0 h^{\varepsilon}\big\rangle\big|+\varepsilon^k \big|\big\langle \frac{E_R^{\varepsilon} \cdot v}{ 2RT_c}h^{\varepsilon},w^2_0 h^{\varepsilon}\big\rangle\big|\nonumber\\
 &+\sum_{i=1}^{2k-1}\varepsilon^i\Big|\Big\langle \big(E_i+v \times B_i \Big)\cdot\nabla_vh^{\varepsilon}+\big(E_R^{\varepsilon}+v \times B_R^{\varepsilon} \big)\cdot\nabla_v F_i,w^2_0 h^{\varepsilon}\Big\rangle\Big|\nonumber\\
 &+\sum_{i=1}^{2k-1}\varepsilon^i\Big|\Big\langle \frac{E_i\cdot v}{ 2RT_c}h^{\varepsilon},w^2_0 h^{\varepsilon}\Big\rangle\Big|
  +\big|\big\langle \CQ_1,w^2_0 h^{\varepsilon}\big\rangle\big|.\nonumber
\end{align}
We now turn to estimate the right hand side of \eqref{L2h1 VML} term by term.
For the 2nd term on the R.H.S. of \eqref{L2h1 VML},
\eqref{wGLd} directly gives
\begin{align*}
\frac{1}{\varepsilon}\big|\big\langle \mathcal{L}_d[ h^{\varepsilon}], w^2_0h^{\varepsilon}\big\rangle\big|\lesssim \frac{\epsilon_1}{\varepsilon} \|w_0h^{\varepsilon}]\|^2_D.
\end{align*}
For the 3rd term on the R.H.S. of \eqref{L2h1 VML}, from \eqref{tt0},
it follows that ${\lag v\rag}w_0\mu^{-\frac{1}{2}}\mathbf{M}$ is bounded,
therefore, we have
\begin{align*}
&\Big|\Big\langle \big(E_R^{\varepsilon}+v \times B_R^{\varepsilon} \big) \cdot \frac{v-u}{ RT}\mu^{-\frac{1}{2}}\mathbf{M},w^2_0 h^{\varepsilon}\Big\rangle\Big|
\lesssim \frac{o(1)}{\varepsilon}\|w_0h^{\varepsilon}]\|^2_D +\varepsilon\Big(\|E_R^{\varepsilon}\|^2+\|B_R^{\varepsilon}\|^2\Big).
\end{align*}
Since $\varepsilon \ll \epsilon_1$, by Lemma \ref{key-VML-2}, we see that the 4th, 8th and 10th terms on the R.H.S. of \eqref{L2h1 VML} can be bounded by
\begin{align*}
&C\epsilon_1 Y\|{\lag v\rag}w_0h^{\varepsilon}\|^2+C\varepsilon^{\frac{2}{3}}\Big(\|E_R^{\varepsilon}\|^2+\|B_R^{\varepsilon}\|^2\Big).
\end{align*}
For the 5th, 9th and 11th terms on the R.H.S. of \eqref{L2h1 VML}, using Lemma \ref{dxlmh VML},  one sees that they share the same upper bound $C\epsilon_1Y\|{\lag v\rag}w_0 h^{\varepsilon}\|^2$.

For the 6th term on the R.H.S. of \eqref{L2h1 VML},
we get from \eqref{wGG} and Sobolev's inequality that
\begin{align}\label{gah-1}
    &\varepsilon^{k-1}\big|\big\langle\Gamma ( h^{\varepsilon},
    h^{\varepsilon} ), w^2_0  h^{\varepsilon}\big\rangle\big|\nonumber\\
    \lesssim&\varepsilon^{k-1}\int_{{\mathbb R}^3}\Big(\big|{\lag v\rag}^{\ell}h^{\varepsilon}\big|_{L^2}|w_0h^{\varepsilon}|_D
    +|w_0h^{\varepsilon}|_{L^2}\big|{\lag v\rag}^{\ell}h^{\varepsilon}\big|_D\Big)|w_0h^{\varepsilon}|_D\, \\
    \lesssim&\varepsilon^{k-1}\Big(\big\|{\lag v\rag}^{\ell}h^{\varepsilon}\big\|_{H^2}\|w_0h^{\varepsilon}\|_D^2
    +\|w_0h^{\varepsilon}\|_{L^2}\big\|{\lag v\rag}^{\ell}h^{\varepsilon}\big\|_{H^2_D}\|w_0h^{\varepsilon}\|_D \Big).\notag
\end{align}
Here we have used Sobolev's inequality $\|f\|_{L^{\infty}_x}\lesssim \|f\|_{H^2_x}$ to bound both $\big|{\lag v\rag}^{\ell}h^{\varepsilon}\big|_{L^2}$ and $\big|{\lag v\rag}^{\ell}h^{\varepsilon}\big|_D$, because either $|w_0h^{\varepsilon}|_{L^2}$ or $|w_0h^{\varepsilon}|_D$ cannot bear more $x-$derivative.
Furthermore, using \eqref{pevi}, one has
\begin{align}\label{zh1}
\big\|{\lag v\rag}^{\ell}h^{\varepsilon}\big\|_{H^2}
\lesssim& \sup_{v\in {\mathbb R}^3}\Big[{\lag v\rag}^2\exp\left(\frac{-\epsilon_1(1+|v|^2)}{8RT_c\ln(\mathrm{e}+t)}\right)\Big] \big\|wh^{\varepsilon}\big\|_{H^2}\lesssim \frac{1}{\epsilon_1}\ln(\mathrm{e}+t) \|wh^{\varepsilon}\|_{H^2},
\end{align}
and similarly
\begin{align}\label{zh2}
\big\|{\lag v\rag}^{\ell}h^{\varepsilon}\big\|_{H^2_D}\lesssim& \sup_{v\in {\mathbb R}^3}\Big[{\lag v\rag}^2\exp\left(\frac{-\epsilon_1(1+|v|^2)}{8RT_c\ln(\mathrm{e}+t)}\right)\Big] \big\|wh^{\varepsilon}\big\|_{H^2_D}\lesssim \frac{1}{\epsilon_1}\ln(\mathrm{e}+t) \|wh^{\varepsilon}\|_{H^2_D}.
\end{align}
Moreover, for $k\geq3$, it follows
\begin{align}\label{hf2b VML}
\varepsilon^{k-1}\|wh^{\varepsilon}\|_{H^i}\lesssim \varepsilon^{k-2-\frac{i}{2}}\lesssim \varepsilon^{1-\frac{i}{2}}
\end{align}
according to \eqref{fbd-vml}.

Thus, in view of $t\leq \varepsilon^{-1/3}$ and the assumption $\varepsilon\ll\epsilon_1$, we get by plugging \eqref{zh1}, \eqref{zh2} and \eqref{hf2b VML} into \eqref{gah-1} that
\begin{align*}
    \varepsilon^{k-1}\big|\big\langle\Gamma ( h^{\varepsilon},
    h^{\varepsilon} ), w^2_0  h^{\varepsilon}\big\rangle\big|
     \lesssim& \frac{1}{\epsilon_1}\ln(\mathrm{e}+t)\|w_0h^{\varepsilon}\|_D^2+\frac{o(1)}{\varepsilon}\|w_0h^{\varepsilon}\|_D^2+\frac{1}{\epsilon_1^2}\ln^2(\mathrm{e}+t)\varepsilon^3 \big\|wh^{\varepsilon}\big\|_{H^2_D}\\
           \lesssim& \frac{o(1)}{\varepsilon}\|w_0h^{\varepsilon}\|_D^2
    +\varepsilon^2\|wh^{\varepsilon}\|_{H^2_D}^2.
\end{align*}
For the 7th term on the R.H.S. of \eqref{L2h1 VML},
 \eqref{em-Fn-es} and Lemma \ref{wggm} give
\begin{align*}
   &\sum_{i=1}^{2k-1}\varepsilon^{i-1}\big|\big\langle[\Gamma(\mu^{-\frac{1}{2}}F_i,h^{\varepsilon})+\Gamma(
 h^{\varepsilon}, \mu^{-\frac{1}{2}} F_i)], w^2_0h^{\varepsilon}\big\rangle\big|\\
    \lesssim&\,\sum_{i=1}^{2k-1}[\varepsilon(1+t)]^{i-1}(1+t)\Big(\|w_0h^{\varepsilon}\|+\|w_0h^{\varepsilon}\|_D\Big)\|w_0h^{\varepsilon}\|_D\\
\lesssim&\frac{o(1)}{\varepsilon}\|w_0h^{\varepsilon}\|_D^2+\varepsilon^{\frac{1}{3}}\|w_0h^{\varepsilon}\|^2\nonumber.
\end{align*}
Similarly, for the last term on the R.H.S. of \eqref{L2h1 VML}, we have
\begin{align*}
\big|\big\langle \CQ_1, w^2_0 h^{\varepsilon}\big\rangle\big|\lesssim \frac{\epsilon_1}{\varepsilon} \|w_0h^{\varepsilon}]\|^2_D +\varepsilon^{2k+1}(1+t)^{4k+2}+\varepsilon^{k}(1+t)^{2k}\|w_0h^{\varepsilon}\|.
\end{align*}

Putting the above estimates into \eqref{L2h1 VML} and multiplying the resulting inequality by $\varepsilon^{\frac{4}{3}}$ yield \eqref{L2h VML}.

\vskip 0.2cm
\noindent\underline{{\it Step 2. First order derivative estimates of $h^{\varepsilon}$ with weight.} }
In this step, we continue to deduce the estimate of $\|w_1\nabla_xh^{\varepsilon}\|$. Applying $\partial_x^{\alpha} (1\leq |\alpha|\leq 2)$ to \eqref{VMLh-2} yields
\begin{align}\label{mh1x VML}
 \partial_t\partial_x^{\alpha}h^{\varepsilon}&+v\cdot\nabla_x\partial_x^{\alpha}h^{\varepsilon}+\partial_x^{\alpha}\Big[\frac{\big(E_R^{\varepsilon}+v \times B_R^{\varepsilon} \big) }{ RT}\cdot \big(v-u\big)\mu^{-\frac{1}{2}}\mathbf{M}\Big]\nonumber\\
&+\partial_x^{\alpha}\Big[\frac{E\cdot v }{ 2RT_c}h^{\varepsilon}\Big]-\partial_x^{\alpha}\big[\big(E+v \times B \big)\cdot\nabla_vh^{\varepsilon}\big]+\frac{\partial_x^{\alpha}\mathcal{L}[h^{\varepsilon}]}{\varepsilon}\nonumber\\
 =&-\frac{\partial_x^{\alpha}\mathcal{L}_d[h^{\varepsilon}]}{\varepsilon}+\varepsilon^{k-1}\partial_x^{\alpha}\Gamma(h^{\varepsilon},h^{\varepsilon})
 +\sum_{i=1}^{2k-1}\varepsilon^{i-1}[\partial_x^{\alpha}\Gamma(\mu^{-\frac{1}{2}}F_i, h^{\varepsilon})+\partial_x^{\alpha}\Gamma(h^{\varepsilon}, \mu^{-\frac{1}{2}}F_i)]\nonumber\\
 &+\varepsilon^k \partial_x^{\alpha}\Big[\Big(E_R^{\varepsilon}+v \times B_R^{\varepsilon}\Big)\cdot\nabla_vh^{\varepsilon}\Big]-\varepsilon^k  \partial_x^{\alpha}\Big[\frac{E_R^{\varepsilon}\cdot v}{ 2RT_c}h^{\varepsilon}\Big]\\
 &+\sum_{i=1}^{2k-1}\varepsilon^i\partial_x^{\alpha}\Big[\Big(E_i+v \times B_i \Big)\cdot\nabla_vh^{\varepsilon}+\Big(E_R^{\varepsilon}+v \times B_R^{\varepsilon} \Big)\cdot\nabla_v F_i\Big]\nonumber\\
 &-\sum_{i=1}^{2k-1}\varepsilon^i\partial_x^{\alpha}\Big(\frac{E_i \cdot v}{ 2RT_c}h^{\varepsilon}\Big)+\varepsilon^{k}\partial_x^{\alpha}\CQ_1.
\nonumber
\end{align}
Next, taking the $L^2$ inner product of $w^2_1 \partial_x^\alpha h^{\varepsilon}$ and \eqref{mh1x VML} with $|\al|=1$ and applying \eqref{wLL}, one has
\begin{align}\label{H1h1 VML}
&\frac{1}{2}\frac{\mathrm{d}}{\mathrm{d} t}\|w_1 \nabla_xh^{\varepsilon}\|^2+Y(t)\big\|{\lag v\rag}w_1\nabla_xh^{\varepsilon}\big\|^2
    +\frac{\delta}{\varepsilon}\|w_1 \nabla_xh^{\varepsilon}]\|^2_D \\
    \leq&\;\frac{C}{\varepsilon}\|f^{\varepsilon}\|^2_{H^1}+\frac{1}{\varepsilon}\big|\big\langle \nabla_x\mathcal{L}_d[ h^{\varepsilon}], w^2_1 \nabla_xh^{\varepsilon}\big\rangle\big|\nonumber\\
    &+\Big|\Big\langle \nabla_x\Big[\big(E_R^{\varepsilon}+v \times B_R^{\varepsilon} \big) \cdot \frac{v-u}{ RT}\mu^{-\frac{1}{2}}\mathbf{M}\Big], w^2_1 \nabla_xh^{\varepsilon}\Big\rangle\Big|
    +\Big|\Big\langle \nabla_x\Big[\frac{E\cdot v }{ 2RT_c}h^{\varepsilon}\Big], w^2_1 \nabla_xh^{\varepsilon}\Big\rangle\Big|\nonumber\\
    &  +\big|\big\langle \nabla_x\big[\big(E+v \times B \big)\cdot\nabla_vh^{\varepsilon}\big],w^2_1 \nabla_x h^{\varepsilon}\big\rangle\big|  +\varepsilon^{k-1}\big|\big\langle \nabla_x\Gamma ( h^{\varepsilon},
    h^{\varepsilon} ), w^2_1 \nabla_x h^{\varepsilon}\big\rangle\big|\nonumber\\
    &+\sum_{i=1}^{2k-1}\varepsilon^{i-1}\big|\big\langle[\nabla_x\Gamma(\mu^{-\frac{1}{2}}F_i,h^{\varepsilon})
    +\nabla_x\Gamma(
 h^{\varepsilon}, \mu^{-\frac{1}{2}} F_i)], w^2_1 \nabla_x h^{\varepsilon}\big\rangle\big|\nonumber\\
 &+\varepsilon^k \big|\big\langle \nabla_x\Big[\Big(E_R^{\varepsilon}+v \times B_R^{\varepsilon} \Big)\cdot\nabla_vh^{\varepsilon}\Big], w^2_1 \nabla_x h^{\varepsilon}\big\rangle\big|+\varepsilon^k \Big|\Big\langle \nabla_x\Big[\frac{E_R^{\varepsilon} \cdot v}{ 2RT_c}h^{\varepsilon}\Big], w^2_1 \nabla_x h^{\varepsilon}\Big\rangle\Big|\nonumber\\
 &+\sum_{i=1}^{2k-1}\varepsilon^i\Big|\Big\langle \nabla_x\Big[\big(E_i+v \times B_i \big)\cdot\nabla_vh^{\varepsilon}+\big(E_R^{\varepsilon}+v \times B_R^{\varepsilon} \big)\cdot\nabla_v F_i\Big], w^2_1 \nabla_x h^{\varepsilon}\Big\rangle\Big|\nonumber\\
 &+\sum_{i=1}^{2k-1}\varepsilon^i\Big|\Big\langle \nabla_x\Big(\frac{E_i\cdot v}{ 2RT_c}h^{\varepsilon}\Big), w^2_1 \nabla_x h^{\varepsilon}\Big\rangle\Big|
  +\big|\big\langle \nabla_x\CQ_1, w^2_1 \nabla_x h^{\varepsilon}\big\rangle\big|.\nonumber
\end{align}
We now turn to estimate the R.H.S. of \eqref{H1f1 VML} individually.
For the 2nd term on the R.H.S. of \eqref{H1f1 VML}, by \eqref{wGLd}, we obtain
\begin{align*}
\frac{1}{\varepsilon}\big|\big\langle \nabla_x\mathcal{L}_d[ h^{\varepsilon}], w^2_1 \nabla_xh^{\varepsilon}\big\rangle\big|\lesssim \frac{\epsilon_1 }{\varepsilon}\Big(\|w_1 \nabla_xh^{\varepsilon}\|^2_D+\|w_1 h^{\varepsilon}\|^2_D\Big).
\end{align*}
For the 3rd term on the R.H.S. of \eqref{H1f1 VML}, direct calculation gives
\begin{align*}
 \frac{o(1)}{\varepsilon}\|w_1 \nabla_xh^{\varepsilon}\|^2_D+C\varepsilon\Big(\|E_R^{\varepsilon}\|_{H^1}^2+\|B_R^{\varepsilon}\|_{H^1}^2\Big).
\end{align*}
Applying Lemma \ref{dxlmh VML}, we can bound the 4th, 9th and 11th terms on the R.H.S. of \eqref{H1f1 VML} by
\begin{align*}
\epsilon_1Y(t)\|{\lag v\rag}w_1 \nabla_xh^{\varepsilon}\|^2+\Big[(1+t)^{-p_0}+\varepsilon^{\frac{2}{3}}\Big]\|w_0 h^{\varepsilon}\|^2.
\end{align*}
For the fifth, eighth and tenth term on the R.H.S. of \eqref{H1f1 VML},
assuming $\varepsilon\ll\epsilon_1$ and using Lemma \ref{key-VML-2}, we see that they enjoy the same upper bound
\begin{align*}
C\epsilon_1Y(t)\|{\lag v\rag}w_1 \nabla_xh^{\varepsilon}\|^2+C\varepsilon^{\frac{2}{3}}\Big(\|E_R^{\varepsilon}\|_{H^1}^2+\|B_R^{\varepsilon}\|_{H^1}^2\Big)+\epsilon_1\|w h^{\varepsilon}\|^2_{H^{1}_D}.
\end{align*}
For the 6th term on the R.H.S. of \eqref{H1h1 VML},
 we get from \eqref{wGG} and Sobolev's inequalities tha
\begin{align*}
&\varepsilon^{k-1}\big|\big\langle\nabla_x\Gamma ( h^{\varepsilon},
    h^{\varepsilon} ), w^2_1 \nabla_x h^{\varepsilon}\big\rangle\big|\\
    \lesssim&\varepsilon^{k-1}\int_{{\mathbb R}^3}\Big(\big|{\lag v\rag}^{\ell-1}h^{\varepsilon}\big|_{L^2}|w_1\nabla_xh^{\varepsilon}|_D
    +|w_1\nabla_xh^{\varepsilon}|_{L^2}\big|{\lag v\rag}^{\ell-1}h^{\varepsilon}\big|_D\Big)|w_1\nabla_xh^{\varepsilon}|_D\, \nonumber\\
    &+\varepsilon^{k-1}\int_{{\mathbb R}^3}\Big(\big|{\lag v\rag}^{\ell-1}\nabla_xh^{\varepsilon}\big|_{L^2}|w_1h^{\varepsilon}|_D
    +|w_1h^{\varepsilon}|_{L^2}\big|{\lag v\rag}^{\ell-1}\nabla_xh^{\varepsilon}\big|_D\Big)|w_1\nabla_xh^{\varepsilon}|_D\, \nonumber\\
    \lesssim&\varepsilon^{k-1}\Big(\big\|{\lag v\rag}^{\ell-1}h^{\varepsilon}\big\|_{H^2}\|w_1\nabla_xh^{\varepsilon}\|_D^2
    +\|w_1\nabla_xh^{\varepsilon}\|_{L^2}\big\|{\lag v\rag}^{\ell-1}h^{\varepsilon}\big\|_{H^2_D}\|w_1\nabla_xh^{\varepsilon}\|_D \Big).
\end{align*}
Next, similar to \eqref{zh1} and \eqref{zh2}, one has
\begin{align*}
\big\|{\lag v\rag}^{\ell-1}h^{\varepsilon}\big\|_{H^2}\lesssim& \sup_{v\in {\mathbb R}^3}\Big[{\lag v\rag}\exp\left(\frac{-\epsilon_1(1+|v|^2)}{8RT_c\ln(\mathrm{e}+t)}\right)\Big] \big\|wh^{\varepsilon}\big\|_{H^2}
\lesssim \frac{\sqrt{\ln(\mathrm{e}+t)}}{\sqrt{\epsilon_1}} \|wh^{\varepsilon}\|_{H^2},
\end{align*}
and
\begin{align*}
\big\|{\lag v\rag}^{\ell-1}h^{\varepsilon}\big\|_{H^2_D}\lesssim& \sup_{v\in {\mathbb R}^3}\Big[{\lag v\rag}\exp\left(\frac{-\epsilon_1(1+|v|^2)}{8RT_c\ln(\mathrm{e}+t)}\right)\Big] \big\|wh^{\varepsilon}\big\|_{H^2_D}
\lesssim \frac{\sqrt{\ln(\mathrm{e}+t)}}{\sqrt{\epsilon_1}} \|wh^{\varepsilon}\|_{H^2_D}.
\end{align*}
These estimates together with \eqref{hf2b VML} leads us to
\begin{align*}
&\varepsilon^{k-1}\big|\big\langle\nabla_x\Gamma ( h^{\varepsilon},
    h^{\varepsilon} ), w^2_1\nabla_x  h^{\varepsilon}\big\rangle\big|\\
    \lesssim& \frac{\sqrt{\ln(\mathrm{e}+t)}}{\sqrt{\epsilon_1}}\|w_1\nabla_xh^{\varepsilon}\|_D^2+ \frac{o(1)}{\varepsilon}\|w_1\nabla_xh^{\varepsilon}\|_D^2
    +\frac{\ln(\mathrm{e}+t)}{\epsilon_1}\varepsilon^2\|wh^{\varepsilon}\|_{H^2_D}^2\\
       \lesssim& \frac{o(1)}{\varepsilon}\|w_1\nabla_xh^{\varepsilon}\|_D^2
    +\varepsilon\|wh^{\varepsilon}\|_{H^2_D}^2,
\end{align*}
where \eqref{hf2b VML}, $t\leq \varepsilon^{-1/3}$ and  $\varepsilon\ll\epsilon_1$ are used again.

For the 7th term on the R.H.S. of \eqref{H1h1 VML},
from \eqref{em-Fn-es} and Lemma \ref{wggm}, we get for $t\leq \varepsilon^{-1/3}$ that
\begin{align*}
&\sum_{i=1}^{2k-1}\varepsilon^{i-1}\big|\big\langle[\nabla_x\Gamma(\mu^{-\frac{1}{2}}F_i,h^{\varepsilon})+\nabla_x\Gamma(
 h^{\varepsilon}, \mu^{-\frac{1}{2}} F_i)], w^2_1\nabla_xh^{\varepsilon}\big\rangle\big|\\
    \lesssim&\,\sum_{i=1}^{2k-1}[\varepsilon(1+t)]^{i-1}(1+t)\Big(\|w_1h^{\varepsilon}\|_{H^1}
    +\|wh^{\varepsilon}\|_{H^1_D}\Big)\|w_1\nabla_xh^{\varepsilon}\|_D\\
    \lesssim&\frac{o(1)}{\varepsilon}\|w_1\nabla_xh^{\varepsilon}\|_D^2+\varepsilon^{\frac{1}{3}}\Big(\|wh^{\varepsilon}\|^2_{H^1}
    +\|w_0h^{\varepsilon}\|^2_{D}\Big)\nonumber.
\end{align*}
Similarly, for the last term on the R.H.S. of \eqref{H1h1 VML}, one has
\begin{align*}
\big|\big\langle \nabla_x\CQ_1, w^2_1\nabla_x h^{\varepsilon}\big\rangle\big|\lesssim \frac{\epsilon_1}{\varepsilon} \|w_1\nabla_xh^{\varepsilon}]\|^2_D +\varepsilon^{2k+1}(1+t)^{4k+2}+\varepsilon^{k}(1+t)^{2k}\|w_1\nabla_xh^{\varepsilon}\|.
\end{align*}

Putting the above estimates into \eqref{H1h1 VML} and multiplying the resulting inequality by $\varepsilon^{7/3}$, we arrive at
\begin{align}\label{H1h VML}
&\frac{\mathrm{d}}{\mathrm{d} t}\big(\varepsilon^{\frac{7}{3}}\|w_1\nabla_xh^{\varepsilon}\|^2\big)+\varepsilon^{\frac{7}{3}}Y\big\|{\lag v\rag}w_1\nabla_xh^{\varepsilon}\big\|^2
    +\delta\varepsilon^{\frac{4}{3}}\|w_1\nabla_xh^{\varepsilon}\|_D^2 \\
    \lesssim&\,\varepsilon^{\frac{4}{3}}\|f^{\varepsilon}\|^2_{H^1}+\varepsilon^{\frac{8}{3}}\Big(\|E_R^{\varepsilon}\|^2_{H^1}+\|B_R^{\varepsilon}\|^2_{H^1}
    +\|wh^{\varepsilon}\|^2_{H^1}\Big)+\varepsilon^{\frac{10}{3}}\|wh^{\varepsilon}\|_{H^2_D}^2\nonumber\\
    &+\epsilon_1\varepsilon^{\frac{7}{3}}\Big((1+t)^{-p_0}\|w_0 h^{\varepsilon}\|^2+\|w_0 h^{\varepsilon}\|^2_D\Big)\nonumber\\
    &+\varepsilon^{2k+\frac{10}{3}}(1+t)^{4k+2}+\varepsilon^{k+\frac{7}{3}}(1+t)^{2k}\|w_1\nabla_xh^{\varepsilon}\|.\nonumber
\end{align}

\vskip 0.2cm
\noindent\underline{{\it Step 3. Second order derivative estimates of $h^{\varepsilon}$ with weight.} }
In this final step, we proceed to deduce the estimate of $\|w_2\nabla_x^2h^{\varepsilon}\|$.
For results in this direction, we have
\begin{align}\label{H2h VML}
\frac{\mathrm{d}}{\mathrm{d} t}\big(\varepsilon^{\frac{10}{3}}&\|w_2\nabla_x^2h^{\varepsilon}\|^2\big)+\varepsilon^{\frac{10}{3}}Y\big\|{\lag v\rag}w_2\nabla_x^2h^{\varepsilon}\big\|^2
    +\delta\varepsilon^{\frac{7}{3}}\|w_2\nabla_x^2h^{\varepsilon}\|_D^2 \\
    \lesssim&\,\varepsilon^{\frac{7}{3}}\|f^{\varepsilon}\|^2_{H^2}+\varepsilon^{\frac{11}{3}}\Big(\|E_R^{\varepsilon}\|^2_{H^2}+\|B_R^{\varepsilon}\|^2_{H^2}
    +\|wh^{\varepsilon}\|^2_{H^2}\Big)\nonumber\\
    &+\epsilon_1\varepsilon^{\frac{10}{3}}\Big[(1+t)^{-p_0}\|w h^{\varepsilon}\|^2_{H^1}+\|wh^{\varepsilon}\|_{H^1_D}^2\Big]\nonumber\\
    &+\varepsilon^{2k+\frac{13}{3}}(1+t)^{4k+2}+\varepsilon^{k+\frac{10}{3}}(1+t)^{2k}\|w_2\nabla_x^2h^{\varepsilon}\|.\nonumber
\end{align}
To show this, we first take the $L^2$ inner product of $w^2_2 \partial_x^\alpha h^{\varepsilon}$ and \eqref{mh1x VML} with $|\alpha|=2$ and use  \eqref{wLL} to obtain
\begin{align}\label{H2h1 VML}
&\frac{1}{2}\frac{\mathrm{d}}{\mathrm{d} t}\|w_2 \nabla_x^2h^{\varepsilon}\|^2+Y(t)\big\|{\lag v\rag}w_1\nabla_x^2h^{\varepsilon}\big\|^2
    +\frac{\delta}{\varepsilon}\|w_2 \nabla_x^2h^{\varepsilon}]\|^2_D \\
    \leq&\;\frac{C}{\varepsilon}\|f^{\varepsilon}\|^2_{H^2}+\frac{1}{\varepsilon}\big|\big\langle \nabla_x^2\mathcal{L}_d[ h^{\varepsilon}], w^2_2 \nabla_x^2h^{\varepsilon}\big\rangle\big|\nonumber\\
    &+\Big|\Big\langle \nabla_x^2\Big[\big(E_R^{\varepsilon}+v \times B_R^{\varepsilon} \big) \cdot \frac{v-u}{ RT}\mu^{-\frac{1}{2}}\mathbf{M}\Big], w^2_2 \nabla_x^2h^{\varepsilon}\Big\rangle\Big|+\Big|\Big\langle \nabla_x^2\Big[\frac{E\cdot v }{ 2RT_c}h^{\varepsilon}\Big], w^2_2 \nabla_x^2h^{\varepsilon}\Big\rangle\Big|\nonumber\\
    &  +\big|\big\langle \nabla_x^2\big[\big(E+v \times B \big)\cdot\nabla_vh^{\varepsilon}\big],w^2_2 \nabla_x^2 h^{\varepsilon}\big\rangle\big|  +\varepsilon^{k-1}\big|\big\langle \nabla_x^2\Gamma ( h^{\varepsilon},
    h^{\varepsilon} ), w^2_2 \nabla_x^2 h^{\varepsilon}\big\rangle\big|\nonumber\\
    &+\sum_{i=1}^{2k-1}\varepsilon^{i-1}\big|\big\langle[\nabla_x^2\Gamma(\mu^{-\frac{1}{2}}F_i,h^{\varepsilon})
    +\nabla_x^2\Gamma(
 h^{\varepsilon}, \mu^{-\frac{1}{2}} F_i)], w^2_2 \nabla_x^2 h^{\varepsilon}\big\rangle\big|\nonumber\\
 &+\varepsilon^k \big|\big\langle \nabla_x^2\Big[\Big(E_R^{\varepsilon}+v \times B_R^{\varepsilon} \Big)\cdot\nabla_vh^{\varepsilon}\Big], w^2_2 \nabla_x^2 h^{\varepsilon}\big\rangle\big|+\varepsilon^k \Big|\Big\langle \nabla_x^2\Big[\frac{E_R^{\varepsilon} \cdot v}{ 2RT_c}h^{\varepsilon}\Big], w^2_2 \nabla_x^2 h^{\varepsilon}\Big\rangle\Big|\nonumber\\
 &+\sum_{i=1}^{2k-1}\varepsilon^i\Big|\Big\langle \nabla_x^2\Big[\big(E_i+v \times B_i \big)\cdot\nabla_vh^{\varepsilon}+\big(E_R^{\varepsilon}+v \times B_R^{\varepsilon} \big)\cdot\nabla_v F_i\Big], w^2_2 \nabla_x^2 h^{\varepsilon}\Big\rangle\Big|\nonumber\\
 &+\sum_{i=1}^{2k-1}\varepsilon^i\Big|\Big\langle \nabla_x^2\Big(\frac{E_i\cdot v}{ 2RT_c}h^{\varepsilon}\Big), w^2_2 \nabla_x^2 h^{\varepsilon}\Big\rangle\Big|
  +\big|\big\langle \nabla_x^2\CQ_1, w^2_2 \nabla_x^2 h^{\varepsilon}\big\rangle\big|.\nonumber
\end{align}
We now turn to estimate the R.H.S. of \eqref{H2h1 VML} term by term.
For the 2nd term on the R.H.S. of \eqref{H2h1 VML}, by \eqref{wGLd}, we have
\begin{align*}
\frac{1}{\varepsilon}\big|\big\langle \nabla_x^2\mathcal{L}_d[ h^{\varepsilon}], w^2_2 \nabla_x^2h^{\varepsilon}\big\rangle\big|\lesssim \frac{\epsilon_1 }{\varepsilon}\|w h^{\varepsilon}\|^2_{H^2_D}.
\end{align*}
The 3rd term on the R.H.S. of \eqref{H2h1 VML} can be directly bounded by
\begin{align*}
&\Big|\Big\langle \nabla_x^2\Big[\big(E_R^{\varepsilon}+v \times B_R^{\varepsilon} \big) \cdot \frac{v-u}{ RT}\mu^{-\frac{1}{2}}\mathbf{M}\Big], w^2_2 \nabla_x^2h^{\varepsilon}\Big\rangle\Big|\\
\lesssim& \frac{o(1)}{\varepsilon}\|w_2 \partial_x^2h^{\varepsilon}\|^2_D+\varepsilon\Big(\|E_R^{\varepsilon}\|_{H^2}^2+\|B_R^{\varepsilon}\|_{H^2}^2\Big).
\end{align*}
For the 4th, 9th and 11th term on the R.H.S. of \eqref{H2f1 VML},
in light of Lemma \ref{dxlmh VML}, we can bound them by
\begin{align*}
\epsilon_1Y(t)\|{\lag v\rag}w_2 \nabla_x^2h^{\varepsilon}\|^2+\Big[(1+t)^{-p_0}+\varepsilon^{\frac{2}{3}}\Big]\|w h^{\varepsilon}\|^2_{H^1}.
\end{align*}
Assuming $\varepsilon\ll\epsilon_1$ and using Lemma \ref{key-VML-2}, we see that the fifth, eighth and tenth terms on the R.H.S. of \eqref{H2f1 VML} share the same bound
\begin{align*}
C\epsilon_1Y(t)\|{\lag v\rag}w_2 \nabla_x^2h^{\varepsilon}\|^2+C\varepsilon^{\frac{2}{3}}\Big(\|E_R^{\varepsilon}\|_{H^2}^2+\|B_R^{\varepsilon}\|_{H^2}^2\Big)+\epsilon_1\|w h^{\varepsilon}\|^2_{H^{2}_D}.
\end{align*}
For the 6th term on the R.H.S. of \eqref{H2h1 VML},
applying \eqref{wGG},  \eqref{hf2b VML}  and Sobolev's inequalities, one has
\begin{align*}
    &\varepsilon^{k-1}\big|\big\langle\nabla_x^2\Gamma ( h^{\varepsilon},
    h^{\varepsilon} ), w^2_2 \nabla_x^2 h^{\varepsilon}\big\rangle\big|\\
    \lesssim&\varepsilon^{k-1}\int_{{\mathbb R}^3}\Big(\big|{\lag v\rag}^{\ell-2}h^{\varepsilon}\big|_{L^2}|w_2\nabla_x^2h^{\varepsilon}|_D
    +|w_2\nabla_x^2h^{\varepsilon}|_{L^2}\big|{\lag v\rag}^{\ell-2}h^{\varepsilon}\big|_D\Big)|w_2\nabla_x^2h^{\varepsilon}|_D\, \nonumber\\
    &+\varepsilon^{k-1}\int_{{\mathbb R}^3}\Big(\big|{\lag v\rag}^{\ell-2}\nabla_x^2h^{\varepsilon}\big|_{L^2}|w_2h^{\varepsilon}|_D
    +|w_2h^{\varepsilon}|_{L^2}\big|{\lag v\rag}^{\ell-2}\nabla_x^2h^{\varepsilon}\big|_D\Big)|w_2\nabla_x^2h^{\varepsilon}|_D\, \nonumber\\
    &+\varepsilon^{k-1}\int_{{\mathbb R}^3}\Big(\big|{\lag v\rag}^{\ell-2}\nabla_xh^{\varepsilon}\big|_{L^2}|w_2\nabla_xh^{\varepsilon}|_D
    +|w_2\nabla_xh^{\varepsilon}|_{L^2}\big|{\lag v\rag}^{\ell-2}\nabla_xh^{\varepsilon}\big|_D\Big)|w_2\nabla_x^2h^{\varepsilon}|_D\, \nonumber\\
    \lesssim&\varepsilon^{k-1}\big\|w_2h^{\varepsilon}\big\|_{H^2}\|w_2h^{\varepsilon}\|_{H^2_D}^2\lesssim \|wh^{\varepsilon}\|_{H^2_D}^2.
\end{align*}
For the 7th term on the R.H.S. of \eqref{H2h1 VML},
by \eqref{em-Fn-es} and Lemma \ref{wggm}, we get
\begin{align*}
&\sum_{i=1}^{2k-1}\varepsilon^{i-1}\big|\big\langle[\nabla_x^2\Gamma(\mu^{-\frac{1}{2}}F_i,h^{\varepsilon})+\nabla_x^2\Gamma(
 h^{\varepsilon}, \mu^{-\frac{1}{2}} F_i)], w^2_2\nabla_x^2h^{\varepsilon}\big\rangle\big|\\
    \lesssim&\,\sum_{i=1}^{2k-1}[\varepsilon(1+t)]^{i-1}(1+t)\Big(\|w_2h^{\varepsilon}\|_{H^2}
    +\|w_2h^{\varepsilon}\|_{H^2_D}\Big)\|w_2\nabla_x^2h^{\varepsilon}\|_D\\
    \lesssim&\frac{o(1)}{\varepsilon}\|w_2\nabla_x^2h^{\varepsilon}\|_D^2 +\varepsilon^{\frac{1}{3}}\Big(\|wh^{\varepsilon}\|^2_{H^2}
    +\|wh^{\varepsilon}\|^2_{H^1_D}\Big)\nonumber.
\end{align*}
Similarly, for the last term on the R.H.S. of \eqref{H2h1 VML}, it follows
\begin{align*}
\big|\big\langle \nabla_x^2\CQ_1, w^2_2\nabla_x^2 h^{\varepsilon}\big\rangle\big|\lesssim \frac{o(1)}{\varepsilon} \|w_2\nabla_x^2h^{\varepsilon}]\|^2_D +\varepsilon^{2k+1}(1+t)^{4k+2}+\varepsilon^{k}(1+t)^{2k}\|w_2\nabla_x^2h^{\varepsilon}\|.
\end{align*}

Submitting these estimates in \eqref{H2h1 VML} and multiplying the resulting inequality by $\varepsilon^{10/3}$ give \eqref{H2h VML}.

Finally, \eqref{h-vml-eng} follows from \eqref{L2h VML},  \eqref{H1h VML} and  \eqref{H2h VML}, this ends the proof of Proposition \ref{h-vml-eng-prop}.
\end{proof}

%%%%%%%%%%%%%%%%%%%%%%%%%%%%%%%%%%%%%%%%%%%%%%%%%%%%%%%%%%%%%%%%%%%%%%%%%%%%%%%%%%
\subsection{Proof of the Theorem \ref{resultVML}} \label{Sec:thmVML}
%%%%%%%%%%%%%%%%%%%%%%%%%%%%%%%%%%%%%%%%%%%%%%%%%%%%%%%%%%%%%%%%%%%%%%%%%%%%%%%%%%
We are now ready to complete the Proof of Theorem~\ref{resultVML}. For brevity, we only derive the energy estimates \eqref{thm1} and omit details of the proof of non-negativity proof of $F^{\varepsilon}$ since it is
similar to the proof given in \cite{Ouyang-Wu-Xiao-arxiv-2022-rL, Ouyang-Wu-Xiao-arxiv-2022-rLM}.

For this purpose, combing Propositions \ref{h-vml-eng-prop} and \ref{f-eng-vml}, one has
\begin{align*}
\frac{\mathrm{d}}{\mathrm{d} t}\mathcal{E}(t)
&+\frac{\delta}{2}\sum_{i=0}^2 \varepsilon^i\Big[\frac{1}{\varepsilon}\|\nabla_x^i({\bf I}-{\bf P}_{\mathbf{M}})[f^{\varepsilon}]\|^2_D+\varepsilon^{\frac{1}{3}}\|w_i\nabla_x^i h^{\varepsilon}\|^2_D+\varepsilon^{\frac{4}{3}}Y\|{\lag v\rag}w_i\nabla_x^i h^{\varepsilon}\|^2\Big]\\
\lesssim & \sum_{i=0}^2\varepsilon^{i+\frac{1}{3}}\Big[\|\nabla_x^if^{\varepsilon}\|^2+\Big(\varepsilon^{\frac{4}{3}}
+\varepsilon(1+t)^{-p_0}\Big)\|w_i\nabla_x^ih^{\varepsilon}\|^2+\epsilon_1\varepsilon\|w_i\nabla_x^ih^{\varepsilon}\|^2_D\Big)\\
&+ \varepsilon^{\frac{10}{3}}\|w h^{\varepsilon}\|^2_{H^2_D}+\sum_{i=0}^2\varepsilon^i\Big((1+t)^{-p_0}+\varepsilon^{\frac{1}{3}}\Big)
\Big[\|\nabla_x^if^{\varepsilon}\|^2
     +\|E_R^{\varepsilon}\|^2_{H^i}+\|B_R^{\varepsilon}\|^2_{H^i}
     \nonumber\\
     &+ C_{\epsilon_1}\exp\left(-\frac{\epsilon_1}{8RT^2_c\sqrt{\varepsilon}}\right)
    \Big(\|h^{\varepsilon}\|^2_{H^i}+\|h^{\varepsilon}\|^2_{H^i_D}\Big)\Big]\\
      &+\varepsilon^{2k+1}(1+t)^{4k+2}+\varepsilon^{k}(1+t)^{2k}\sum_{i=0}^2\varepsilon^i\Big(\|\nabla_x^if^{\varepsilon}\|+\varepsilon^{\frac{4}{3}}\|w_i\nabla_x^i h^{\varepsilon}\|\Big),
\end{align*}
where $\mathcal{E}(t)$ is defined in \eqref{eg-vml}.

Next, choosing $\varepsilon_0>0$ sufficiently small such that
$$C_{\epsilon_1}\exp\left(\frac{-\epsilon_1}{8RT^2_c\sqrt{\varepsilon}}\right)\lesssim \varepsilon_0^2,$$
we have
\begin{align*}
&\frac{\mathrm{d}}{\mathrm{d} t}\mathcal{E}(t)+\mathcal{D}(t)\lesssim \Big[\varepsilon^{\frac{1}{3}}+(1+t)^{-p_0}+\varepsilon^{k}(1+t)^{2k}\Big] \mathcal{E}(t)+\varepsilon^{2k+1}(1+t)^{4k+2},
\end{align*}
where $\mathcal{D}(t)$ is given by \eqref{dn-vml}. Then for $0<\varepsilon\leq \varepsilon_0$, we apply Gronwall's inequality over $t\in [0, \varepsilon^{-1/3}]$ to obtain
\begin{align*}
&\mathcal{E}(t)+\int_0^t\mathcal{D}(s)\, d s\lesssim \mathcal{E}(0)+1.
\end{align*}
This completes the proof of Theorem~\ref{resultVML}.

%%%%%%%%%%%%%%%%%%%%%%%%%%%%%%%%%%%%%%%%%%%%%%%%%%%%%%%%%%%%%%%%%%%%%%%%%%%%%%%%%%

\end{document}